\documentclass[12pt]{amsart}
\usepackage{fullpage}
\usepackage{bm}       
\usepackage{mathtools} 
\usepackage[all]{xy}
\usepackage{booktabs} 
\usepackage{hyperref}
\hypersetup{colorlinks=true,urlcolor=blue,citecolor=blue,linkcolor=blue}

\makeatletter
\let\@wraptoccontribs\wraptoccontribs
\makeatother

\def\ni{\noindent}
\def\dim{\mathop{\rm dim}}
\def\e{\mathop{\varepsilon}}

\usepackage{amssymb,amsthm,amsmath,amstext,amsxtra}
\usepackage{enumerate}
\usepackage{ stmaryrd }
\usepackage{longtable}
\usepackage{enumerate}
\usepackage{comment}
\usepackage{colonequals}
\usepackage{multirow}

\numberwithin{equation}{subsection}

\theoremstyle{plain}
\newtheorem*{theorema*}{Theorem A}
\newtheorem*{thm*}{Theorem}
\newtheorem*{theorem1*}{Theorem 1}
\newtheorem*{theorem2*}{Theorem 2}

\newtheorem{thm}[equation]{Theorem}

\newtheorem{prop}[equation]{Proposition}
\newtheorem{lemma}[equation]{Lemma}
\newtheorem{lem}[equation]{Lemma} 

\newtheorem{cor}[equation]{Corollary}

\newtheorem{conj}[equation]{Conjecture}
\newtheorem*{conj*}{Conjecture}

\theoremstyle{definition}

\newtheorem{defn}[equation]{Definition}
\newtheorem{example}[equation]{Example}

\newtheorem{alg}[equation]{Algorithm}

\theoremstyle{remark}
\newtheorem{remark}[equation]{Remark}
\newtheorem{rmk}[equation]{Remark}

\newenvironment{enumalph}
{\begin{enumerate}}
{\end{enumerate}}

\newenvironment{enumroman}
{\begin{enumerate}}
{\end{enumerate}}

\newenvironment{enumalg}
{\begin{enumerate}}
{\end{enumerate}}

\newcommand{\Z}{\mathbb{Z}}
\newcommand{\Q}{\mathbb{Q}}
\newcommand{\R}{\mathbb{R}}
\newcommand{\F}{\mathbb{F}}
\newcommand{\Fl}{\F_\ell}

\newcommand{\A}{\mathbb{A}}
\newcommand{\C}{\mathbb{C}}

\newcommand{\coho}{\mathrm{H}}
\newcommand{\cocyc}{\mathrm{Z}}
\newcommand{\cobound}{\mathrm{B}}
\newcommand{\redG}{\mathrm{G}}
\newcommand{\Liesp}{\mathbf{\mathsf{sp}}}

\newcommand{\HeckeT}{\mathrm{T}}

\newcommand{\Half}{\mathcal H}

\newcommand{\SiegelH}{\mathcal{H}}

\newcommand{\legen}[2]{\biggl(\displaystyle{\frac{#1}{#2}}\biggr)}

\newcommand{\calP}{\mathcal{P}}

\newcommand{\defi}[1]{\emph{\textsf{#1}}} 	

\newcommand{\al}{\textup{al}}  

\DeclareMathOperator{\ad}{ad}
\DeclareMathOperator{\cond}{cond}

\DeclareMathOperator{\Jac}{Jac}

\newcommand{\M}{\mathrm{M}}

\DeclareMathOperator{\SO}{SO}
\DeclareMathOperator{\Obc}{Obc}

\DeclareMathOperator{\GL}{GL}
\DeclareMathOperator{\diag}{diag}

\DeclareMathOperator{\Lie}{Lie}
\DeclareMathOperator{\im}{img}
\DeclareMathOperator{\img}{\im}

\DeclareMathOperator{\GSp}{GSp}

\DeclareMathOperator{\Sp}{Sp}
\DeclareMathOperator{\SU}{SU}
\DeclareMathOperator{\Gal}{Gal}
\DeclareMathOperator{\Aut}{Aut}

\DeclareMathOperator{\Frob}{Frob}

\DeclareMathOperator{\Ob}{Ob}

\DeclareMathOperator{\ord}{ord}

\DeclareMathOperator{\antidiag}{antidiag}

\DeclareMathOperator{\spin}{spin}
\DeclareMathOperator{\utr}{utr}

\newcommand{\id}{\mathrm{id}}

\newcommand{\indic}{\mathbf{1}}
\newcommand{\rhobar}{\overline{\rho}}

\DeclareMathOperator{\tr}{\mathrm{tr}}
\DeclareMathOperator{\Tr}{\mathrm{Tr}}

\DeclareMathOperator{\End}{\mathrm{End}}

\DeclareMathOperator{\TB}{\mathrm{TB}}
\DeclareMathOperator{\Grit}{\mathrm{Grit}}

\newcommand{\frakl}{\mathfrak{l}}
\newcommand{\frakm}{\mathfrak{m}}

\newcommand{\frakp}{\mathfrak{p}}

\newcommand{\Qbar}{{\Q^\al}}
\newcommand{\Abar}{{A^\al}}
\newcommand{\Fbar}{{F^\al}}

\newcommand{\psmod}[1]{~(\textup{\text{mod}}~{#1})}

\newcommand{\sqtimes}{\boxtimes}

\usepackage{xcolor}
\definecolor{darkred}{HTML}{CC1F1F}
\definecolor{green}{rgb}{.4,.7,.4}
\definecolor{blue}{rgb}{.2,.6,.75}
\definecolor{pastelyellow}{rgb}{0.992157, 0.552941, 0.235294}
\definecolor{pastelorange}{rgb}{0.941176, 0.231373, 0.12549}
\definecolor{pastelred}{rgb}{0.741176, 0., 0.14902}
\definecolor{darkbrown}{rgb}{0.25098, 0., 0.0745098}

\newcommand\mymat[4]{
{\left(
\begin{smallmatrix}#1&#2\\#3&#4\end{smallmatrix}
\right)}}

\def\LP{Laurent--Puiseux}  

\def\cusp{^{\text{cusp}}}
\def\inv{^{-1}}

\def\transpose{^{\mathsf T}}

\DeclareMathOperator{\BP}{Borch}
\DeclareMathOperator{\Coeff}{coeff}
\def\myT{T} 
\def\myA{G}\def\myB{H} 

\hypersetup{pdftitle={On the paramodularity of typical abelian surfaces},
pdfauthor={Armand Brumer, Ariel Pacetti, Cris Poor, Gonzalo Tornar{\'\i}a, John Voight, and David S.\ Yuen}}

\begin{document}

\title[On the paramodularity of typical abelian surfaces]{On the paramodularity of typical abelian surfaces \\ (and reduction of $G$-covariant bilinear forms)}

\author{Armand Brumer}
\address{Department of Mathematics, Fordham University, Bronx, NY 10458}
\email{brumer@fordham.edu}
\urladdr{}

\author{Ariel Pacetti} 
\address{FAMAF-CIEM, Universidad Nacional de
  C\'ordoba. C.P:5000, C\'ordoba, Argentina.}
\email{apacetti@famaf.unc.edu.ar}
\urladdr{}

\author{Cris Poor}
\address{Department of Mathematics, Fordham University, Bronx, NY 10458}
\email{poor@fordham.edu}
\urladdr{}

\author{Gonzalo Tornar{\'\i}a}
\address{Universidad de la Rep\'ublica, Montevideo, Uruguay}
\email{tornaria@cmat.edu.uy}
\urladdr{}

\author{John Voight}
\address{Department of Mathematics, Dartmouth College, 6188 Kemeny Hall, Hanover, NH 03755, USA}
\email{jvoight@gmail.com}
\urladdr{\url{http://www.math.dartmouth.edu/~jvoight/}}

\author[D.~Yuen]{David S.~Yuen \\ (appendix by J.P.~Serre)}
\address{Department of Mathematics, University of Hawaii, Honolulu, HI 96822 USA}
\email{yuen@math.hawaii.edu}

\address{Coll\`{e}ge de France, 3 rue d'Ulm, Paris}
\email{jpserre691\@gmail.com}


\begin{abstract}
Generalizing the method of Faltings--Serre, we rigorously verify that certain abelian surfaces without extra endomorphisms are paramodular.  To compute the required Hecke eigenvalues, we develop a method of specialization of Siegel paramodular forms to modular curves.
In the appendix, Serre proves a result extending his work on the reduction of $G$-invariant bilinear forms modulo primes to the case of $G$-covariant forms.
\end{abstract}

\date{\today}

\maketitle

\setcounter{tocdepth}{1}
\tableofcontents

\section{Introduction}

\subsection{Paramodularity}

The Langlands program predicts deep connections between geometry and automorphic forms, encoded in associated $L$-functions and Galois representations.  The celebrated modularity of elliptic curves $E$ over $\Q$ \cite{Wiles,TW,BCDT} provides an important instance of this program: to the isogeny class of $E$ of conductor $N$, we associate a classical cuspidal newform $f \in S_2(\Gamma_0(N))$ of weight $2$ and level $N$ with rational Hecke eigenvalues such that $L(E,s)=L(f,s)$, and conversely.  In particular, $L(E,s)$ shares the good analytic properties of $L(f,s)$ including analytic continuation and functional equation, and the $\ell$-adic Galois representations of $E$ and of $f$ are equivalent.  More generally, by work of Ribet \cite{Ribet92} and the proof of Serre's conjecture by Khare--Wintenberger \cite{KW1,KW2}, isogeny classes of abelian varieties $A$ of dimension $d$, of $\GL_2$-type over $\Q$, and of conductor $N^d$ are in bijection with Galois orbits of classical cuspidal newforms $f \in S_2(\Gamma_1(N))$, with matching (imprimitive) $L$-functions and $\ell$-adic Galois representations.

Continuing this program, let $A$ be an abelian surface over $\Q$; for instance, we may take $A=\Jac(X)$ the Jacobian of a curve of genus $2$ over $\Q$.  We suppose that $\End(A)=\Z$, i.e., $A$ has minimal endomorphisms defined over $\Q$, and in particular $A$ is \emph{not} of $\GL_2$-type over $\Q$.  For example, if $A$ has prime conductor, then $\End(A)=\Z$ by a theorem of Ribet (see Lemma \ref{lem:squarefree}).  A conjecture of H.~Yoshida \cite{Yoshida1,Yoshida2} compatible with the Langlands program is made precise by a conjecture of Brumer--Kramer \cite[Conjecture 1.1]{BK14}, restricted here for simplicity.

\begin{conj}[Brumer--Kramer] \label{conj:BrumKram00}
To every abelian surface $A$ over $\Q$ of conductor $N$ with $\End(A)=\Z$, there exists a cuspidal Siegel paramodular newform $f$ of degree~$2$, weight $2$, and level $N$ with rational Hecke eigenvalues that is not a Gritsenko lift, such that 
\begin{equation} \label{eqn:LBsLfs}
L(A,s)=L(f,s,\spin). 
\end{equation}
Moreover, $f$ is unique up to (nonzero) scaling and depends only on the isogeny class of $A$; and if $N$ is squarefree, then this association is bijective.
\end{conj}

Conjecture \ref{conj:BrumKram00} is often referred to as the \defi{paramodular conjecture}; in what follows, we say \defi{nonlift} for not a Gritsenko lift.   As pointed out by Frank Calegari, in general it is necessary to include abelian fourfolds with quaternionic multiplication for the converse assertion: for a precise statement for arbitrary $N$ and further discussion, see Brumer--Kramer \cite{BKcorrig}.  

Extensive experimental evidence  \cite{BK14,PY15} supports Conjecture \ref{conj:BrumKram00}.  There is also theoretical evidence for this conjecture when the abelian surface $A$ is potentially of $\GL_2$-type, acquiring extra endomorphisms over a quadratic field: see Johnson-Leung--Roberts \cite{JLR12} for real quadratic fields, Berger--Demb\'el\'e--Pacetti--\c{S}eng\"un \cite{BDPS15} for imaginary quadratic fields, and Demb\'el\'e--Kumar \cite{DemKum} for explicit examples.  For a complete treatment of the many possibilities for the association of modular forms to abelian surfaces with potentially extra endomorphisms, see work of Booker--Sijsling--Sutherland--Voight--Yasaki \cite{BSSVY}.  What remains is the case where $\End(A_{\Qbar})=\Z$, which is to say that $A$ has minimal endomorphisms defined over the algebraic closure $\Qbar$; we say then that $A$ is \defi{typical}.  (We do not say \emph{generic}, since it is not a Zariski open condition on the moduli space.)  

Recently, there has been dramatic progress in modularity lifting theorems for nonlift Siegel modular forms (i.e., forms not \emph{of endoscopic type}): see Pilloni \cite{Pilloni} for $p$-adic overconvergent modularity lifting, as well as recent work by Calegari--Geraghty \cite[\S 1.2]{CG}, Berger--Klosin with Poor--Shurman--Yuen \cite{BKPSY} establishing modularity in the reducible case when certain congruences are provided, and a recent manuscript by Boxer--Calegari--Gee--Pilloni \cite{BCGP} establishing potential modularity over totally real fields.  

\subsection{Main result}

For all \emph{prime} levels $N<277$, the paramodular conjecture is known: there are no paramodular forms of the specified type by work of Poor--Yuen \cite[Theorem 1.2]{PY15}, and correspondingly there are no abelian surfaces by work of Brumer--Kramer \cite[Proposition 1.5]{BK14}.  At level $N=277$, there exists a cuspidal, nonlift Siegel paramodular cusp form, unique up to scalar multiple, by work of Poor--Yuen \cite[Theorem 1.3]{PY15}: this form is given explicitly as a rational function in Gritsenko lifts of ten weight $2$ theta blocks---see \eqref{eqn:fasQL277}.  

Our main result is as follows.

\begin{thm} \label{thm:mainthm277}
Let $X$ be the curve over $\Q$ defined by
\[ y^2+(x^3+x^2+x+1)y=-x^2-x; \]
let $A=\Jac(X)$ be its Jacobian, a typical abelian surface over $\Q$ of conductor $277$.  Let $f$ be the cuspidal, nonlift Siegel paramodular form of genus $2$, weight $2$, and conductor $277$, unique up to scalar multiple.  Then 
\[ L(A,s) = L(f,s,\spin). \]
\end{thm}

Theorem \ref{thm:mainthm277} is not implied by any of the published or announced results on paramodularity, and its announcement in October 2015 makes it the first established typical case of the paramodular conjecture.  More recently, Berger--Klosin with Poor--Shurman--Yuen \cite{BKPSY} recently established the paramodularity of an abelian surface of conductor $731$ using a congruence with a Siegel Saito--Kurokawa lift.  

Returning to the paramodular conjecture, by work of Brumer--Kramer \cite[Theorem 1.2]{BK15} there is a unique isogeny class of abelian surfaces 
(LMFDB label \href{http://www.lmfdb.org/Genus2Curve/Q/277/a/}{\textsf{\textup{277.a}}}) of conductor $277$.  Therefore, the proof of Conjecture \ref{conj:BrumKram00} for $N=277$ is completed by Theorem \ref{thm:mainthm277}.  (More generally, Brumer--Kramer \cite{BK14} also consider odd semistable conductors at most $1000$.)

The theorem implies, and we prove directly, the equality of polynomials $L_p(A,T)=Q_p(f,T)$ for all primes $p$ arising in the Euler product for the corresponding $L$-series.  These equalities are useful in two ways.  On the one hand, the Euler factors $L_p(A,T)$ can be computed much more efficiently than for $Q_p(f,T)$: without modularity, to compute the eigenvalues of a Siegel modular form $f$ is difficult and sensitive to the manner in which $f$ was constructed, whereas computing $L_p(A,T)$ can be done in average polynomial time \cite{Harvey} and also efficiently in practice \cite{HS}.  On the other hand, the $L$-series $L(A,s)$ is endowed with the good analytic properties of $L(f,s,\spin)$: without (potential) modularity, one knows little about $L(A,s)$ beyond convergence in a right half-plane.

By work of Johnson-Leung--Roberts, there are infinitely many quadratic characters $\chi$ such that the twist $f_\chi$ of the paramodular cusp form by $\chi$ is nonzero \cite[Main Theorem]{JLR-IJNT2014}.  By a local calculation \cite[Theorem 3.1]{JLR-IJNT2017}, we have $Q_p(f_\chi,T) = Q_p(f,\chi(p)T)$ and similarly $L_p(A_\chi,T) = L_p(A,\chi(p)T)$ for good primes $p$.
Consequently, we have $L(A_\chi,s)=L(f_\chi,s,\spin)$ for infinitely many characters $\chi$, and in this way we also establish the paramodularity of infinitely many twists.

We also establish paramodularity for two other isogeny classes in this article of conductors $N=353$ and $N=587$, and our method is general enough to establish paramodularity in a wide variety of cases.

\subsection{The method of Faltings--Serre}

We now briefly discuss the method of proof and a few relevant details.  
Let $\Gal_\Q \colonequals \Gal(\Qbar \,|\, \Q)$ be the absolute Galois group of $\Q$.  To establish paramodularity, we associate $2$-adic Galois representations $\rho_A,\rho_f\colon \Gal_\Q \to \GSp_4(\Q_2^{\textup{al}})$ to $A$ and $f$, and then we prove by an extension of the Faltings--Serre method that these Galois representations are equivalent.  The Galois representation for $A$ arises via its Tate module.  By contrast, the construction of the Galois representation for the Siegel paramodular form---for which the archimedean component of the associated automorphic representation is a holomorphic limit of discrete series---is much deeper: see Theorem \ref{thm:galoisrep} for a precise statement, attribution, and further discussion.  

The first step in carrying out the Faltings--Serre method is to prove equivalence modulo $2$, which can be done using information on $\overline{\rho}_f$ obtained by computing $Q_p(f,T)$ modulo $2$ for a few small primes $p$.
For example, $p=3,5$ are enough for $N=277$ (see Lemma~\ref{lem:residuallyagree}) and in this case the mod $2$ residual Galois representations 
\[ \overline{\rho}_A, \overline{\rho}_f \colon \Gal_{\Q} \to \GSp_4(\F_2) \simeq S_6 \]
have common image
$S_5(b)$ up to conjugation.
(There are two nonconjugate subgroups of $S_6$ isomorphic to $S_5$, interchanged by an outer automorphism of $S_6$: see \eqref{table:subgroups}.)  
  
  The second step is to show that the traces of the two representations agree for an effectively computable set of primes $p$.  For example, to finish the proof of Theorem~\ref{thm:mainthm277} in level $N=277$, it suffices to show equality of traces for primes $p \leq 43$.  
  
We also carry out this strategy to prove paramodularity for two other isogeny classes of abelian surfaces.  For $N=353$, we have the isogeny class with LMFDB label \href{http://www.lmfdb.org/Genus2Curve/Q/353/a/}{\textsf{\textup{353.a}}}; we again represent the paramodular form as a rational function in Gritsenko lifts; and the common mod $2$ image is instead the wreath product $S_3 \wr S_2$ of order $72$.  For $N=587$, we have the class with label \href{http://www.lmfdb.org/Genus2Curve/Q/587/a/}{\textsf{\textup{587.a}}}; instead, we represent the form as a Borcherds product; and in this case the mod $2$ image is the full group $S_6$.

\subsection{Contributions and organization}

Our contributions in this article are threefold.  First, we show how to extend the Faltings--Serre method from $\GL_2$ to a general algebraic group when the residual mod $\ell$ representations are absolutely irreducible.  We then discuss making this practical by consideration of core-free subgroups in a general context, and we hope this will be useful in future investigations.  We then make these extensions explicit for $\GSp_4$ and $\ell=2$.  Whereas for $\GL_2$, Serre's original ``quartic method'' considers extensions whose Galois groups are no larger than $S_4$, for $\GSp_4$ we must contemplate large polycyclic extensions of $S_6$-extensions---accordingly, the Galois theory and class field theory required to make the method explicit and to work in practice are much more involved.  It would be much more difficult (perhaps hopeless) to work with $\GL_4$ instead of $\GSp_4$, so our formulation is crucial for practical implementation.

By other known means, the task of calculating the required traces for $\rho_f$ would be extremely difficult.  Our second contribution in this article is to devise and implement a method of \emph{specialization} of the Siegel modular form to a classical modular form, making this calculation a manageable task.  

Our third contribution is to carry out the required computations.  There are nine absolutely irreducible subgroups of $\GSp_4(\F_2)$.  The three examples we present cover each of the three possibilities for the residual image when it is absolutely irreducible and the level is squarefree (see Lemma \ref{lem:s5s6s3s2}).  Our methods work for any abelian surface whose mod $2$ image is absolutely irreducible, as well as situations for paramodular forms of higher weight.  Our implementations are suitable for further investigations along these lines.

The paper is organized as follows.  In section \ref{sec:faltingsserre}, we explain the extension of the method of Faltings--Serre in a general (theoretical) algorithmic context; we continue in section \ref{sec:corefree} by noting a practical extension of this method using some explicit Galois theory.  We then consider abelian surfaces, paramodular forms, and their associated Galois representations tailored to our setting in section \ref{sec:galoisreps}.  Coming to our intended application, we provide in section \ref{sec:GSp4} the group theory and Galois theory needed for the Faltings--Serre method for $\GSp_4(\Z_2)$.  In section \ref{sec:computehecke}, we explain a method to compute Hecke eigenvalues of Siegel paramodular forms using restriction to a modular curve.  Finally, in section \ref{sec:verifypara}, we combine these to complete our task and verify paramodularity.

\subsection{Acknowledgements}

The authors would like to thank several people for helpful conversations: Frank Calegari, Jennifer John\-son-Leung, Kenneth Kra\-mer, Chung Pang Mok, David P.\ Roberts (in particular for Proposition \ref{prop:discbound}), Drew Sutherland, and Eric Urban (in particular for help with showing that the representation is symplectic in Theorem \ref{thm:galoisrep}).  We also thank Fordham University's Academic Computing Environment for the use of its servers.  Thanks also to the anonymous referees for their feedback.  Pacetti was partially supported by PIP 2014-2016 11220130100073 and  Voight was supported by an NSF CAREER Award (DMS-1151047) and a Simons Collaboration Grant (550029).

This large collaborative project was made possible by the generous support of several host institutes, to which we express our thanks: the Institute for Computational and Experimental Research in Mathematics (ICERM), the International Centre for Theoretical Physics (ICTP), and the Hausdorff Institute of Mathematics (HIM).

\section{A general Faltings--Serre method}  \label{sec:faltingsserre}

In this section, from the point of view of general algorithmic theory, we formulate the Faltings--Serre method to show that two $\ell$-adic Galois representations are equivalent, under the hypothesis that the residual representations are absolutely irreducible.  A practical method for the group $\GSp_4(\Z_2)$ is given in section \ref{sec:GSp4}.  For further reading on the Faltings--Serre method, see the 
original criterion given by Serre \cite{Se1} for elliptic curves over $\Q$, an extension for residually reducible representations by Livn\'e \cite[\S 4]{Livne}, the general overview for $\GL_2$ over number fields by Dieulefait--Guerberoff--Pacetti \cite[\S 4]{DGP}, and the description for $\GL_n$ by Sch\"utt \cite[\S 5]{Schutt}.  For an algorithmic approach in the pro-$p$ setting, see Greni\'e \cite{Grenie}.

\subsection{Trace computable representations}

Let $F$ be a number field with ring of integers $\Z_F$.  Let $\Fbar$ be an algebraic closure of $F$; we take all algebraic extensions of $F$ inside $\Fbar$.  Let $\Gal_F \colonequals \Gal(\Fbar\,|\,F)$ be the absolute Galois group of $F$.  Let $S$ be a finite set of places of $F$, let $\Gal_{F,S}$ be the Galois group of the maximal subextension of $\Fbar \supseteq F$ unramified away from $S$.  By a \defi{prime} of $F$ we mean a nonzero prime ideal $\frakp \subset \Z_F$, or equivalently, a finite place of $F$.  

Let $\redG \subseteq \GL_n$ be an embedded algebraic group over $\Q$. 
Let $\ell$ be a prime of good reduction for the inclusion $\redG \subseteq \GL_n$.  A \defi{representation} $\Gal_{F,S} \to \redG(\Z_\ell)$ is a continuous homomorphism.

\begin{defn} \label{defn:rho1rho2}
Let $\rho_1,\rho_2\colon \Gal_{F,S} \to \redG(\Z_\ell)$ be two representations.  We say $\rho_1$ and $\rho_2$ are ($\GL_n$-)\defi{equivalent}, and we write $\rho_1 \simeq \rho_2$, if there exists $g \in \GL_n(\Z_\ell)$ such that 
\[ \rho_1(\sigma) = g\rho_2(\sigma)g^{-1}, \quad \text{for all $\sigma \in \Gal_{F,S}$.} \]  
\end{defn}

\begin{defn}
A representation $\rho \colon \Gal_{F,S} \to \redG(\Z_\ell)$ is \defi{trace computable} there exists a deterministic algorithm to compute $\tr(\Frob_\frakp)$ for $\frakp \not \in S$, where $\Frob_\frakp$ denotes the conjugacy class of the Frobenius automorphism at $\frakp$.
\end{defn}

In particular, if $\rho$ is trace computable then the values $\tr \rho(\Frob_\frakp)$ belong to a computable subring of $\Z_\ell$.  For precise definitions and a thorough survey of the subject of computable rings, see Stoltenberg-Hansen--Tucker \cite{SHT}.  See Cohen \cite{Cohen} for background on algorithmic number theory.

\begin{rmk}
Galois representations arising in arithmetic geometry are often trace computable.  For example, by counting points over finite fields, we may access the trace of Frobenius acting on Galois representations arising from the \'etale cohomology of a nice variety: then the trace takes values in $\Z \subseteq \Z_\ell$ (independent of $\ell$).  Similarly, algorithms to compute modular forms give as output Hecke eigenvalues, which can then be interpreted in terms of the trace of Frobenius on the associated Galois representation.  
\end{rmk}

Looking only at the trace of a representation is justified in certain cases by the following theorem, a cousin to the Brauer--Nesbitt theorem.  For $r \geq 1$, write 
\[ \rho \bmod{\ell^r} \colon \Gal_{F,S} \to \redG(\Z/\ell^r \Z) \]
for the reduction of $\rho$ modulo $\ell^r$,
and as a shorthand write
\[
\overline{\rho}\colon\Gal_{F,S}\to\redG(\F_\ell)
\]
for the residual representation
$\overline{\rho}=\rho\bmod\ell$.
Given two representations $\rho_1,\rho_2:\Gal_{F,S} \to \redG(\Z_\ell)$, we write $\rho_1 \simeq \rho_2 \pmod{\ell^r}$ to mean that $(\rho_1 \bmod \ell^r) \simeq (\rho_2 \bmod \ell^r)$ are equivalent as in Definition \ref{defn:rho1rho2} but over $\Z/\ell^r \Z$; we write $\rho_1 \equiv \rho_2 \pmod{\ell^r}$ to mean that $(\rho_1 \bmod \ell^r) = (\rho_2 \bmod \ell^r)$; and we write $\tr \rho_1 \equiv \tr \rho_2 \pmod{\ell^r}$ if $\tr \rho_1(\sigma) \equiv \tr \rho_2(\sigma) \pmod{\ell^r}$ for all $\sigma \in \Gal_{F,S}$.  Finally, we say that $\overline{\rho}$ is \defi{absolutely irreducible} if the representation $\Gal_{F,S} \to \redG(\F_\ell) \hookrightarrow \GL_n(\F_\ell)$ is absolutely irreducible.

\begin{thm}[Carayol]  \label{thm:Brauer-Nesbit}  
  Let $\rho_1,\rho_2:\Gal_{F,S} \to \redG(\Z_\ell)$ be two representations such that
  $\overline{\rho}_1$ is absolutely irreducible and let $r \geq 1$.  Then
  $\rho_1 \simeq \rho_2  \bmod{\ell^r}$ if and only if
  $\tr \rho_1 \equiv \tr \rho_2 \psmod{\ell^r}$.
\end{thm}

\begin{proof}
See Carayol \cite[Th\'eor\`eme 1]{Carayol}.
\end{proof}

We now state the main result of this section.  We say that a prime $\frakp$ of $F$ is a \defi{witness} to the fact that $\rho_1 \not\simeq \rho_2$ if $\tr \rho_1(\Frob_\frakp) \neq \tr \rho_2(\Frob_\frakp)$.

\begin{thm} \label{thm:state}
There is a deterministic algorithm that takes as input
\begin{equation} \label{eqn:inputdata}
\begin{minipage}{0.8\textwidth}
\begin{center} 
an algebraic group $\redG$ over $\Q$, a number field $F$, \\ 
a finite set $S$ of primes of $F$, a prime $\ell$, \\
and $\rho_1,\rho_2\colon \Gal_{F,S} \to \redG(\Z_\ell)$ trace computable representations \\
 with $\overline{\rho}_1,\overline{\rho}_2$ absolutely irreducible,
\end{center}
\end{minipage}
\end{equation}
and gives as output
\begin{center}
\textup{\texttt{true}} if $\rho_1 \simeq \rho_2$; or \\
\textup{\texttt{false}} and a witness prime $\frakp \not \in S$ if $\rho_1 \not\simeq \rho_2$.
\end{center}
\end{thm}

The algorithm does not operate on the representations $\rho_1,\rho_2$ themselves, only their traces.  The proof of Theorem \ref{thm:state} will occupy us throughout this section.

\subsection{Testing equivalence of residual representations}

We first prove a variant of our theorem for the residual representations.
For a finite extension $K_0 \supseteq F$ of fields
with $[K_0:F]=n$ and
with Galois closure $K$,
we write $\Gal(K_0\,|\, F) \leq S_n$
for the Galois group $\Gal(K\,|\, F)$
as a permutation group on the roots of a minimal polynomial
of a primitive element for $K_0$.

\begin{lem} \label{lem:enumerateallfields}
There exists a deterministic algorithm that takes as input 
\begin{center}
a number field $F$, \\
a finite set $S$ of places of $F$, \\
and a transitive group $G \leq S_n$,
\end{center}
and gives as output 
\begin{center}
all extensions $K_0 \supseteq F$ (up to isomorphism) of degree $n$ \\ 
unramified at all places $v \not\in S$ \\
such that $\Gal(K_0 \,|\, F) \simeq G$ as permutation groups.
\end{center}
Moreover, every Galois extension $K\supseteq F$ unramified outside $S$
such that $\Gal(K\,|\,F)\simeq G$ as groups
appears as the Galois closure of at least one such $K_0 \supseteq F$.
\end{lem}

\begin{proof}
The extensions $K_0$ have degree $n$ and are unramified away from $S$, so they have effectively bounded discriminant by Krasner's lemma.  Therefore, there are finitely many such fields up to isomorphism, by a classical theorem of Hermite.  The enumeration can be accomplished algorithmically by a \emph{Hunter search}: see Cohen \cite[\S 9.3]{Co2}.  The computation and verification of Galois groups can also be accomplished effectively.  

The second statement follows from basic Galois theory.
\end{proof}

 \begin{remark}
 For theoretical purposes, it is enough to consider $G \hookrightarrow S_n$ in
 its regular representation ($n=\#G$), for which the algorithm
 yields Galois extensions $K=K_0\supseteq F$.
 For practical purposes, it is crucial to work
 with small permutation representations.
 \end{remark}

\begin{alg} \label{alg:detresidual}
The following algorithm takes as input the data \eqref{eqn:inputdata} and gives as output 
\begin{center}
\textup{\texttt{true}} if $\overline{\rho}_1 \simeq \overline{\rho}_2$; or \\
\textup{\texttt{false}} and a witness prime $\frakp \not\in S$ if $\overline{\rho}_1 \not\simeq \overline{\rho}_2$.
\end{center}

\begin{enumalg}
\item Using the algorithm of Lemma \textup{\ref{lem:enumerateallfields}}, enumerate all Galois extensions $K \supseteq F$ up to isomorphism that are unramified away from $S$ and such that $\Gal(K \,|\, F)$ is isomorphic to a subgroup of $\redG(\F_\ell)$.
\item For each of these finitely many fields, enumerate all injective group homomorphisms $\theta\colon \Gal(K \,|\, F) \hookrightarrow \redG(\F_\ell)$
up to conjugation by $\GL_n(\F_\ell)$. 
\item Looping over primes $\frakp \not \in S$ of $F$, rule out pairs $(K,\theta)$ such that 
\[ \tr \rho_1(\Frob_\frakp) \not\equiv \tr \theta(\Frob_\frakp) \psmod{\ell} \] 
for some $\frakp$ until only one possibility $(K_1,\theta_1)$ remains.
\item Let $\calP$ be the set of primes used in Step 3.
If
\[ \tr \rho_2(\Frob_\frakp) \equiv \tr \theta_1(\Frob_\frakp) \psmod{\ell} \] 
for all $\frakp\in\calP$, return \texttt{true}; otherwise, return \texttt{false} and a prime $\frakp \in \calP$ such that $\tr \rho_2(\Frob_\frakp) \not\equiv \tr \theta_1(\Frob_\frakp)$.  
\end{enumalg}
\end{alg}

\begin{proof}[Proof of correctness]
Let $K_1$ be the fixed field under $\ker \overline{\rho}_1$; then $K_1$ is unramified away from $S$, and we have an injective homomorphism $\overline{\rho}_1\colon \Gal(K_1\,|\, F) \hookrightarrow \redG(\F_\ell)$.  Thus $(K_1,\overline{\rho}_1)$ is among the finite list of pairs $(K,\theta)$ computed in Step 2.  

Combining Theorem \ref{thm:Brauer-Nesbit} (for $r=1$) and the Chebotarev density theorem, we can effectively determine if $\overline{\rho}_1 \not \simeq \theta$ by finding a prime $\frakp$ such that $\tr \rho_1(\Frob_\frakp) \not\equiv \tr \theta(\Frob_\frakp) \psmod{\ell}$.  So by looping over the primes $\frakp \not \in S$ of $F$ in Step 3, we will eventually rule out all of the finitely many candidates except one $(K_1',\theta_1')$ and, in the style of Sherlock Holmes, we must have $K_1=K'_1$ and $\overline{\rho}_1 \simeq \theta_1$. 

For the same reason,
if $\tr\rho_2(\Frob_\frakp)\equiv\tr\theta_1(\Frob_\frakp) \pmod{\ell}$ for all $\frakp\in\calP$ we must have $\overline{\rho}_2\simeq\theta_1\simeq\overline{\rho}_1$.
Otherwise, we find a witness prime $\frakp\in\calP$.
\end{proof}

\begin{remark}
In practice, we may also use the characteristic polynomial of
$\rhobar_i(\Frob_\frakp)$ when it is computable, since it gives more information about the residual image and thereby
limits the possible subgroups of $\redG(\F_\ell)$ we need to consider in Step 1.
This allows for a smaller list of pairs $(K,\theta)$ and
a smaller list of primes: 
see Lemma~\ref{lem:residuallyagree} for an example.
\end{remark}

\subsection{Faltings--Serre and deformation} \label{sec:FSdefo}

With the residual representations identified, we now explain the key idea of the Faltings--Serre method: we exhibit another representation that measures the failure of two representations to be equivalent.  This construction is quite natural when viewed in the language of deformation theory: see Gouv\^ea \cite[Lecture 4]{Gouvea} for background.  

For the remainder of this section, let $\rho_1,\rho_2 \colon \Gal_{F,S} \to \redG(\Z_\ell)$ be representations such that $\rho_1 \simeq \rho_2 \pmod{\ell^r}$ for some $r \geq 1$.  Conjugating $\rho_2$, we may assume $\rho_1 \equiv \rho_2 \pmod{\ell^r}$, and we write $\overline{\rho} \colonequals \overline{\rho}_1=\overline{\rho}_2$ for the common residual representation modulo $\ell$.  We suppose throughout that $\overline{\rho}$ is absolutely irreducible.

Let $\Lie(\redG) \leq \M_n$ be the Lie algebra of $\redG$ over $\Q$ as a commutative algebraic group.  Attached to $\overline{\rho}$ is the \defi{adjoint residual representation} 
\begin{equation}
\begin{aligned}
    \ad \overline{\rho} \colon \Gal_{F,S} &\to \Aut_{\F_\ell}(\M_n(\F_\ell)) \\
    \sigma&\mapsto \sigma_{\ad}
\end{aligned}
\end{equation}
defined by
$\sigma_{\ad}(a) \colonequals\overline{\rho}(\sigma)a\overline{\rho}(\sigma)^{-1}$ for $a \in \M_n(\F_\ell)$.  The adjoint residual representation $\ad \overline{\rho}$ also restricts to take values in $\Aut_{\F_\ell}(\Lie(\redG)(\F_\ell))$, but we will not need to introduce new notation for this restriction.

Because we consider representations with values in $\redG$ up to equivalence in $\GL_n$, it is natural that our deformations will take values in $\Lie(\redG)$ up to equivalence in $\M_n$.  With this in mind, 
we define the group of \defi{cocycles}
\begin{equation}
\begin{aligned}
&\cocyc^1(F, \ad\rhobar; \Lie(\redG)(\F_\ell)) \colonequals  \\
&\qquad \bigl\{
  (\mu\colon \Gal_{F,S}\to \Lie(\redG)(\F_\ell)) : 
  \mu(\sigma\tau) = \mu(\sigma)
+ \sigma_{\ad}(\mu(\tau))
\text{ for all $\sigma,\tau \in \Gal_{F,S}$}
\bigr\} 
\end{aligned}
\end{equation}
and the subgroup of \defi{coboundaries}
\begin{equation}
\begin{aligned}
&\cobound^1(F, \ad\rhobar; \M_n(\F_\ell)) \colonequals \\
&\qquad \bigl\{
    \mu \in \cocyc^1(F, \ad\rhobar;\Lie(\redG)(\F_\ell))
:\text{there exists $a \in \M_n(\F_\ell)$ such that} \\
  &\qquad\qquad\qquad\qquad\qquad\qquad\qquad\quad \text{$\mu(\sigma) = a - \sigma_{\ad}(a)$ for all $\sigma \in \Gal_{F,S}$}\bigr\}.
\end{aligned}
\end{equation}

From the exact sequence
\begin{equation} 
1 \to 1+\ell^r \Lie(\redG)(\F_\ell) \to \redG(\Z/\ell^{r+1} \Z) \to \redG(\Z/\ell^r\Z) \to 1,
\end{equation}
we conclude that for all $\sigma\in\Gal_{F,S}$
there exists $\mu(\sigma) \in \Lie(\redG)(\F_\ell)$ such that
\begin{equation} \label{eqn:rho1mu}
\rho_1(\sigma) \equiv (1+\ell^r \mu(\sigma))\rho_2(\sigma) \pmod{\ell^{r+1}}\,.
\end{equation}

\begin{lem} \label{lem:iscocyc}
The following statements hold.
\begin{enumalph}
\item    The map $\sigma \mapsto \mu(\sigma)$ defined by \eqref{eqn:rho1mu} is a cocycle $\mu \in \cocyc^1(F, \ad\rhobar;\Lie(\redG)(\F_\ell))$.
\item We have
$\rho_1 \simeq \rho_2 \pmod{\ell^{r+1}}$
if and only if $\mu\in\cobound^1(F, \ad\rhobar;\M_n(\F_\ell))$.
\end{enumalph}
\end{lem}
  
\begin{proof}
We verify the cocycle condition as follows: 
\begin{align*} 
\rho_1(\sigma\tau) = \rho_1(\sigma)\rho_1(\tau) &\equiv (1+\ell^r \mu(\sigma))\rho_2(\sigma) (1+\ell^r \mu(\tau))\rho_2(\tau) \\
&\equiv (1+\ell^r(\mu(\sigma)+\rho_2(\sigma)\mu(\tau)\rho_2(\sigma)^{-1}))\rho_2(\sigma)\rho_2(\tau) \\
&\equiv (1+\ell^r\mu(\sigma\tau))\rho_2(\sigma\tau) \pmod{\ell^{r+1}}
\end{align*}
so $\mu(\sigma\tau) = \mu(\sigma)+\sigma_{\ad}(\mu(\tau))$ as claimed.
For the second statement, by definition $\rho_1\simeq\rho_2 \pmod{\ell^{r+1}}$ if and only if there exists $a_r \in \GL_n(\Z/\ell^{r+1}\Z)$ such that for all $\sigma \in \Gal_{F,S}$ we have
\begin{equation}  \label{eqn:rho1rho2}
\rho_1(\sigma)\equiv a_r\rho_2(\sigma)a_r^{-1} \pmod{\ell^{r+1}}\,.
\end{equation}
Since $\rho_1(\sigma) \equiv \rho_2(\sigma) \pmod{\ell^r}$, the image of $a_r$ in $\GL_n(\Z/\ell^r\Z)$ centralizes the image of $\rho \pmod{\ell^r}$.  Since the image is irreducible, by Schur's lemma we have $a_r \bmod{\ell^r}$ is scalar, so without loss of generality we may suppose $a_r \equiv 1 \pmod{\ell^r}$, so that $a_r = 1+\ell^r a$ for some $a \in \M_n(\F_\ell)$.  Expanding \eqref{eqn:rho1rho2} then yields 
\begin{align*} 
\rho_1(\sigma)&\equiv(1+\ell^r a)\,\rho_2(\sigma)\,(1+\ell^r a)^{-1}
\equiv (1+\ell^r a)\,\rho_2(\sigma)\,(1-\ell^r a) \\
&\equiv (1+\ell^r a - \ell^r \rho_2(\sigma)a\rho_2(\sigma)^{-1})\rho_2(\sigma) \\
&\equiv (1+\ell^r(a-\sigma_{\ad}(a)))\rho_2(\sigma) \pmod{\ell^{r+1}}
\end{align*}
so $\mu(\sigma)=a-\sigma_{\ad}(a)$ by definition \eqref{eqn:rho1mu}.
\end{proof}

Our task now turns to finding an effective way to detect when $\mu$ is a coboundary.  For this purpose, we work with extensions of our representations using explicit parabolic groups.  The adjoint action of $\GL_n$ on $\M_n$ gives an exact sequence
\begin{equation} \label{eqn:LieGG}
0 \to \M_n \to \M_n \rtimes \GL_n \to \GL_n \to 1
\end{equation}
which extends to a linear representation via the \emph{parabolic subgroup}, as follows.  We embed 
\begin{equation} \label{eqn:GL2n}
\begin{aligned} 
\M_n \rtimes \GL_n &\hookrightarrow \GL_{2n} \\
(a,g) &\mapsto \begin{pmatrix} 1 & a \\ 0 & 1 \end{pmatrix} \begin{pmatrix} g & 0 \\ 0 & g \end{pmatrix} = \begin{pmatrix} g & ag \\ 0 & g \end{pmatrix} 
\end{aligned}
\end{equation}
(on points, realizing $\M_n \rtimes \GL_n$ as an algebraic matrix group).  The embedding \eqref{eqn:GL2n} is compatible with the exact sequence \eqref{eqn:LieGG}: the natural projection map 
\begin{equation}
\pi\colon \M_n \rtimes \GL_n \to \GL_n
\end{equation} 
corresponds to the projection onto the top left entry, it is split by the diagonal embedding $\GL_n \hookrightarrow \GL_{2n}$, and it has kernel isomorphic to $\M_n$ in the upper-right entry.  We will identify $\M_n \rtimes \GL_n$ and its subgroups with their image in $\GL_{2n}$.

Let $\utr \colon (\M_{n}\rtimes\GL_n)(\F_\ell) \to \F_\ell$ denote the trace of the upper right $n \times n$-block.

\begin{lem} \label{lem:utrconj}
The map $\utr$ is well-defined on conjugacy classes in $(\M_n\rtimes\GL_n)(\F_\ell)$.
\end{lem}

\begin{proof}
For all $g,h \in \GL_n(\F_\ell)$ and $a,b \in \M_n(\F_\ell)$ we have
\begin{equation}
 \begin{pmatrix} h & bh \\ 0 & h \end{pmatrix}\begin{pmatrix} g & ag \\ 0 & g \end{pmatrix}\begin{pmatrix} h^{-1} & -h^{-1}b \\ 0 & h^{-1} \end{pmatrix} = \begin{pmatrix} hgh^{-1} & hagh^{-1}+bhgh^{-1}-hgh^{-1}b \\ 0 & hgh^{-1} \end{pmatrix} 
 \end{equation}
 so the upper trace is $\tr(hagh^{-1}+bhgh^{-1}-hgh^{-1}b) = \tr(ag)$.
 \end{proof}

For $\mu\in\cocyc^1(F,\ad\rhobar;\Lie(\redG)(\F_\ell))$ we define
\begin{equation}  \label{eqn:varphiFSglie}
\begin{aligned}
\varphi_{\mu} \colon \Gal_{F,S} &\to (\Lie(\redG) \rtimes \redG)(\F_\ell) \leq \GL_{2n}(\F_\ell) \\
\sigma &\mapsto (\mu(\sigma),\overline{\rho}(\sigma)) = \begin{pmatrix} \overline{\rho}(\sigma) & \mu(\sigma)\overline{\rho}(\sigma) \\
0 & \overline{\rho}(\sigma) \end{pmatrix} \,.
\end{aligned}
\end{equation}

\begin{prop} \label{prop:muandtrace}
Let $\mu\in\cocyc^1(F,\ad\rhobar;\Lie(\redG)(\F_\ell))$.  Then the following statements hold.
\begin{enumalph}
\item  The map $\varphi_\mu$ defined by \eqref{eqn:varphiFSglie}
is a group homomorphism, and $\pi\circ\varphi_{\mu}=\overline{\rho}$.
\item We have $\mu\in\cobound^1(F, \ad\rhobar;\M_n(\F_\ell))$ if and only if $\varphi_\mu$ is conjugate to $\varphi_0=\mymat{\rhobar}00{\rhobar}$ by an element of $\M_n(\Fl)\leq (\M_n \rtimes \GL_n)(\F_\ell)$.  
\begin{equation}  \label{eqn:uppertrace}
\utr \varphi_\mu(\sigma) = \tr\bigl(\mu(\sigma)\overline{\rho}(\sigma)\bigr) 
\equiv \frac{\tr \rho_1(\sigma) - \tr \rho_2(\sigma)}{\ell^r} \pmod{\ell}. 
\end{equation}
\end{enumalph}
\end{prop}

\begin{proof}
For (a), the cocycle condition implies that $\varphi_\mu$ is a group homomorphism: the upper right entry of $\varphi_\mu(\sigma\tau)$ is
\[ \mu(\sigma\tau)\overline{\rho}(\sigma\tau)= (\mu(\sigma)+\overline{\rho}(\sigma)\mu(\tau)\overline{\rho}(\sigma)^{-1})\overline{\rho}(\sigma)\overline{\rho}(\tau) = \mu(\sigma)\overline{\rho}(\sigma)\overline{\rho}(\tau) + \overline{\rho}(\sigma)\mu(\tau)\overline{\rho}(\tau) \]
which is equal to the upper right entry of $\varphi_\mu(\sigma)\varphi_\mu(\tau)$ obtained by matrix multiplication.  

For (b), the calculation
\begin{equation} \label{eqn:pmatrixaconj}
\begin{pmatrix} 1 & a \\ 0 & 1 \end{pmatrix}\begin{pmatrix} \overline{\rho}(\sigma) & 0 \\ 0 & \overline{\rho}(\sigma) \end{pmatrix} \begin{pmatrix} 1 & -a \\ 0 & 1 \end{pmatrix}
= \begin{pmatrix} \overline{\rho}(\sigma) & a\overline{\rho}(\sigma)-\overline{\rho}(\sigma)a \\ 0 & \overline{\rho}(\sigma) \end{pmatrix} 
\end{equation}
shows that $\varphi_\mu=a\varphi_0 a^{-1}$ for $a \in \M_n(\F_\ell)$ if and only if $\mu(\sigma)\overline{\rho}(\sigma) = a\overline{\rho}(\sigma)-\overline{\rho}(\sigma)a$ for all $\sigma \in \Gal_{F,S}$.  Multiplying on the right by $\overline{\rho}(\sigma)^{-1}$, we see this is equivalent to $\mu(\sigma) = 
a-\sigma_{\ad}(a)$ for all $\sigma \in \Gal_{F,S}$.

Finally, (c) follows directly from \eqref{eqn:rho1mu}.
\end{proof}

\begin{defn} \label{defn:extendingrho}
Let $K$ be the fixed field under $\overline{\rho}$.  We say a pair $(L,\varphi)$ \defi{extends} $(K,\overline{\rho})$ if
\[ \varphi\colon \Gal_{F,S} \rightarrow (\Lie(\redG) \rtimes \redG)(\F_\ell) \leq \GL_{2n}(\F_\ell) \] 
is a representation with fixed field $L$ such that
$\pi\circ\varphi = \overline{\rho}$.
\end{defn}

If $(L,\varphi)$ extends $(K,\overline{\rho})$, then
$L\supseteq K$ is an $\ell$-elementary abelian extension unramified outside $S$, since
$\varphi$ induces an injective group homomorphism
$\Gal(L\,|\,K)\hookrightarrow\Lie(\redG)(\F_\ell)$.

\begin{defn} \label{defn:obstructing}
A pair $(L,\varphi)$ extending $(K,\overline{\rho})$ is \defi{obstructing} if $\utr \varphi \not \equiv 0 \psmod{\ell}$, and we call the group homomorphism $\varphi$ an \defi{obstructing extension} of $\overline{\rho}$.  An element $\sigma \in \Gal(L \,|\, F)$ such that $\utr \varphi(\sigma) \not \equiv 0 \psmod{\ell}$ is called \defi{obstructing} for $\varphi$. 
\end{defn}

We note the following corollary of Proposition \ref{prop:muandtrace}.

\begin{cor} \label{cor:yupmuFS}
Let $\mu$ be defined by \eqref{eqn:rho1mu} and
$\varphi_\mu$ by \eqref{eqn:varphiFSglie}.
Then $\varphi_\mu$ extends $\overline{\rho}$, and $\varphi_\mu$ is obstructing if and only if $\mu\not\in\cobound^1(F,\ad\rhobar;\M_n(\F_\ell))$.
\end{cor}

\begin{proof}
The map $\varphi_\mu$ extends $\overline{\rho}$ by Proposition~\ref{prop:muandtrace}(a).  We prove the contrapositive of the second statement: $\mu\in\cobound^1(F,\ad\rhobar;\M_n(\F_\ell))$ if and only if $\utr\varphi_\mu\equiv 0\pmod{\ell}$.  The implication ${(\Rightarrow)}$ is immediate from Proposition~\ref{prop:muandtrace}(b) and the invariance of $\utr$ by conjugation (Lemma \ref{lem:utrconj}).  For $(\Leftarrow)$, if $\utr\varphi_\mu\equiv 0\pmod{\ell}$ then $\tr\rho_1\equiv\tr\rho_2\pmod{\ell^{r+1}}$ by Proposition \ref{prop:muandtrace}(c). Now Theorem~\ref{thm:Brauer-Nesbit} implies $\rho_1\simeq\rho_2\pmod{\ell^{r+1}}$,
hence $\mu\in\cobound^1(F,\ad\rhobar;\M_n(\F_\ell))$ by Lemma~\ref{lem:iscocyc}(b).
\end{proof}

Before we conclude this section, we note the following important improvement.
Let $\Lie^0(\redG)\leq\Lie(\redG)$ be the subgroup of trace zero matrices,
and note that $\Lie^0(\redG)(\Fl)$ is invariant by the adjoint residual representation.

\begin{lemma} \label{lem:detmatch}
If $\det \rho_1=\det \rho_2$, then $\mu$ takes values in $\Lie^0(\redG)(\F_\ell)$.
\end{lemma}
\begin{proof}
By \eqref{eqn:rho1mu}, we have $1=\det(\rho_1\rho_2^{-1})=\det(1+\ell^r \mu) \equiv 1 + \ell^r \tr \mu \pmod{\ell^{2r}}$ so accordingly $\tr \mu(\sigma) \equiv 0 \psmod{\ell}$ and $\mu(\sigma)\in\Lie^0(\redG)(\F_\ell)$ for all $\sigma\in\Gal_{F,S}$.
\end{proof}

In view of Lemma \ref{lem:detmatch}, we note that Proposition \ref{prop:muandtrace} and Corollary \ref{cor:yupmuFS} hold when replacing $\Lie(\redG)$ by $\Lie^0(\redG)$.  

\subsection{Testing equivalence of representations}

We now use Corollary~\ref{cor:yupmuFS} to prove Theorem \ref{thm:state}.  

\begin{alg} \label{alg:detrep}
The following algorithm takes as input the data \eqref{eqn:inputdata} and gives as output 
\begin{center}
\textup{\texttt{true}} if $\rho_1 \simeq \rho_2$; or \\
\textup{\texttt{false}} and a witness prime $\frakp$ if $\rho_1 \not\simeq \rho_2$.
\end{center}

\begin{enumalg}
\item Apply Algorithm \ref{alg:detresidual}; if $\overline{\rho}_1 \not\simeq \overline{\rho}_2$, return \texttt{false} and the witness prime $\frakp$.  Otherwise, let $K$ be the fixed field under the common residual representation $\overline{\rho}$.
\item Using the algorithm of Lemma \ref{lem:enumerateallfields}, enumerate all $\ell$-elementary abelian extensions $L \supseteq K$ unramified away from $S$ and such that $\Gal(L \,|\, F)$ is isomorphic to a subgroup of $(\Lie(\redG) \rtimes \redG)(\F_\ell)$.  
\item For each of these finitely many fields $L$, by enumeration of injective group homomorphisms $\Gal(L\,|\,F)\hookrightarrow(\Lie(\redG)\rtimes\redG)(\F_\ell)$, find all obstructing pairs $(L,\varphi)$ extending $(K,\overline{\rho})$ up to conjugation by $(\M_n\rtimes\GL_n)(\F_\ell)$.
\item For each such pair $(L,\varphi)$, find a prime $\frakp \not\in S$ such that $\utr \varphi(\Frob_\frakp) \not \equiv 0 \psmod{\ell}$.
\item Check if $\tr \rho_1(\Frob_\frakp) = \tr \rho_2(\Frob_\frakp)$ for the primes in Step 4.  If equality holds for all primes, return \texttt{true}; if equality fails for $\frakp$, return \texttt{false} and the prime $\frakp$.
\end{enumalg}
\end{alg}

\begin{remark}
In Step 2, we may instead use algorithmic class field theory (and we will do so in practice).  Moreover, if we know that $\det\rho_1=\det\rho_2$, then we can replace $\Lie(\redG)$ by $\Lie^0(G)$ by Lemma \ref{lem:detmatch}.
\end{remark}

\begin{proof}[Proof of correctness]
By the Chebotarev density theorem, in Step 4 we will eventually find a prime $\frakp \not\in S$, since $\utr$ is well-defined on conjugacy classes by Lemma \ref{lem:utrconj}.  In the final step, if equality does not hold for some prime $\frakp$, we have found a witness, and we correctly return \texttt{false}.  

Otherwise, we return \texttt{true} and we claim that $\rho_1 \simeq \rho_2$ so the output is correct.  Indeed, assume for purposes of contradiction that $\rho_1 \not\simeq \rho_2$.  Then there exists $r\geq 1$ such that $\rho_1\simeq\rho_2\pmod{\ell^r}$ but $\rho_1\not\simeq\rho_2\pmod{\ell^{r+1}}$. We can assume as before that
$\rho_1\equiv\rho_2\pmod{\ell^r}$.
We define $\mu$ by \eqref{eqn:rho1mu} and $\varphi_\mu$ by \eqref{eqn:varphiFSglie}. Let $L_\mu$ be the fixed field of $\varphi_\mu$. By Lemma~\ref{lem:iscocyc} we have $\mu\not\in\cobound^1(F,\Lie(\redG)(\F_\ell);\M_n(\F_\ell))$, hence by Corollary~\ref{cor:yupmuFS} $\varphi_\mu$ extends $\overline{\rho}$ and is obstructing. It follows that the pair $(L_\mu,\varphi_\mu)$ is, up to conjugation by $(\M_n\rtimes\GL_n)(\F_\ell)$, among the pairs computed in Step~3. In particular there is a prime $\frakp$ in Step 4 such that $\utr\varphi_\mu(\Frob_\frakp) \not\equiv 0 \psmod{\ell}$.  But then by \eqref{eqn:uppertrace} we would have
$\tr \rho_1(\Frob_\frakp) \neq \tr \rho_2(\Frob_\frakp)$, contradicting the verification carried out in Step 5.
\end{proof}

The correctness of Algorithm \ref{alg:detrep} then proves Theorem \ref{thm:state}.  

\begin{remark}
In the case $\redG=\GSp_{2g}$, using an effective version of the Chebotarev density theorem, Achter \cite[Lemma 1.2]{Achter} has given an effective upper bound in terms of the conductor and genus to detect when two abelian surfaces are isogenous.  This upper bound is of theoretical interest, but much too large to be useful in practice.  In a similar way, following the above strategy one could give theoretical (but practically useless) upper bounds to detect when two Galois representations are equivalent.  
\end{remark}

\section{Core-free subextensions} \label{sec:corefree}

The matrix groups arising in the previous section are much too large to work with in practice.  In this section, we find comparatively small extensions whose Galois closure give rise to the desired representations.

\subsection{Core-free subgroups}

We begin with a condition that arises naturally in group theory and Galois theory.  

\begin{defn}
Let $G$ be a finite group.  A subgroup $H \leq G$ is \defi{core-free} if $G$ acts faithfully on the cosets $G/H$.  
\end{defn}  

Equivalently, $H \leq G$ is core-free if and only if $\bigcap_{g \in G} gHg^{-1} = \{1\}$.  For example, the subgroup $\{1\}$ is core-free.

\begin{defn}
Let $K \supseteq F$ be a finite Galois extension of fields with $G=\Gal(K \,|\, F)$.  A subextension $K \supseteq K_0 \supseteq F$ is \defi{core-free} if $\Gal(K \,|\, K_0) \leq G$ is a core-free subgroup.
\end{defn}

\begin{lem} \label{lem:iscorefree}
The subextension $K \supseteq K_0 \supseteq F$ is core-free if and only if $K$ is the Galois closure of $K_0$ over $F$.
\end{lem}

\begin{proof}
Immediate.
\end{proof}

If $K \supseteq K_0 \supseteq F$ is a core-free subextension of $K \supseteq F$ with $K_0=F(\alpha)$, then by definition the action of $\Gal(K \,|\, F)$ on the conjugates of $\alpha$ defines a faithful permutation representation,
equivalent to its action on the left cosets of $\Gal(K\,|\,K_0)$.

We slightly augment the notion of core-free subextension for two-step extensions of fields, as follows.

\begin{defn}
Let
\begin{equation} \label{eqn:VEGseq}
1 \to V \to E \xrightarrow{\pi} G \to 1 
\end{equation}
be an exact sequence of finite groups.  A core-free subgroup $D \leq E$ is \defi{exact} (\defi{relative to} \eqref{eqn:VEGseq}) if $\pi(D)$ is a core-free subgroup of $G$.
\end{defn}

If $D \leq E$ is an exact core-free subgroup we let $H \colonequals \pi(D)$ and $W \colonequals V \cap D = \ker \pi|_D$, so there is an exact subsequence
\begin{equation} \label{eqn:WDHres}
1 \to W \to D \xrightarrow{\pi} H \to 1 
\end{equation}
with both $D \leq E$ and $H \leq G$ core-free.  (We do not assume that $W \leq V$ is core-free.)

Now let $L \supseteq K \supseteq F$ be a two-step Galois extension with $V \colonequals \Gal(L \,|\, K)$, $E \colonequals \Gal(L\,|\,F)$, $G \colonequals \Gal(K\,|\, F)$ and $\pi:E\to G$ the restriction, so we have an exact sequence as in \eqref{eqn:VEGseq}.

\begin{defn} \label{def:corefreesubexte}
We say $L_0 \supseteq K_0 \supseteq F$ is an \defi{exact core-free subextension} of $L \supseteq K \supseteq F$ if 
$L_0=L^{D}$ and $K_0=K^{\pi(D)}$ where $D \leq E$ is an exact core-free subgroup.
\end{defn}

Let $L_0 \supseteq K_0 \supseteq F$ be an exact core-free subextension of $L \supseteq K \supseteq F$,
so that $\Gal(L\,|\,L_0)=D$.
As above we let $H\colonequals\pi(D)=\Gal(K\,|\,K_0)$
and $W\colonequals V\cap D=\Gal(L\,|\,KL_0)$.
By \eqref{eqn:WDHres} we have
$H \simeq D/W = \Gal(KL_0\,|\,L_0)$,
and we have the following field diagram:
\begin{equation} \label{eqn:fielddiagram}
\begin{minipage}{\textwidth}
\xymatrix{
& L \ar@{-}[ddl]_{D} \ar@{-}[dr]_{W} \ar@{-}@/^4pc/[ddr]^{V} \\
& & KL_0 \\
L_0 \ar@{-}[d] \ar@{-}[urr]^{H} & & K \ar@{-}[u] \ar@{-}[ddl]^{G} \\
K_0 \ar@{-}[dr] \ar@{-}[urr]^{H} \\
& F
} 
\end{minipage}
\end{equation}
By Lemma \ref{lem:iscorefree}, $L$ is the Galois closure of $L_0$ over $F$, and $K$ is the Galois closure of $K_0$ over $F$.  We read the diagram \eqref{eqn:fielddiagram} as giving us a way to reduce the Galois theory of the extension $L \supseteq K \supseteq F$ to $L_0 \supseteq K_0 \supseteq F$: the larger we can make $D$, the smaller the extension $L_0 \supseteq K_0 \supseteq F$, and the better for working explicitly with the corresponding Galois groups.

\subsection{Application to Faltings--Serre} \label{sec:appfaltserre}

We now specialize the preceding discussion to our case of interest; although working with core-free extensions does not improve the theoretical understanding, it is a crucial simplification in practice.  

In Steps 2--3 of Algorithm \ref{alg:detrep}, we are asked to enumerate obstructing pairs $(L,\varphi)$ extending $(K,\overline{\rho})$, with $\varphi\colon \Gal(L\,|\,F) \hookrightarrow (\Lie(\redG) \rtimes\,\redG)(\F_\ell)$.  

Let $G \colonequals \img \overline{\rho} \leq \redG(\F_\ell)$.  Given $(L,\varphi)$, the image of $\varphi$ is a subgroup $E \leq \Lie(\redG)(\F_\ell) \rtimes G$ with $\pi(E)=G$; letting $V \colonequals \Lie(\redG)(\F_\ell) \cap E$ we have an exact sequence
\begin{equation} \label{eqn:1VEG1}
1 \to V \to E \xrightarrow{\pi} G \to 1 
\end{equation}
arising from \eqref{eqn:LieGG}.

So we enumerate the subgroups $E \leq \Lie(\redG)(\F_\ell) \rtimes G$ with $\pi(E)=G$, up to conjugation by $\M_n(\F_\ell) \rtimes G$.
The enumeration of these subgroups depends only on $G$, so it may be done as a precomputation step, independent of the representations. 

For each such $E$, let $D$ be an exact core-free subgroup relative to \eqref{eqn:1VEG1}. We let $L_0=L^D$ and $K_0=K^{\pi(D)}$, hence
$L_0\supseteq K_0\supseteq F$ is an exact core-free subextension
of $L\supseteq K\supseteq F$ and we have the field diagram~\eqref{eqn:fielddiagram} where $H=\pi(D)$ and $W=V\cap D$ as before.
Since $V$ is abelian, $KL_0 \supseteq K$ is Galois and hence $L_0 \supseteq K_0$ is also Galois, with common abelian Galois group $\Gal(L_0\,|\,K_0)\simeq\Gal(KL_0 \,|\, K) \simeq V/W$.  So better than a Hunter search as in Lemma \ref{lem:enumerateallfields}, we can use algorithmic class field theory (see Cohen \cite[Chapter 4]{Co2}) to enumerate the possible fields $L_0 \supseteq K_0$.  

Accordingly, we modify Steps 2--3 of Algorithm \ref{alg:detrep} then as follows.

\begin{enumerate}
\item[2${}^{\prime}$.]
Enumerate the subgroups $E \leq \Lie(\redG)(\F_\ell) \rtimes G$ with $\pi(E)=G$,
up to conjugation by $\M_n(\F_\ell) \rtimes G$,
such that $\utr(E) \not\equiv 0 \psmod{\ell}$.
For each such subgroup $E$, perform the following steps.
\begin{enumerate}
\item[a.] Compute a set of representatives $\xi$ of (outer) automorphisms of $E$ such that $\xi$ acts by an inner automorphism on $G$, modulo inner automorphisms by elements of $\M_n(\F_\ell) \rtimes G$.  
\item[b.] Find an exact core-free subgroup $D \leq E$ and let $W,H$ be as in \eqref{eqn:WDHres}.
\item[c.] Let $K_0=K^H$ and use algorithmic class field theory to enumerate all possible extensions $L_0 \supseteq K_0$ unramified away from $S$ such that $\Gal(L_0\,|\,K_0) \simeq V/W$.
\end{enumerate}
\item[3${}^{\prime}$.] For each extension $L_0$ from Step 2${}^{\prime}$c and for each $E$, perform the following steps.
\begin{enumerate}
\item[a.] Compute an isomorphism of groups $\varphi_0\colon \Gal(L\,|\,F) \xrightarrow{\sim} E$ extending $\overline{\rho}$; if no such isomorphism exists, proceed to the next group $E$.
\item[b.] Looping over $\xi$ computed in Step 2${}^{\prime}$a, let $\varphi \colonequals \xi \circ \varphi_0$, and record the pair $(L,\varphi)$.
\end{enumerate}
\end{enumerate}

\begin{proof}[Proof of equivalence with Steps \textup{2}--\textup{3}]
We show that these steps enumerate all obstructing pairs $(L,\varphi)$ up to equivalence.  

Let $L$ be an obstructing extension.  For an obstructing extension $\varphi$ of $\overline{\rho}$, the image $E=\img \varphi$ arises up to conjugation in the list computed in Step 2${}^{\prime}$; such conjugation gives an equivalent representation.  So we may restrict our attention to the set $\Phi$ of obstructing extensions $\varphi$ whose image is \emph{equal} to $E$.

With respect to the core-free subgroup $D$, the field $L$ arises as the Galois closure of the field $L_0=L^{D}$, and so $L_0$ will appear in the list computed in Step 2${}^\prime$c.  An exact core-free subgroup always exists as we can always take $D$ the trivial group. 

In Step 3${}^{\prime}$a, we compute one obstructing extension $\varphi_0 \in \Phi$.  Any other obstructing extension $\varphi \in \Phi$ is of the form $\varphi=\xi \circ \varphi_0$ where $\xi$ is an automorphism of $E$ that induces an inner automorphism on $G$; when $\xi$ arises from conjugation by an element of $\Lie(\redG)(\F_\ell) \rtimes G$, we obtain a representation equivalent to $\varphi_0$, so the representatives $\xi$ computed in Step 2${}^{\prime}$a cover all possible extensions $\varphi$ up to equivalence.
\end{proof}

We now explain in a bit more detail Steps 2${}^{\prime}$a and 3${}^{\prime}$a---in these steps, we need to understand how $\Gal(L\,|\,F)$ restricts to $\Gal(K\,|\,F)$ via its permutation representation.  The simplest thing to do is just to ignore the conditions on $\xi$, i.e., in Step 2${}^{\prime}$a allow all outer automorphisms and in Step 3${}^{\prime}$a take any isomorphism of groups: \emph{a fortiori}, we will still encounter every one satisfying the extra constraint.  To nail it down precisely, we compute the group $\Aut(L_0\,|\,F)$ of $F$-automorphisms of the field $L_0$, for each automorphism $\tau$ of order $2$ compute the fixed field, until we find a field isomorphic to $K_0$: then $\Gal(K\,|\,F)$ is the stabilizer of $\{\beta,\tau(\beta)\}$, and so we can look up the indices of these roots in the permutation representation of $\Gal(L\,|\,F)$.  

In the above, we may also use $\Lie^0(\redG)$ in place of $\Lie(\redG)$ if we are also given $\det\rho_1=\det \rho_2$, by the discussion at the end of section \ref{sec:FSdefo}.

\subsection{Computing conjugacy classes, in stages} \label{sec:instages}

We now discuss Step 4 of Algorithm \ref{alg:detrep}, where we are given $(L,\varphi)$ and we are asked to find a witness prime.  In theory, to accomplish this task we compute the conjugacy class of $\Frob_\frakp$ in $\Gal(L\,|\,K)$ using an algorithm of Dokchitser--Dokchitser \cite{DD-frob} and then calculate $\utr \varphi(\sigma)$ for any $\sigma$ in this conjugacy class.  

In practice, because of the enormity of the computation, we may not want to spend time computing the conjugacy class if we can get away with less.  In particular, we would like to minimize the amount of work done per field.  So we now describe in stages ways to find obstructing primes; each stage gives correct output, but in refining the previous stage we may be able to find smaller primes.  Each of these stages involves a precomputation step that only depends of the group-theoretic data.

In Step 2${}^{\prime}$ above, we enumerate subgroups $E$ and identify an exact core-free subgroup $D$.  We identify $E$ with the permutation representation on the cosets $E/D$.  

In Step 3${}^{\prime}$ above, we see the extension $L \supseteq K \supseteq F$ via a core-free extension $L_0 \supseteq K_0 \supseteq F$, and these fields are encoded by minimal polynomials of primitive elements.  We may compute $\Gal(L\,|\,F)$ as a permutation group with respect to some numbering of the roots, and then insist that the isomorphism $\varphi_0\colon \Gal(L\,|\,F) \xrightarrow{\sim} E$ computed in Step 3${}^{\prime}$a is an isomorphism of permutation representations.  

For $\frakp \not\in S$, for the conjugacy class $\Frob_\frakp$, the cycle type $c(\Frob_\frakp,L_0)$ can be computed very quickly by factoring the minimal polynomial of $L_0$ modulo a power $\frakp^k$ where it is separable (often but not always $k=1$ suffices).  This cycle type may not uniquely identify the conjugacy class, but we can try to find a cycle type which is \emph{guaranteed} to be obstructing as follows.  

\begin{enumerate}
\item[4${}^{\prime}$.]
Perform the following steps.
\begin{enumerate}
\item[a.] For each group $E$ computed in Step 2${}^\prime$ with core-free subgroup $D$, identify $E$ with the permutation representation on the cosets $E/D$.  For each $\xi$ computed in Step 2${}^\prime$a for $E$, compute the set of cycle types 
\begin{align*} 
&\Obc(E,\xi) \colonequals \{c(\xi(\gamma)) : \gamma \in E \text{ and}\utr \gamma \not \equiv 0 \psmod{\ell}\}\\ 
&\qquad\qquad\qquad\qquad\qquad\qquad  \smallsetminus 
\{c(\xi(\gamma)) : \gamma \in E \text{ and}\utr \gamma \equiv 0 \psmod{\ell}\}. 
\end{align*}
\item[b.] For each field $(L,\varphi)$, with $L$ encoded by the core-free subfield $L_0$ and $\varphi \leftrightarrow \xi$ as computed in Step 3${}^\prime$b, find a prime $\frakp$ such that $c(\Frob_\frakp,L_0) \in \Obc(E,\xi)$.
\end{enumerate}
\end{enumerate}

In computing $\Obc(E,\xi)$, of course it suffices to restrict to $\gamma$ in a set of conjugacy classes for $E$.

Step 4${}^{\prime}$ gives correct output because the set of cycle types in $\Obc(E,\xi)$ are precisely those for which \emph{every} conjugacy class in $E$ with the given cycle type is obstructing.  It is the simplest version, and it is the quickest to compute provided that $\Obc(E,\xi)$ is nonempty.

\begin{rmk}
In Step 4${}^{\prime}$a, there may be a cycle type which arises in two ways, from $\gamma,\gamma' \in E$, with $\utr \gamma \not \equiv 0 \psmod{\ell}$ and $\utr \gamma' \equiv 0 \psmod{\ell}$; such a cycle type is not guaranteed to be obstructing.
\end{rmk}

\begin{rmk} \label{rmk:restricttoallxi}
In a situation where there are many outer automorphisms $\xi$ to consider, it may be more efficient (but give potentially larger primes and possibly fail more often) to work with the set
\begin{equation} 
\Obc(E) \colonequals \bigcap_{\xi} \Obc(E,\xi) 
\end{equation}
consisting of cycle types with the property that every conjugacy class in $E$ under \emph{every} outer automorphism $\xi$ is obstructing.  In this setting, in Step 4${}^{\prime}$b, we can loop over just the fields $L$ and look for $\frakp$ with $c(\Frob_\frakp) \in \Obc(E)$.  
\end{rmk}

In the next stage, we seek to combine also cycle type information from $\Gal(K\,|\,F)$, arising as a permutation group from the field $K_0$.  Via the isomorphism $\varphi\colon \Gal(L \,|\, F) \xrightarrow{\sim} E$ and the construction of the core-free extension, as a permutation group $\Gal(L\,|\,F)$ is isomorphic to the permutation representation of $E$ on the cosets of $D$.  (The numbering might be different, but there is a renumbering for which the representations are equal.)  In the same way, the group $\Gal(K \,|\, F)$ is isomorphic as a permutation group to the permutation representation of $\pi(E) = G$ on the cosets of the subgroup $\pi(D)=H$, where $\pi\colon E \to G$ is the projection.  So we have the following second stage.

\begin{enumerate}
\item[4${}^{\prime\prime}$.]
Perform the following steps.
\begin{enumerate}
\item[a.] For each group $E$ computed in Step 2${}^\prime$ and each $\xi$ computed in Step 2${}^\prime$a for $E$, compute the set of pairs of cycle types 
\begin{align*} 
&\Obc(E,G,\xi) \colonequals \{(c(\xi(\gamma)),c(\pi(\gamma))) : \gamma \in E \text{ and }\utr \gamma \not \equiv 0 \psmod{\ell}\}\\ 
&\qquad\qquad\qquad\qquad\qquad\qquad  \smallsetminus 
\{(c(\xi(\gamma)),c(\pi(\gamma))) : \gamma \in E \text{ and }\utr \gamma \equiv 0 \psmod{\ell}\}. 
\end{align*}
\item[b.] For each field $(L,\varphi)$, with $L$ encoded by $L_0$ and $\varphi \leftrightarrow \xi$, find a prime $\frakp$ such that 
\[ (c(\Frob_\frakp,L_0),c(\Frob_\frakp,K_0)) \in \Obc(E,G,\xi). \]
\end{enumerate}
\end{enumerate}

Step 4${}^{\prime\prime}$ works for the same reason as in Step 4${}^\prime$: the cycle type pairs in $\Obc(E,G,\xi)$ are precisely those for which every conjugacy class in $E$ with the given pair of cycle types is obstructing.  The precomputation is a bit more involved in this case, but the check for each field is still extremely fast.

\begin{rmk}
Instead of the cycle type, a weaker alternative to Step 4${}^{\prime\prime}$ would be to record the order of $\Frob_\frakp \in \Gal(K\,|\,F)$.  
\end{rmk}

\begin{rmk}
Assuming that $\tr \overline{\rho}(\Frob_p)$ can be computed efficiently, one additional piece of data that may be appended to the pair of cycle types is $\tr \overline{\rho}(\gamma)$.
\end{rmk}

\begin{rmk}
If $L$ arises from several different choices of core-free subgroup, then these subgroups give different (but conjugate) fields $L_0$.  Because we are not directly accessing the conjugacy class above, but only cycle type information, it is possible that replacing $L_0$ by a conjugate field will give smaller witnesses.  In other words, in Step 4${}^{\prime}$b or $4{}^{\prime\prime}$b above, we could loop over the core-free subgroups $D$ and take the smallest witness among them.  
\end{rmk}

Finally, we may go all the way and compute conjugacy classes.  Write $[\gamma]_E$ for the conjugacy class of a group element $\gamma \in E$.

\begin{enumerate}
\item[4${}^{\prime\prime\prime}$.]
Perform the following steps.
\begin{enumerate}
\item[a.] For each group $E$ computed in Step 2${}^\prime$ and each $\xi$ computed in Step 2${}^\prime$a for $E$, compute the set of obstructing conjugacy classes 
\[ \Ob(E,\xi) \colonequals \{[\gamma]_E : \gamma \in E \text{ and }\utr \gamma \not \equiv 0 \psmod{\ell}\} \]
\item[b.] For each field $(L,\varphi)$, with $L$ encoded by $L_0$ and $\varphi \leftrightarrow \xi$, find a prime $\frakp$ such that $\Frob_\frakp \in \Ob(E,G,\xi)$.
\end{enumerate}
\end{enumerate}

We now explain some examples in detail which show the difference between these stages.

\begin{example}
Anticipating one of our three core cases, we consider $\redG = \GSp_4$ and $\ell=2$ over $F=\Q$.  (The reader may wish to skip ahead and read sections \ref{sec:galoisreps}--\ref{sec:GSp4} to read the details of the setup, but this example is still reasonably self-contained.)  We consider the case of a residual representation with image $G=S_5(b) \leq \GSp_4(\F_2)$ (see \eqref{table:subgroups}), and then a subgroup $E \leq \Liesp_4 \rtimes G$ with $\dim_{\F_2} V=10$.  We find a core-free subgroup $D$ where $\#H=10$ and $[V:W]=2$.  

We compute in Step 2${}^{\prime}$a that we need to consider $8$ automorphisms $\xi$, giving rise to $8$ homomorphisms $\varphi$.  With respect to one such $\xi$, we find that there are $48$ conjugacy classes that are obstructing.  Among these, computing as in Step 4${}^{\prime}$a, we find that $17$ are recognized by their $L_0$-cycle type:
\begin{equation} \label{eqn:ObcVrGxi}
\begin{aligned}
\Obc(E;\xi) &= \{3^6 2^1, 4^1 2^4 1^8, 4^1 2^5 1^6, 4^3 1^8, 4^3 2^1 1^6, 6^1 3^4 2^1, \\
&\qquad\qquad 8^1 4^2 2^2, 8^1 4^3, 10^2, 12^1 3^2 2^1, 12^1 6^1 2^1 \}. 
\end{aligned}
\end{equation}
If instead we call Step 4${}^{\prime\prime}$a, we find that $35=\#\Obc(E,G,\xi)$ are recognized by the pair of $L_0,K_0$-cycle types (and $22$ recognized by $L_0$-cycle type and $K_0$-order).  This leaves $13$ conjugacy classes that cannot be recognized purely by cycle type considerations, for which Step 4${}^{\prime\prime\prime}$ would be required.  

For the other choices of $\xi$, we obtain similar numbers but different cycle types.  If we restrict to just $L_0$-cycle types that work for \emph{all} such as in Remark \ref{rmk:restricttoallxi}, we are reduced to a set of $8$:
\begin{equation} \label{eqn:obcvVgexma} 
\Obc(E) = \{4^1 2^4 1^8, 4^1 2^5 1^6, 4^3 1^8, 4^3 2^1 1^6, 6^1 3^4 2^1, 8^1 4^2 2^2, 8^1 4^3, 10^2 \}. 
\end{equation}

To see how this plays out with respect to the sizes of primes, we work with the field $K$ arising as the Galois closure of $K_0=K^H$ defined by a root of the polynomial
\[ x^{10} + 3x^9 + x^8 - 10x^7 - 17x^6 - 7x^5 + 11x^4 + 18x^3 + 13x^2 + 5x + 1 \]
and similarly $L_0=L^{D}$ by a root of 
\[ x^{20} + 3x^{18} + 5x^{16} + 2x^{14} - 5x^{12} - 13x^{10} - 13x^8 - 6x^6 + x^4 + x^2 - 1. \]
If we restrict to the cycle types in \eqref{eqn:ObcVrGxi} (or \eqref{eqn:obcvVgexma}), we obtain the multiset of witnesses $\{5,5,5,5,23,23,29,29\}$.  If we work with $\Obc(E,G,\xi)$, we find $\{5,5,5,5,19,19,23,23\}$ instead; the difference is two cases where the witness $p=29$ is replaced by $p=19$, so we dig a bit deeper into one of these two cases.

In $L_0$, the factorization pattern of $19$ is $6^2 3^2 2^1$.  But apparently we cannot be guaranteed to have $\utr(\Frob_p) \equiv 1 \pmod{2}$ just looking at cycle type.  Indeed, there are three conjugacy classes with this cycle type: one of order $1280$ and two of order $2560$, represented by the permutations
\begin{align*}
&(1\, 9\, 18)(2\, 15\, 6\, 12\, 5\, 16)(3\, 20\, 7\, 13\, 10\, 17)(4\, 14)(8\, 11\, 19), \\
&(1\, 19\, 8\, 11\, 9\, 18)(2\, 15\, 6\, 12\, 5\, 16)(3\, 20\, 17)(4\, 14)(7\, 13\, 10), \\
&(1\, 10\, 2\, 3\, 8\, 4)(5\, 9\, 6)(7\, 17)(11\, 20\, 12\, 13\, 18\, 14)(15\, 19\, 16) 
\end{align*}
in $S_{20}$ mapping respectively to the matrices
\[ 
\left(\begin{smallmatrix}
1 & 0 & 1 & 1 & 0 & 0 & 0 & 0 \\
0 & 1 & 1 & 1 & 0 & 0 & 0 & 0 \\
0 & 1 & 0 & 1 & 0 & 1 & 1 & 1 \\
0 & 0 & 0 & 1 & 0 & 0 & 0 & 0 \\
0 & 0 & 0 & 0 & 1 & 0 & 1 & 1 \\
0 & 0 & 0 & 0 & 0 & 1 & 1 & 1 \\
0 & 0 & 0 & 0 & 0 & 1 & 0 & 1 \\
0 & 0 & 0 & 0 & 0 & 0 & 0 & 1
\end{smallmatrix}\right),
\left(\begin{smallmatrix}
1 & 0 & 1 & 1 & 0 & 0 & 0 & 0 \\
0 & 1 & 1 & 1 & 0 & 0 & 0 & 0 \\
0 & 1 & 0 & 1 & 0 & 0 & 0 & 0 \\
0 & 0 & 0 & 1 & 1 & 0 & 1 & 1 \\
0 & 0 & 0 & 0 & 1 & 0 & 1 & 1 \\
0 & 0 & 0 & 0 & 0 & 1 & 1 & 1 \\
0 & 0 & 0 & 0 & 0 & 1 & 0 & 1 \\
0 & 0 & 0 & 0 & 0 & 0 & 0 & 1 
\end{smallmatrix}\right),
\left(\begin{smallmatrix}
1 & 0 & 1 & 0 & 1 & 0 & 1 & 1 \\
1 & 0 & 0 & 0 & 1 & 0 & 1 & 1 \\
0 & 1 & 0 & 1 & 0 & 0 & 0 & 0 \\
0 & 1 & 0 & 0 & 1 & 0 & 1 & 1 \\
0 & 0 & 0 & 0 & 1 & 0 & 1 & 0 \\
0 & 0 & 0 & 0 & 1 & 0 & 0 & 0 \\
0 & 0 & 0 & 0 & 0 & 1 & 0 & 1 \\
0 & 0 & 0 & 0 & 0 & 1 & 0 & 0 \\
\end{smallmatrix}\right). \]
So precisely the first two conjugacy classes have upper trace $1$ and are obstructing, whereas the third has upper trace $0$ and is not obstructing.  So by cycle types in $L_0$ alone, indeed, we cannot proceed.

But we recover using the $K_0$-cycle type.  For the obstructing classes, the cycle type in the permutation representation of $G$ is $3^3 1^1$, whereas for the nonobstructing class the cycle type is $6^1 3^1 1^1$.  We compute that the factorization pattern for $19$ in $K_0$ is type $3^3 1^1$, which means $19$ belongs to an obstructing class.  
If we go all the way to the end, we can compute that the conjugacy class of $\Frob_{19}$ in fact belongs to the second case.
\end{example}

\section[Galois representations]{Abelian surfaces, paramodular forms, and Galois representations} \label{sec:galoisreps}

We pause now to set up notation and input from the theory of abelian surfaces, paramodular forms, and Galois representations in our case of interest. 

\subsection{Galois representations from abelian surfaces}

Let $A$ be a polarized abelian variety over $\Q$.  For example, if $X$ is a nice (smooth, projective, geometrically integral) genus $g$ curve over $\Q$, then its Jacobian $\Jac X$ with its canonical principal polarization is a principally polarized abelian variety over $\Q$ of dimension $g$.  Let $N \colonequals \cond(A)$ be the conductor of $A$.  We say $A$ is \defi{typical} if $\End(\Abar)=\Z$, where $\Abar \colonequals A_{\Qbar}$ is the base change of $A$ to $\Qbar$.  

\begin{lem} \label{lem:nonsquare}
Let $A$ be a simple, semistable abelian surface over $\Q$ with
nonsquare conductor. Then $A$ is typical.
\end{lem}

\begin{proof}
By Albert's classification, either $\End(A) =\Z$ or $\End(A)$ is an
order in a quadratic field.  In the latter case, $\cond(A)$ is a
square by the conductor formula (see Brumer--Kramer \cite[Lemma
3.2.9]{BK14}), a contradiction.  Therefore $\End(A)=\Z$.  Since $A$ is
semistable, all endomorphisms of $A^{\textup{al}}$ are defined over
$\Q$ by a result of Ribet \cite[Corollary 1.4]{Ribet75}.  Thus
$\End(A^{\textup{al}})=\End(A)=\Z$, and $A$ is typical.
\end{proof}

\begin{lem} \label{lem:squarefree}
An abelian surface over $\Q$ of prime conductor is typical.
\end{lem}

\begin{proof}
If $A$ is not simple over $\Q$, then we have any isogeny $A \sim A_1
\times A_2$ over $\Q$ to the product of abelian varieties $A_1,A_2$
over $\Q$, and $\cond(A)=\cond(A_1)\cond(A_2)$.  But since $A$ is
prime, without loss of generality $\cond(A_1)=1$, contradicting the
result of Fontaine \cite{Fontaine} that there is no abelian variety
over $\Q$ with everywhere good reduction.  Therefore $A$ is simple
over~$\Q$.  Since $N=\cond(A)$ is prime, $A$ is semistable at $N$, and
the result then follows from Lemma \ref{lem:nonsquare}.
\end{proof}

From now on, suppose that $g=2$ and $A$ is a polarized abelian surface over $\Q$.
Let $\ell$ be a prime with $\ell \nmid N$ and $\ell$ coprime to the degree of the polarization on $A$.   Let $S$ be a finite set of places of $\Q$ containing $\ell,\infty$ and the primes of bad reduction of $A$.  
  Let 
\[ \chi_\ell \colon \Gal_{\Q,S} \to \Z_\ell^\times \] 
denote the $\ell$-adic cyclotomic character, so that $\chi_\ell(\Frob_p)=p$.  Then the action of $\Gal_\Q$ on the $\ell$-adic Tate module 
\[ T_{\ell}(A) \colonequals \varprojlim_n A[\ell^n] \simeq \coho^1_{\textup{\'et}}(A,\Z_\ell)\spcheck \simeq \Z_\ell^4 \] 
(where $A[\ell^n]$ denotes the $\ell^n$-torsion of $A$) provides a continuous Galois representation
\begin{equation} \label{eqn:rhoAell}
\rho_{A,\ell}\colon \Gal_{\Q,S} \to \GSp_4(\Z_\ell)
\end{equation}
with determinant $\chi_\ell^2$ and similitude character $\chi_\ell$ that is unramified outside $\ell N$.  We may reduce the representation \eqref{eqn:rhoAell} modulo $\ell$ to obtain a residual representation 
\[ \overline{\rho}_{A,\ell} \colon \Gal_{\Q,S} \to \GSp_4(\F_\ell), \] 
which can be concretely understood via the Galois action on the field $\Q(A[\ell])$.

For a prime $p \neq \ell$, slightly more generally we define
\begin{equation} 
L_p(A,T) \colonequals \det(1-T\Frob_p^*  \mid \coho^1_{\textup{\'et}}(\Abar,\Q_\ell)^{I_p}) 
\end{equation}
where $\Frob_p^*$ is the geometric Frobenius automorphism, $I_p \leq \Gal_{\Q,S}$ is an inertia group at $p$, and the definition is independent of the auxiliary prime $\ell \neq p$ (by the semistable reduction theorem of Grothendieck \cite[Exp.\ IX, Th\'eor\`eme 4.3(b)]{SGA7}).  In particular, when $p \nmid \ell N$, we have 
\begin{equation}
\det(1-\rho_{A,\ell}(\Frob_p)T)=L_p(A,T) = 1-a_p T + b_{p^2} T^2 - p a_p T^3 + p^2 T^4 \in 1 + T \Z[T]. 
\end{equation}
Moreover, if $A=\Jac X$ and $p$ does not divide the minimal discriminant $\Delta$ of $X$, then 
\[ Z(X\bmod p,T) \colonequals \exp\left(\sum_{r=1}^{\infty} \# X(\F_{p^r}) \frac{T^r}{r}\right) = \frac{L_p(A,T)}{(1-T)(1-pT)} \]
so the polynomials $L_p(A,T)$ may be efficiently computed by counting points on $X$ over finite fields.  We define
\begin{equation}
L(A,s) \colonequals \prod_p L_p(A,p^{-s})^{-1};
\end{equation}
this series converges for $s \in \C$ in a right half-plane.

\subsection{Paramodular forms} \label{sec:paramodularforms}

We follow Freitag \cite{Freitag} for the theory of Siegel modular forms.  Let $\Half_2   \subset \M_2(\C)$ be the Siegel upper half-space.  We write matrices $M=\begin{psmallmatrix} A & B \\ C & D \end{psmallmatrix} \in \GL_4(\R)$ as block matrices, in particular $J \colonequals \begin{psmallmatrix} 0 & 1 \\ -1 & 0 \end{psmallmatrix} \in \GL_4(\R)$ as usual.  We write ${}\transpose$ for the transpose.  Define
\[ \GSp_4^{+}(\R) \colonequals \{ M \in \GL_4(\R) : M\transpose JM=\mu J\text{ for some $\mu \in \R_{>0}$}\}. \]
For $M \in \GSp_4^{+}(\R)$, we have $\mu=\det(M)^{1/2}>0$.  

For a holomorphic function $f\colon \Half_2 \to \C$ and $M \in \GSp_4^+(\R)$ and $k \in \Z_{\geq 0}$, we define the classical slash
\begin{equation} 
(f\,|_k\,M)(Z) \colonequals \mu^{2k-3} \det(CZ + D)^{-k} f( (AZ+B)(CZ+D)\inv).
\end{equation}

Let $\Gamma \leq \Sp_4(\R)$ be a subgroup commensurable with $\Sp_4(\Z)$.  We denote by 
\[ M_k(\Gamma) \colonequals \{f \colon \Half_2 \to \C : (f\,|_k\,\gamma)(Z) = f(Z) \text{ for all $\gamma \in \Gamma$}\} \]
the $\C$-vector space of Siegel modular forms with respect to $\Gamma$, and $S_k(\Gamma) \subseteq M_k(\Gamma)$ the subspace of forms vanishing at the cusps of $\Gamma$, called the space of cuspforms.

To each double coset $\Gamma M \Gamma$ with $M \in \GSp_4^{+}(\Q) \colonequals \GSp_4(\Q) \cap \GSp_4^{+}(\R)$, 
we define the \defi{Hecke operator}
\begin{equation} \label{eqn:TMGammadef}
\HeckeT({\Gamma M \Gamma})\colon M_k( \Gamma ) \to M_k( \Gamma ) 
\end{equation}
as follows: from a  decomposition $ \Gamma M \Gamma = \bigsqcup_j \Gamma M_j $ of the double coset  into disjoint single cosets, we define
$ f \,|_k\,  \HeckeT({\Gamma M \Gamma}) = \sum_j f \,|_k\, M_j$.
The action is well-defined, depending only on the double coset, and $\HeckeT({\Gamma M \Gamma})$ maps $S_k(\Gamma)$ to $S_k(\Gamma)$.  

Let $N \in \Z_{\geq 1}$.  The \defi{paramodular group} $K(N)$ of level~$N$ in degree two is 
defined by 
\begin{equation}
K(N) \colonequals
\begin{pmatrix} 
\Z  &  N\Z  &  \Z  &  \Z  \\
\Z  &  \Z  &  \Z  &  N^{-1}\Z  \\
\Z  &  N\Z  &  \Z  &  \Z  \\
N\Z  &  N\Z  &  N\Z  &  \Z  
\end{pmatrix} 
\cap \Sp_4(\Q).
\end{equation}
The paramodular group $K(N)$ has a normalizing \defi{paramodular Fricke involution},      
$\mu_N \in \Sp_4(\R)$, given by 
\[ \mu_N = \begin{pmatrix}(F_N^{-1})\transpose & 0 \\ 0 F_N \end{pmatrix}, \] 
where 
$F_N =\frac{1}{ \sqrt{N} } \mymat{0}{1}{-N}{0}$ is the Fricke involution
for $\Gamma_0(N)$.
Consequently, for all $k$ we may decompose
\begin{equation} \label{eqn:eigenspaceplusminus}
M_k( K(N) ) = M_k(K(N))^{+} \oplus M_k( K(N))^{-} 
\end{equation}
into plus and minus $\mu_N$-eigenspaces.  

Write $e(z)=\exp(2\pi\sqrt{-1} z)$ for $z \in \C$.  The Fourier expansion of $f \in M_k( K(N) )$ is
\begin{equation} \label{eqn:Fourierexppara}
f(Z) = \sum_{T\ge 0} a(T;f)e(\tr(TZ)) 
\end{equation} 
for $Z \in \Half_2$ and the sum over semidefinite matrices
\[ T =\begin{pmatrix} n  &  r/2  \\  r/2  &  Nm \end{pmatrix} \in \M_2^{\textup{sym}}(\Q)_{\ge 0} \text{ with $n,r,m \in \Z$.} \]
  For a subring $R \subseteq \C$, we denote by
 \begin{equation} \label{eqn:MkKNR}
 M_k(K(N),R) \colonequals \{f \in M_k(K(N)) : a(T;f) \in R \text{ for all $T \geq 0$}\} 
 \end{equation}
 the $R$-module of paramodular forms whose 
 Fourier coefficients all lie in~$R$, and similarly we write $S_k(K(N),R)$ for cusp forms and $S_k(K(N),R)^{\pm}$ for the eigenspaces under the Fricke involution.  The ring of paramodular forms with coefficients in $R$ 
 \[ M( K(N),R) \colonequals \bigoplus_{k=0}^{\infty} M_k( K(N),R) \]
 is a graded $R$-algebra.  

For a prime~$p \nmid N$, 
the first (more familiar) Hecke operator we will use is 
\begin{equation} 
T(p) \colonequals \HeckeT({K(N) \diag(1,1,p,p) K(N)}) 
\end{equation}
whose decomposition into left cosets is given by 
\begin{equation} \label{eqn:heckeopersM}
\begin{aligned}
&K(N) \diag(1,1,p,p) K(N) \\
&\qquad = K(N) 
\begin{pmatrix} 
p  &  0  &  0  &  0  \\
0  &  p  &  0  &  0  \\ 
    &         &  1  &  0  \\
    &         &  0  &  1 
\end{pmatrix}
 + \sum_{i \bmod{p}}
 K(N) 
\begin{pmatrix} 
1  &  0  &  i  &  0  \\
0  &  p  &  0  &  0  \\ 
    &         &  p  &  0  \\
    &         &  0  &  1 
\end{pmatrix}  
\\
& \qquad\qquad + \sum_{i,j \bmod{p}} K(N) 
\begin{pmatrix} 
p  &  0  &  0  &  0  \\
i  &  1   &  0  &  j  \\ 
    &         &  1  &  -i  \\
    &         &  0  &  p 
\end{pmatrix}
 + \sum_{i,j,k \bmod{p}}  
 K(N) 
\begin{pmatrix} 
1  &  0  &  i  &  j  \\
0  &  1   &  j  &  k  \\ 
    &         &  p  &  0  \\
    &         &  0  &  p 
\end{pmatrix}
\end{aligned}
\end{equation}
with indices taken over residue classes modulo $p$. 
Writing $T[u]=u\transpose Tu$ for $T,u \in \M_2(\Q)$, the action of $T(p)$ on Fourier coefficients $a(T;f)$ is given by
\begin{equation}
\begin{aligned} \label{eqn:heckeopersFC}
a( T; f\,|_k\, T(p)) &= a( pT;f) 
+ p^{k-2} \sum_{j \text{\rm \ mod}\  p} a\!\left( \frac1p \,T \begin{bmatrix} 1 & 0 \\ j & p \end{bmatrix};f\right)  \\
& \qquad + p^{k-2} a\!\left( \frac1p \,T \begin{bmatrix} p & 0 \\ 0 & 1 \end{bmatrix};f\right)
+ p^{2k-3} a\!\left( \frac1p \,T;f\right).  
\end{aligned}
\end{equation}
Hence for $k \ge 2$, the Hecke operator $T(p)$ stabilizes $S_k(K(N),R)$.  In particular, taking $R=\Z$ we see that if $f$ has integral Fourier coefficients, then $f\,|_k\, T(p)$  
 has integral Fourier coefficients for $k \ge 2$.  

We will also make use of another, perhaps less familiar, Hecke operator.  For $K(N)$ and a prime~$p \nmid N$, we define
\begin{equation} 
T_1(p^2)= \HeckeT({K(N) \diag(1,p,p^2,p) K(N)}). 
\end{equation}

\begin{lem}
The coset decomposition for $T_1(p^2)$ is given by:
\begin{equation} \label{eqn:heckeopersMM}
\begin{aligned}
&K(N) \diag(1,p,p^2,p) K(N) \\
&\qquad = K(N) 
\begin{pmatrix} 
p  &  0  &  0  &  0  \\
0  &  p^2  &  0  &  0  \\ 
    &         &  p  &  0  \\
    &         &  0  &  1 
\end{pmatrix}
 + \sum_{i \bmod{p}}
 K(N) 
\begin{pmatrix} 
p^2  &  0  &  0  &  0  \\
pi  &  p  &  0  &  0  \\ 
    &         &  1  &  -i  \\
    &         &  0  &  p 
\end{pmatrix}  
\\
& \qquad\qquad + \sum_{i \not\equiv 0 \bmod{p}} K(N) 
\begin{pmatrix} 
p  &  0  &  i  &  0  \\
0  &  p   &  0  &  0  \\ 
    &         &  p  & 0  \\
    &         &  0  &  p 
\end{pmatrix}
 + \sum_{\substack{i \bmod{p},\\ j \not\equiv 0 \bmod{p}}}  
 K(N) 
\begin{pmatrix} 
p  &  0  &  i^2j  &  ij  \\
0  &  p   &  ij  &  j  \\ 
    &         &  p  &  0  \\
    &         &  0  &  p 
\end{pmatrix} \\
& \qquad\qquad + \sum_{\substack{i \bmod{p}, \\  j \bmod{p^2}}} K(N) 
\begin{pmatrix} 
1  &  0  &  j  &  i  \\
0  &  p   &  pi  &  0  \\ 
    &         &  p^2  &  0  \\
    &         &  0  &  p 
\end{pmatrix}
 + \sum_{\substack{i,j \bmod{p}, \\  k \bmod{p^2}}}    
 K(N) 
\begin{pmatrix} 
p  &  0  &  0  &  pj  \\
i  &  1   &  j  &  k  \\ 
    &         &  p  &  -pi  \\
    &         &  0  &  p^2 
\end{pmatrix}
\end{aligned}
\end{equation}
\end{lem}

\begin{proof}
The cosets are from Roberts--Schmidt \cite[(6.6)]{RS07} after swapping rows one and two and columns one and two, applying an inverse, and multiplying by the similitude~$p^2$.
\end{proof}

Define the indicator function $\indic(p\mid y)$ by $1$ if $p\mid y$ and by $0$ if $ p \nmid y$.  Then the action of $T_1(p^2)$ on the Fourier coefficients is:
\begin{equation} \label{eqn:heckeopersFCMM}
\begin{aligned} 
a( T; f\,|_k\,T_1(p^2)) &= 
  p^{k-3} \sum_{x \bmod{p}} a( T \begin{bmatrix} 1 & 0 \\ x & p \end{bmatrix};f) 
+ p^{k-3}  a\!\left( T \begin{bmatrix} p & 0 \\ 0 & 1 \end{bmatrix};f\right)  \\
&\qquad + p^{3k-6}   \sum_{j \bmod{p}} a\!\left( \frac1{p^2} \,T \begin{bmatrix} 1 & 0 \\ j & p \end{bmatrix};f\right)
+ p^{3k-6} a\!\left( \frac1{p^2}\,T \begin{bmatrix} p & 0 \\ 0 & 1 \end{bmatrix} ;f\right)  \\
&\qquad + p^{2k-6} \left( p\,\indic\!\left(p \mid T\begin{bmatrix} 1  \\ 0 \end{bmatrix} \right)-1 \right) a( T;f) \\
&\qquad + p^{2k-6} \sum_{\lambda \bmod{p}}\left( p\,\indic\!\left(p \mid T\begin{bmatrix} \lambda  \\ 1 \end{bmatrix} \right)-1 \right)a( T;f).
\end{aligned}
\end{equation}
Hence for $k \ge 3$, the Hecke operator $T_1(p^2)$ stabilizes $S_k(K(N),R)$.  
In particular, if $f$ has integral Fourier coefficients, then $f\,|_k\,T_1(p^2)$  
 has integral Fourier coefficients for $k \ge 3$.  However, for $k=2$, we only 
 know that $p^2 f\,|_k \,T_1(p^2)$  is integral when $f$ is (and there 
 are examples where $f\,|_2\,T_1(p^2)$ has $p^2$ in the denominator of some Fourier coefficients).

Summarizing the above, we have:
\begin{equation}
\begin{aligned} \label{everything}
T(p)   &= \HeckeT(K(N) \diag(1,1,p,p) K(N))  ; \qquad 
\deg T(p)   =  (1+p)(1+p^2) \\ 
T_1(p^2)   &= \HeckeT(K(N) \diag(1,p,p^2,p) K(N))  ; \quad 
\deg T_1(p^2)     =  (1+p)(1+p^2)p.
\end{aligned}
\end{equation}
We define two new operators:
\begin{equation}
\begin{aligned}
T_2(p^2) & \colonequals \HeckeT(K(N) \diag(p,p,p,p) K(N))  = p^{2k-6}\id \\
B(p^2)  &\colonequals p(  T_1(p^2)  + (1+p^2) T_2(p^2))  
\end{aligned}
\end{equation}

If $f$ is an eigenform of weight $k$ for the operators $T(p)$ and $T_1(p^2)$, with corresponding eigenvalues $a_p(f),a_{1,p^2}(f) \in \C$, then $f$ is an eigenform for the operator $B(p^2)$ with eigenvalue 
\begin{equation} 
b_{p^2}(f) \colonequals p\,a_{1,p^2}(f) + p^{2k-5}(1+p^2).
\end{equation}

\begin{lem} \label{lem:k2fint}
If $k=2$ and $f$ has integral Fourier coefficients, then $b_{p^2}(f) \in \Z$.
\end{lem}

\begin{proof}
We have observed that $p^2a_{1,p^2}(f) \in \Z$.  From \ref{eqn:heckeopersFCMM}, we observe the congruence 
 \[ p^2(f \,|_2\, T_1(p^2)) = p^2a_{1,p^2}(f)f \equiv -f \pmod{p}. \]  
so $p \mid (p^2a_{1,p^2}(f) + 1)$.  Therefore
\[ b_{p^2}(f) = p\,a_{1,p^2}(f)+(1+p^2)/p = (p^2a_{1,p^2}(f)+1)/p + p \in \Z. \qedhere \]
\end{proof}

Following Roberts--Schmidt \cite{RS06,RS07}, to $f$ we then assign the \defi{spinor Euler factor} at $p \nmid N$ in the arithmetic normalization by
\begin{equation} \label{eqn:QpfT}
Q_p(f,T) \colonequals 1 - a_p(f)T +b_{p^2}(f)T^2 - p^{2k-3} a_p(f)T^3 + p^{4k-6}T^4 \in 1+T\C[T].
\end{equation}
We will also call $Q_p(f,T)$ the \defi{spinor Hecke polynomial} at $p$. 
If $f$ has integral Fourier coefficients, then by Lemma \ref{lem:k2fint} we have $Q_p(f,T) \in 1+T\Z[T]$.

\subsection{Galois representations from Siegel modular forms}

We now seek to match the Galois representation coming from an abelian surface with one coming from an automorphic form.  In this section, we explain the provenance of the latter.  

We follow the presentation of Schmidt \cite{Schmidt17} for the association of an automorphic representation to a paramodular eigenform.  Let $\Gamma \le \GSp_4(\Q)^{+}$ be a subgroup commensurable with~$\Sp_4(\Z)$ and let $f \in S_k(\Gamma)$ be a cuspidal eigenform at all but finitely many places.  
In general, the representation~$\pi_f$ generated by the adelization of~$f$ 
may be reducible and hence not an automorphic representation at all.  It is still 
possible however, to associate a global Arthur parameter for $\GSp_4(\A)$ to~$f$ as follows.  
Because $f$ is cuspidal, the representation~$\pi_f$ decomposes as 
the direct sum of a finite number of automorphic representations, and each summand has the same 
global Arthur parameter among one of six types: the general type~{\bf (G)}, 
the Yoshida type~{\bf (Y)}, the finite type~{\bf (F)}, or types~{\bf (P)}, {\bf (Q)} or {\bf (B)} named after parabolic subgroups. 
Thus we may associate a global Arthur parameter directly to a paramodular eigenform~$f$.  The only type of 
global Arthur parameter that concerns us here is type~{\bf (G)} given by the formal tensor 
$\mu \sqtimes 1$, where $\mu$ is a cuspidal, self-dual, symplectic, unitary, automorphic representation 
of $\GL_4(\A)$ and $1$ is the trivial representation of $\SU_2(\A)$.  

\begin{rmk}
One can consider the eigenforms of type~{\bf (G)} to be those that \emph{genuinely} 
belong on $\GSp_4$.  
\end{rmk}

Second, when $f$ is of type~{\bf (G)} or~{\bf (Y)},  the associated representation~$\pi_f$ is irreducible 
and~$f$ is necessarily an eigenform at all good primes.  
Third, the type of~$f$ may be determined by checking \emph{one} Euler factor at a good prime.  
We state the paramodular case $\Gamma=K(N)$.  

\begin{prop}[\textup{Schmidt}]  \label{prop:schmidt} 
Let $f \in S_k(K(N))$ be a cuspidal eigenform for all primes $p \nmid N$.  Let $p \nmid N$ be prime 
and let $Q_p(f,T)$ be the Hecke polynomial of~$f$ at~$p$ defined in \textup{\eqref{eqn:QpfT}} in the arithmetic normalization.  Then $f$ is of type~{\bf (G)} if and only if 
all reciprocal roots of $Q_p(f,T)$ have complex absolute value $p^{k-\frac32}$. 
\end{prop}

\begin{proof}
Converting from analytic to arithmetic normalization, by Schmidt \cite[Proposition 2.1]{Schmidt17} the stated local factor condition
  implies that $f$ is of type~{\bf (G)} or {\bf (Y)}, but paramodular cusp forms cannot be type~{\bf (Y)} also by Schmidt \cite[Lemma
  2.5]{Schmidt17}.
\end{proof}

Fourth, continuing in the paramodular case~$\Gamma=K(N)$, 
the global conductor of $\pi_f$ divides~$N$, and is equal to~$N$ 
if and only if $f$ is a newform.  
Finally, if $f$ is a newform---see Roberts--Schmidt \cite{RS06} for the global newform theory of paramodular forms---then $f$ is a Hecke eigenform at all  primes and for 
all paramodular Atkin-Lehner involutions.

We need one final bit of notation, concerning archimedian $L$-parameters. 
The real Weil group is $W(\R) = \C^{\times} \cup \C^{\times}j$, with $j^2=-1$ 
and $j z j\inv = {\bar z}$ for $ z \in \C^{\times}$.  
For $w,m_1,m_2 \in \Z$ with $m_1 > m_2 \ge 0$ and $w+1 \equiv m_1 + m_2 \pmod{2}$, we 
define the \defi{archimedean $L$-parameter} 
$\phi(w,m_1,m_2)\colon W(\R) \to \GSp_4(\R)$ 
by sending $z\in \C^{\times}$ to the diagonal matrix
\begin{equation}
 |z|^{-w} \diag\biggl( \left(\frac{z}{\bar z}\right)^{\frac{m_1+m_2}{2}},  \left(\frac{z}{\bar z}\right)^{\frac{m_1-m_2}{2}} , \left(\frac{z}{\bar z}\right)^{\frac{m_2-m_1}{2}}, \left(\frac{z}{\bar z}\right)^{\frac{-(m_1+m_2)}{2}}  \biggr)
\end{equation}
and $j$ to the antidiagonal matrix 
$\antidiag( (-1)^{w+1},  (-1)^{w+1}, 1, 1)$.  
The archimedean $L$-packet of $\GSp_4(\R)$ corresponding to $\phi(w,m_1,m_2)$ has two elements, one holomorphic and one generic: for $m_2 >0$ these are both discrete series representations, whereas for $m_2=0$ they are limits of discrete series.     
     
We are now ready to associate a Galois representation to a paramodular eigenform of type~\textbf{(G)}.  

\begin{thm}[Taylor--Laumon--Weissauer--Schmidt--Mok] \label{thm:galoisrep} 
Let $f \in S_k(K(N))$ be a Siegel paramodular newform of weight $k \geq 2$ and level $N$.  
 Suppose that $f$ is of type~{\bf (G)}.  
Then for any prime $\ell \nmid N$, there exists a
continuous, semisimple Galois representation   
\[
\rho_{f,\ell}\colon \Gal_\Q \to \GSp_4(\Q_\ell^{\al})
\]
with the following properties:
\begin{enumroman}
\item $\det(\rho_{f,\ell})=\chi_{\ell}^{4k-6}$;
\item The similitude character of $\rho_{f,\ell}$ is $\chi_\ell^{2k-3}$;
\item $\rho_{f,\ell}$ is unramified outside $\ell N$;
\item $\det(1-\rho_{f,\ell}(\Frob_p) T) = Q_p(f,T)$ for all $p \nmid \ell N$; and
\item The local Langlands correspondence holds for all primes $p \neq \ell$, up to semisimplification.
\end{enumroman}
\end{thm}

By (v), we mean that the Weil--Deligne representations associated to the restriction of the Galois representation $\rho_{f,\ell}$ to $\Gal(\Q_p^{\textup{al}}\,|\,\Q_p)$ agrees with that associated to the $\GL_n(\Q_p)$-representation $\pi_p$ attached by the local Langlands correspondence up to semisimplification \emph{without} information about the nilpotent operator $N$: in the notation of Taylor--Yoshida \cite[p.\ 468]{TaylorYoshida} we mean $(V,r,N)^{\textup{ss}}=(V,r^{\textup{ss}},0)$.

\begin{proof}
The existence and properties (i)--(ii) follow from the construction and an argument of Taylor \cite[Example 1, section 1.3]{Taylor}.  Properties (iii) and (iv) are provided by Berger--Klosin \cite[Theorem 8.2]{BKPSY} (they claim in the subsequent Remark 8.3 that the result is ``well-known'').  

We now sketch the construction, and we use the argument of Mok to conclude also property (v).  By the discussion above, following Schmidt \cite{Schmidt17}, we may attach to $f$ a cuspidal automorphic representation $\Pi_f$ of $\GSp_4(\A)$ of type~{\bf (G)}.  
The hypothesis that $f$ is of type~\textbf{(G)} assures that the automorphic representation $\Pi_f$ is irreducible.  If $k \geq 3$, then the automorphic representation is of cohomological type, and from a geometric construction we obtain a Galois representation $\rho_{f,\ell}\colon \Gal_\Q \to \GSp_4(E)$ by work of Laumon \cite{Laumon} and Weissauer \cite[Theorems I and IV]{Weissauer}, where $E$ is the finite extension of $\Q_\ell$ containing the Hecke eigenvalues of $f$ (choosing an isomorphism between the algebraic closure of $\Q$ in $\C$ and in $\Q_\ell^{\al}$): one shows that the representation takes values in $\GL_4(E)$ and that it preserves a nondegenerate symplectic bilinear form invariant under $\rho_{f,\ell}(\Gal_\Q)$ so lands in $\GSp_4(E)$.  Thereby, properties (i)--(iv) are verified.  

For all $k \geq 2$, with the above conventions (including archimedean $L$-parameters) we verify that $\Pi_f$ satisfies the hypotheses of a theorem of Mok \cite[Theorem 4.14]{Mok}:  from this theorem we obtain a unique, continuous semisimple representation $\rho_{f,\ell} \colon \Gal_\Q \to \GL_4(\Q_\ell^{\al})$ where $\Q_\ell^{\al}$ is an algebraic closure of $\Q_\ell$.  For $k=2$, Mok constructs the representation by $\ell$-adic deformation using Hida theory from those of Laumon and Weissauer, and so properties (i)--(iv) and the fact that the representation is symplectic continue to hold in the limit; and property (v) is a conclusion of his theorem.  

To illustrate this convergence argument, we show that the representation is symplectic.  Let $\{f_n\}_n$ be a sequence of Siegel paramodular forms of weights $k_n > 2$ such that $f_n$ converge $p$-adically to $f$ (for example, multiplying by powers of the Hasse invariant).  By the previous paragraph, each $f_n$ is symplectic with representation $\rho_n$ so 
\begin{equation} \label{eqn:rhotrivpsi}
\textstyle{\bigwedge}^2 \rho_n (3-2k_n) \simeq \rho_{\textup{triv}} \oplus \psi_n
\end{equation} 
is equivalent to the direct sum of the trivial representation $\rho_{\textup{triv}}$ of degree $1$ and the representation $\psi$ of degree $5$ with values in $\SO_5(\Q_\ell^{\al})$.  The sequence $\Tr \psi_n$ of pseudorepresentations converges to a pseudorepresentation by \eqref{eqn:rhotrivpsi} and continuity of the trace, and this limit is the trace of a representation $\psi$.  From this identity of traces, we conclude
\[ \textstyle{\bigwedge}^2 \rho(-1) \simeq \rho_{\textup{triv}} \oplus \psi \]
and thus $\rho$ is symplectic with cyclotomic similitude character.

Mok's theorem relies on work of Arthur in a crucial way.  For further attribution and discussion, see Mok \cite[About the proof, pp.\ 524ff]{Mok} and the overview of the method by Jorza \cite[\S\S 1--3]{Jorza}. 
\end{proof}

Let $f$ be as in Theorem \textup{\ref{thm:galoisrep}}, with Galois representation $\rho_{f,\ell} \colon \Gal_{\Q,S} \to \GSp_4(\Q_\ell^{\al})$ where $S \colonequals \{p : p \mid N\} \cup \{\ell,\infty\}$.  By the Baire category theorem, we may descend the representation to a finite extension $E' \subseteq \Q_\ell^{\al}$ of $\Q_\ell$.  Let $\frakl'$ be the prime above $\ell$ in the valuation ring $R'$ of $E'$ and let $k'$ be the residue field of $R'$.  Choose a stable $R'$-lattice in the representation space $V' \colonequals (E')^4$ and reduce modulo $\frakl'$; the semisimplification yields a semisimple residual representation $\overline{\rho}_{f,\ell}^{\textup{ss}} \colon \Gal_{\Q,S} \to \GL_4(k')$, unique up to equivalence.  

Applying a recent result of Serre, we now show that the residual representation is symplectic.

\begin{lem} \label{lem:Serrealtisthere}
The semisimplification $\overline{\rho}_{f,\ell}^{\textup{ss}} \colon \Gal_{\Q,S} \to \GL_4(k')$
is compatible with a nondegenerate alternating form with similitude character $\overline{\chi}_\ell^{2k-3}$; in particular, up to equivalence its image lies in $\GSp_4(k')$.
\end{lem}

\begin{proof}
We refer to Serre \cite{Serre:altform} and to the Appendix.  Let $\langle\,,\,\rangle$ be the alternating form on $V'$ with similitude character $\epsilon \colonequals \chi_\ell^{2k-3}$ provided by Theorem \ref{thm:galoisrep}.  Then $V'$ is a finite-dimensional $E'$-vector space equipped with a continuous action by $\Gal_{\Q,S}$ via $\rho_{f,\ell}$.  Moreover, for all $\sigma \in \Gal_{\Q,S}$ and all $x,y \in V'$ we have 
\begin{equation} \label{eqn:bilinearformalt}
\begin{aligned} 
\langle \sigma x, \sigma y \rangle = \chi_\ell^{2k-3}(\sigma) \langle x, y \rangle;
\end{aligned}
\end{equation}
i.e., the pairing $\langle\,,\rangle$ is $\chi_\ell^{2k-3}$-covariant under the action of $\Gal_{\Q,S}$.  Let $V'_{k'}$ be the $k'$-vector space underlying the semisimplification $\overline{\rho}_{f,\ell}^{\textup{ss}}$.  Then Serre proves in Theorem 1 in the Appendix that there exists a nondegenerate $k'$-bilinear alternating form on $V'_{k'}$ that again is covariant with respect to (the reduction of)
$\chi_\ell^{2k-3}$ under the action of $\Gal_{\Q,S}$. 

The final statement holds because, up to equivalence by $\GL_4(k')$, we may assume the alternating form is the standard form, so now the image lands in $\GSp_4(k')$, as claimed.  
\end{proof}

Next, we seek descent preserving the symplectic form.  Let $E$ be the extension of $\Q_\ell$ generated by the Hecke eigenvalues of $f$ (with respect to a choice of isomorphism between the algebraic closure of $\Q$ in $\C$ and in $\Q_\ell^{\al}$); then $E$ also contains all coefficients of the Hecke polynomials $Q_p(f,T)$.  Let $R$ be the valuation ring of $E$ and let $k$ be its residue field.  We have $E \subseteq E'$, and we would like to be able to descend the representation to take values in $\GSp_4(E)$.  However, there is a possible obstruction coming from the Brauer group of $\Q_\ell$; such an obstruction arises for example in the Galois representation afforded by a QM abelian fourfold at a prime $\ell$ dividing the discriminant of the quaternion algebra $B$, which has image in $\GL_2(B \otimes \Q_\ell)$ and not $\GSp_4(\Q_\ell)$.  
Under an additional hypothesis, we may ensure descent following Carayol and Serre as follows.  

\begin{lem} \label{lem:descentwhenirred}
With hypotheses as in Theorem \textup{\ref{thm:galoisrep}}, the following statements hold.
\begin{enumalph}
\item The semisimplified residual representation $\overline{\rho}_{f,\ell}^{\textup{ss}}$ descends to
\[ \overline{\rho}_{f,\ell}^{\textup{ss}} \colon \Gal_{\Q,S} \to \GSp_4(k) \]
up to equivalence.
\item If $\overline{\rho}_{f,\ell}^{\textup{ss}}=\overline{\rho}_{f,\ell}$ is absolutely irreducible, then $\rho_{f,\ell}$ descends to
\[ \rho_{f,\ell} \colon \Gal_{\Q,S} \to \GSp_4(E) \]
up to equivalence, where $E$ is the extension of $\Q_\ell$ generated by the Hecke eigenvalues of $f$ as above.
\end{enumalph}
\end{lem}

\begin{proof}
We begin with (a).  First, a semisimple representation into $\GL_4(k')$ is determined by its traces, and so up to equivalence we may descend $\overline{\rho}_{f,\ell}^{\textup{ss}}$ to take values in $\GL_4(k) \subseteq \GL_4(k')$ (for a complete proof, see e.g.\ Taylor \cite[Lemma 2, part 2]{Taylor}).  The semisimplification $\overline{\rho}_{f,\ell}^{\textup{ss}}$ was only well-defined up to equivalence (in $\GL_4(k')$) anyway, so Lemma \ref{lem:Serrealtisthere} still applies and the underlying space $V_k = k^4$ of $\overline{\rho}_{f,\ell}^{\textup{ss}}$ has the property that its extension $V_{k'} = (k')^4$ to $k'$ carries an alternating form with $k$-valued similitude character $\overline{\chi}_\ell^{2k-3}$.  The set of such alternating forms with fixed similitude character is defined by linear conditions over $k$ since the image of $\overline{\rho}_{f,\ell}$ belongs to $\GL_4(k)$; therefore, the existence of a form defined on $V_{k'}$ implies the existence of such a form on $V_k$ with the same similitude character.  Again up to equivalence, the image of $\overline{\rho}_{f,\ell}^{\textup{ss}}$ may be taken to lie in $\GSp_4(k)$.

For statement (b), by a theorem of Carayol \cite[Th\'eor\`eme 2]{Carayol} under the hypothesis that the \emph{residual representation is absolutely irreducible}, the representation $\rho_{f,\ell}$ takes values in $\GL_4(E)$.  Again we have a nondegenerate alternating form compatible with $\Gal_{\Q,S}$, and repeating the first part of the argument in the previous paragraph we may assume it takes values in $E$; conjugating, we conclude that the image is in $\GSp_4(E)$.
\end{proof}

\begin{remark}
The statement of Theorem \ref{thm:galoisrep} is not the most general statement that could be proven (in several respects), but it is sufficient for our purposes.

  Berger--Klosin \cite[Theorem 8.2]{BKPSY} attach to any paramodular newform $f$ a Galois representation into $\GL_4(\Q_\ell^{\text{al}})$, not just those of type~{\bf (G)}.
The remaining types are related to constructions
  of automorphic representations from those in $\GL_2(\A)$, where
  the local Langlands correspondence is known.  We do not know a reference for a complete argument for these remaining cases.  In this article, we are only concerned with forms of type~{\bf (G)}.  
\end{remark}

A consequence of Mok's proof of Theorem
\ref{thm:galoisrep}(v) is encoded in the following result.

\begin{lemma} \label{lem:fixedfieldordp}
Let $K$ be the fixed field of $\ker \overline{\rho}_{f,\ell}$ and let $\cond(\overline{\rho}_{f,\ell})$ be the Artin conductor of the representation $\overline{\rho}_{f,\ell}$ of $\Gal(K\,|\,\Q)$.  If $p \parallel N$ is odd, then $\ord_p(\cond(\overline{\rho}_{f,\ell})) \leq 1$.
\end{lemma}

\begin{proof}
The proof of Theorem \ref{thm:galoisrep}(v) is only up to semisimplification, so we do not know the complete statement of local Langlands under the patching argument that is employed.  However, in specializing the family to the accumulation point $f$ in the family, there is nevertheless an \emph{upper bound} on the level: the representation is necessarily either unramified or is Steinberg with level $p$, and accordingly the conductor has $p$-valuation $0$ or $1$.  
\end{proof}

\section{Group theory and Galois theory for \texorpdfstring{$\GSp_4(\F_2)$}{GSp4(F2)}} \label{sec:GSp4}

In this section, we carry out the needed Galois theory for the group $\GSp_4(\F_2)$.  Specifically, we carry out the task outlined in section \ref{sec:appfaltserre}: given $G = \img \overline{\rho} \leq \GSp_4(\F_2)$, and for each obstructing extension $\varphi$ extending $\overline{\rho}$, we compute an exact core-free subgroup $D \leq E$ (as large as possible) and the list of $E$-conjugacy classes of elements whose upper trace is nonzero.  The arguments provided in this section are done once and for all for the group $\GSp_4(\F_2)$; we apply these to our examples in section \ref{sec:verifypara}.

Since any similitude factor from $\GSp_4(\F_2)$ belongs to $\F_2^\times$ and is therefore trivial, we have an equality $\Sp_4(\F_2)=\GSp_4(\F_2)$.  Rather than writing one or the other throughout, we use the notation that seems natural to us in the given context.

\subsection{Symplectic group as permutation group} \label{sec:symperm}

We pause for some basic group theory.  We have an isomorphism $\iota\colon S_6 \xrightarrow{\sim} \Sp_4(\F_2)$, where $S_6$ is the symmetric group on $6$ letters, which we make explicit in the following manner.  Let $U \colonequals \F_2^6$, and equip $U$ with the coordinate action of $S_6$ and the standard nondegenerate alternating (equivalently, symmetric) bilinear form 
$\langle x,y \rangle=\sum_{i=1}^6 x_i y_i$ visibly compatible with the $S_6$-action.  Let $U^0 \subset U$ be the trace $0$ hyperplane, let $L$ be the $\F_2$-span of $(1,\dots,1)$, and let $Z \colonequals U^0/L$ be the quotient, so $\dim_{\F_2} Z = 4$.  Then $Z$ inherits both an action of $S_6$ and a symplectic pairing, which remains nondegenerate: specifically, the images
\[ e_1 \colonequals (1,1,0,0,0,0),\ e_2 \colonequals (0,0,1,1,0,0),\ e_3 \colonequals (0,0,0,1,1,0),\ e_4 \colonequals (0,1,0,0,0,1) \in Z \]
are a basis for $Z$ in which the Gram matrix of the induced pairing is the anti-identity matrix, so e.g.\ $\langle e_1,e_4 \rangle = \langle e_2,e_3 \rangle = 1$.  (An alternating pairing over $\F_2$ is symmetric, and we have chosen the standard such form.)  We compute that
\begin{equation} \label{eqn:S6GSp2}
\begin{aligned}
\iota\colon S_6 &\to \Sp_4(\F_2) \\
(1\ 2\ 3\ 4\ 5),(1\ 6) &\mapsto \begin{pmatrix}1&0&0&1\\0&0&1&0\\1&1&0&1\\1&0&1&0 \end{pmatrix}, \begin{pmatrix}0&0&0&1\\0&1&0&0\\0&0&1&0\\1&0&0&0 \end{pmatrix}.
\end{aligned}
\end{equation}

We have
\begin{equation} \label{eqn:VLie0}
\Lie^0(\GSp_4)(\F_2)=\Liesp_4(\F_2)=\{A \in \M_4(\F_2) : A^{\transpose}J + JA = 0\} \simeq \F_2^{10} 
\end{equation}
where $J \in \M_4(\F_2)$ is the anti-identity matrix (with $1$ along the anti-diagonal), and we have an exact sequence
\begin{equation} 
1 \to \Liesp_4(\F_2) \to \Liesp_4(\F_2) \rtimes G \xrightarrow{\pi} G \to 1
\end{equation}
with $\pi\colon \Liesp_4(\F_2) \rtimes G \to G$ the natural projection map.  As in \eqref{eqn:GL2n} we identify
\begin{equation} \label{eqn:VGGL2}
\begin{aligned} 
\Liesp_4(\F_2) \rtimes G \leq \M_4(\F_2) \rtimes G &\hookrightarrow \GL_{8}(\F_2) \\
(a,g) &\mapsto \begin{pmatrix} 1 & a \\ 0 & 1 \end{pmatrix} \begin{pmatrix} g & 0 \\ 0 & g \end{pmatrix} = \begin{pmatrix} g & ag \\ 0 & g \end{pmatrix}.
\end{aligned}
\end{equation}

The following lemmas follow from straightforward computation.

\begin{lem}
The group $\Sp_4(\F_2)$ has elements of orders $1,\dots,6$ with the following possibilities for their characteristic polynomials:
\begin{equation}
\begin{array}{c|c}
\textup{Order} & \textup{Characteristic polynomial} \\
\hline	\hline
\rule{0pt}{2.5ex} 1,2,4 & x^4+1 \\
3,6 & x^4+x^2+1 \textup{ or } x^4+x^3+x+1 \\
5 & x^4+x^3+x^2+x+1 \\
\end{array}
\end{equation}
\end{lem}

There is a unique outer automorphism of $S_6$ up to inner automorphisms  \cite{HMSV}; it sends transpositions to products of three transpositions, and interchanges the trace of some order $3$ and order $6$ elements.

\begin{lem}
There are, up to inner automorphism, exactly $9$ subgroups of $\Sp_4(\F_2) \simeq S_6$ with absolutely irreducible image. They are listed in the following table with a property that determines them uniquely (where `$-$' indicates there is a unique conjugacy class of subgroup with that order):  

\begin{equation} \label{table:subgroups}
\begin{array}{c|c|cc}
\textup{Subgroup} & \textup{Order}	&  \textup{Element orders} & \textup{Distinguishing property} \\
\hline	\hline
\rule{0pt}{2.5ex}   S_6 & 720 & 1,\dots, 6 & - \\
A_6 & 360 & 1,\dots, 5 & - \\
S_5(a) & 120 & 1,\dots, 6 & \textup{Elements of order $3,6$ have trace $0$} \\
S_5(b) & 120 & 1,\dots, 6 & \textup{Elements of order $3,6$ have trace $1$} \\
S_3 \wr S_2 & 72 & 1,2,3,4,6 & - \\
A_5(b) & 60 & 1,2,3,5 & \textup{Elements of order $3$ have trace $1$} \\
C_3^2 \rtimes C_4 & 36 & 1,2,3,4 & \textup{No elements of order $6$} \\
S_3(a)^2 & 36 & 1,2,3,6 & \textup{Elements of order $6$ have trace $0$} \\
C_5 \rtimes C_4 & 20 & 1,2,4,5 & -
\end{array}
\end{equation}
\end{lem}

\begin{example}
The conjugacy classes of subgroups $S_5(a),S_5(b) \leq S_6$ are exchanged by the outer automorphism of $S_6$.  For example, under the restriction of \eqref{eqn:S6GSp2}, we have 
\begin{equation} \label{eqn:S5GSp2}
\begin{aligned}
\iota\colon S_5(b) &\to \Sp_4(\F_2) \\
(1\ 2\ 3\ 4\ 5),(1\ 2),(1\ 2\ 3) &\mapsto \begin{psmallmatrix}1&0&0&1\\0&0&1&0\\1&1&0&1\\1&0&1&0 \end{psmallmatrix}, \begin{psmallmatrix}1&0&0&1\\ 0&1&0&0\\0&0&1&0\\0&0&0&1 \end{psmallmatrix},\begin{psmallmatrix}1&0&0&1\\0&1&0&0\\1&1&1&1\\1&1&0&0\end{psmallmatrix}.
\end{aligned} 
\end{equation}
Another way to distinguish $S_5(a)$ from $S_5(b)$ is that $\iota(S_5(b))$ has transvections while $\iota(S_5(a))$ does not.   
\end{example}

\begin{example}
There is a subgroup $A_5(a) \leq S_6$ that is similarly exchanged with $A_5(b)$ but that is not absolutely irreducible.  
\end{example}

\subsection{Images and discriminants}

For the purposes of establishing the first typical cases of the paramodular conjecture, we observe the following.

\begin{lem} \label{lem:s5s6s3s2}
Suppose $N$ is odd and squarefree and let $A$ be an abelian surface over $\Q$ of conductor $N$ equipped with a polarization of odd degree.  Then the residual representation 
\[ \overline{\rho}_{A,2}\colon \Gal_{\Q,S} \to \GSp_4(\F_2) \] 
(where $S=\{p: p \mid N\} \cup \{\ell,\infty\}$) is absolutely irreducible if and only if its image is isomorphic to $S_5(b)$, $S_6$, or $S_3 \wr S_2$. 
\end{lem}

\begin{proof}
By work of Brumer--Kramer \cite[\S 7.3]{BK14}, whenever $N$ is not a square, the image is either $S_5$, $S_6$, or $S_3 \wr S_2$.  To force $S_5(b)$, it suffices that there is a prime $p \mid N$ such that $A_p$ has toroidal dimension one (i.e., $p \parallel N$) and that $p$ be ramified in $\Q(A[2])$.  If $A$ is semistable and the Galois group is $S_5(a)$, then the toroidal dimension at the bad primes is $2$ since there are no transvections.
\end{proof}

\begin{rmk}
In general, if $A[2]$ is absolutely irreducible, then the degree of any \emph{minimal} polarization on $A$ is odd.
\end{rmk}

Next, we convert the upper bound from Lemma~\ref{lem:fixedfieldordp} on the conductor into an upper bound on the discriminant.  We first recall the following standard result.

\begin{lemma} \label{lem:aQx}
  Let $a(x) \in \Q[x]$ be irreducible and let $\Omega$ be the set of roots of $a(x)$ in $\Qbar$.  Let $\alpha \in \Omega$, let $K_0 = \Q(\alpha)$, and let $K$ be the normal closure of $K_0$.  Let $\frakp$ be a prime of $K$ that is tamely ramified in the extension $K \supseteq \Q$, and let $p \in \Z$ be the prime lying below $\frakp$.  Finally, let $I_\frakp \leq \Gal(K \,|\, \Q)$ denote the inertia group at $\frakp$.  Then 
  \[ \ord_p(d_{K_0}) = \deg a(x) - \#\Omega/I_\frakp \] 
  where $\#\Omega/I_\frakp$ denotes the number of orbits of $I_\frakp$ acting on $\Omega$.  
\end{lemma}

We now specialize to our case of interest.
\begin{prop} \label{prop:discbound} Let $p \parallel N$ be odd.  Let
  $K$ be the fixed field of $\ker \overline{\rho}_{f,2}$.  
  \begin{enumalph}
  \item If
  $\Gal(K\,|\,\Q) \simeq S_3 \wr S_2$ (resp., $S_m$ with $m=5,6$), then
  $K$ is the normal closure of a field $K_0$ of degree $6$ (respectively, $m$)
  with $\ord_p d_{K_0} \leq 1$. 
  \item If
  $\Gal(K\,|\,\Q) \simeq A_m$, with $m=5,6$, then 
  $K$ is the normal closure of a field $K_0$ of degree $m$
  with $p$ unramified in $K_0$ (i.e., $\ord_p d_{K_0}=0$).
  \end{enumalph}
\end{prop}

\begin{proof}
Decomposing the Weil--Deligne representation at $p$, we see by Lemma~\ref{lem:fixedfieldordp} that the image of inertia is either trivial or a $2 \times 2$-Jordan block.  If trivial, the extension is unramified and the result holds, so suppose we are in the latter case.
 Under the isomorphism $\GSp_4(\F_2) \simeq S_6$ above
\eqref{eqn:S6GSp2}, nontrivial elements of this Jordan block
correspond to cycle decomposition $2+2+2$ or $2+1+1+1+1$, and these
are exchanged by an outer automorphism. 

For (a), by a faithful permutation representation on the cosets of a core-free subgroup, a field $K_0$ of the given degree exists. If the residual image inside $S_6$ is invariant under such an automorphism (which holds for $S_6$ and $S_3 \wr S_2$), then we can choose our
subfield $K_0$ corresponding to the latter case, and conclude
$\ord_p d_{K_0} \leq 1$ by Lemma \ref{lem:aQx}. 
If $\Gal(K\,|\,\Q) \simeq S_5$, we have only the possibility $2+1+1+1$
again giving $\ord_p d_{K_0} \leq 1$.  

Finally, for (b) and the groups $A_5,A_6$, we find no possibilities and reach a contradiction, so we conclude that $K_0$ is unramified at $p$.
\end{proof}

\subsection{Core-free extensions and obstructing elements} \label{sec:corefreethings} 

We will compute all obstructing extensions $\varphi\colon \Gal(L \,| F) \hookrightarrow E$ extending $\overline{\rho}$ (Definition \ref{defn:extendingrho}); we represent $L \supseteq K \supseteq F$ by an exact core-free subextension $L_0 \supseteq K_0 \supseteq F$ (Definition \ref{def:corefreesubexte}) arising from an exact core-free subgroup $D \leq E$ which is as large as possible, to make the degree of the subextension as small as possible.  

For each $G$ in \eqref{table:subgroups}, we therefore first seek subgroups $\varphi\colon E \hookrightarrow \Liesp_4(\F_2) \rtimes G$ such that $\pi(E)=G$; such extensions are obstructing (Definition \ref{defn:obstructing}) if they have nonzero upper trace in the matrix realization \eqref{eqn:VGGL2}.

Consider first the case $G=S_5(b)$. 

\begin{thm} \label{thm:GcoreS5b}
For $G=S_5(b)$, there are exactly $10$ extension groups $E$ up to conjugacy in $\M_4(\F_2) \rtimes G$, with $\#V=[E:G]=2^k$ where $k=0,0,1,4,4,5,5,6,9,10$, respectively.

Furthermore, let
\[ H = D_6(b) \colonequals \langle (1\ 2),(1\ 3),(4\ 5) \rangle \leq
  G; \] then for all $E \not\simeq G$, there is an exact core-free subgroup $D \leq E$ of index $2$ such that $\pi(D)=H$ as in \eqref{eqn:WDHres}.
\end{thm}

\begin{proof}
This theorem is proven by explicit computation in \textsf{Magma} \cite{Magma}; the code is available online \cite{coderepo} together with the verbose output.  There are exactly $18$ conjugacy classes of subgroups
  $\varphi \colon E \hookrightarrow \Liesp_4(\F_2) \rtimes G$ with $\pi(E)=G$;
  these subgroups fall into $10$ conjugacy classes in
  $\M_4(\F_2) \rtimes G$.
Let $H = D_6(b) \colonequals \langle (1\ 2),(1\ 3),(4\ 5) \rangle \leq G$ be as in the statement. Then $H$ is dihedral of order $\#H=12$ and index $[G:H]=10$ and it can be verified that for each such
$E \not\simeq G$, there is at least one subgroup $W \leq V$ of index
$2$ such that $D \leq E$ is an exact core-free
subgroup.
\end{proof}

The somewhat complicated field diagram \eqref{eqn:fielddiagram} in our case simplifies to:
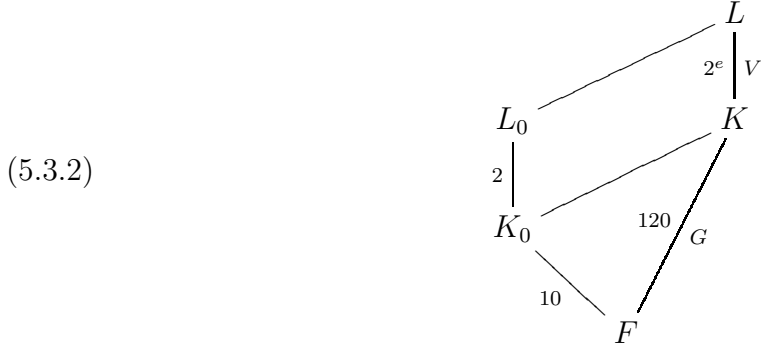
\begin{equation} \label{eqn:LL0S5b}
\begin{minipage}{\textwidth}
\xymatrix{
& & L \ar@{-}[d]_{2^k}^{V} \\
L_0 \ar@{-}[d]_{2} \ar@{-}[urr] & & K \ar@{-}[dll] \ar@{-}@/^0pc/[ddl]^{G}_{120} \\
K_0 \ar@{-}[dr]_{10} &  \\
& F
} 
\end{minipage}
\end{equation}
We understand the large extension $L \supseteq K \supseteq F$ as the Galois closure of the exact core-free subextension $L_0 \supseteq K_0 \supseteq F$, with $L_0 \supseteq K_0$ quadratic.  The extension $K_0$ is realized explicitly as follows: if $K \supseteq F$ is the splitting field of a quintic polynomial $f(x)$ with roots $\alpha_1,\dots,\alpha_5$ permuted by $S_5$, then $K_0=K^H=F(\alpha_4+\alpha_5)$.

In a similar way, we have the result for the remaining two groups.  

\begin{thm} \label{thm:Gcore} 
\ 
\begin{enumalph}
\item For $G=S_3 \wr S_2 \leq \GSp_4(\F_2)$, there are exactly $20$ extension groups $E$ up to conjugacy in $\M_4(\F_2) \rtimes G$, with $\#V=[E:G]=2^k$ and
\[ k = 0,0,1,1,2,4,4,4,4,5,5,5,5,6,6,8,8,9,9,10. \]
Let $H=C_2^2 \leq G$ with $[G:H]=18$.  Then for each such $E$, there is an exact core-free subgroup $D \leq E$ such that $\pi(D)=H$. 
\item For $G=S_6 \simeq \GSp_4(\F_2)$, the analogous statement to \textup{(a)} holds, with $7$ groups having $k=0,0,1,5,5,6,10$ and $H=S_3(b)^2$.  
\end{enumalph}
\end{thm}

\begin{remark}
With reference to computing conjugacy classes in stages as in section \ref{sec:instages}, we note that the index 2 subgroups of the $18$ subgroups $C_2^2$ of $S_3\wr S_2$ are not sufficient to find obstructing classes for all $20$ extension groups if one applies the more limited strategy exhibited in Remark \ref{rmk:restricttoallxi}.
\end{remark}

\begin{remark}
The remaining cases of subgroups $G \leq \GSp_4(\F_2)$ may be computed with the same method and the same code.
\end{remark}

\section{Computing Hecke eigenvalues by specialization}  \label{sec:computehecke}

Having set up the required Galois theory, we now compute Hecke eigenvalues of 
particular Siegel paramodular newforms.  
In this section, we use the technique of restriction to a modular curve to accomplish these eigenvalue computations.  We continue the notation from section \ref{sec:paramodularforms}.

\subsection{Jacobi forms and Borcherds products}

We construct our paramodular forms using Gritsenko lifts of Jacobi forms and Borcherds products.  In this section, we quickly review what we need from these theories. 

We begin with Jacobi forms; we refer to Eichler--Zagier \cite{EZ} for further reference.  
Each Jacobi form $\phi \in J_{k,N}$ of weight $k$ and index $N$ has a Fourier expansion
\begin{equation} \label{eqn:fourierJacobi}
\phi(\tau,z) = \sum_{\substack{n, r \in \Z}} c(n,r;\phi) q^n \zeta^r, 
\end{equation}
where $q=e(\tau)$ and $\zeta=e(z)$.  
We write  $\phi \in J_{k,N}(R)$ if all the Fourier coefficients of~$\phi$ 
lie in a ring $R \subseteq \C$.  We will need the level-raising operators 
$V_{m}\colon J_{k,N} \to J_{k,mN}$ (see Eichler--Zagier \cite[p.\ 41]{EZ}) that act on 
$\phi \in J_{k,N}$ via 
\begin{equation} 
\label{eqn:GritFormula}
c(n,r; \phi\,|\, V_{m}) = \sum_{\delta \mid \gcd(n,r,m)} \delta^{k-1} c\left(\frac{mn}{\delta^2},\frac{r}{\delta};\phi\right) .
\end{equation}

The Gritsenko lift \cite{Grits}
\[ \Grit \colon J_{k,N}\cusp \to S_k( K(N)) \]
lifts a Jacobi cusp form~$\phi$ to a paramodular form~$f$ by the rule 
\[a\!\left(  \begin{psmallmatrix} n  &  r/2  \\  r/2  &  Nm \end{psmallmatrix}; \Grit(\phi) \right) = c(n,r; \phi\,|\,V_m).  \]
We also have $\Grit(\phi) |_k\, \mu_N = (-1)^k \Grit(\phi)$, 
so that a Gritsenko lift has paramodular Fricke sign~$(-1)^k$.  

One convenient way to construct Jacobi forms is to use the theta blocks 
created by Gritsenko--Skoruppa--Zagier  \cite{GSZ}.  
Recall the  Dedekind $\eta$-function and the Jacobi $\vartheta$-function 
\begin{align*} 
\eta(\tau) &=q^{1/24} \prod_{n=1}^{\infty} (1-q^n) = \sum_{n=1}^{\infty} \legen{12}{n} q^{n^2/24}, \\
\vartheta(\tau,z) &= \sum_{n=-\infty}^{\infty} (-1)^n q^{(2n+1)^2/8} \zeta^{(2n+1)/2} 
= q^{1/8} (\zeta^{1/2}-\zeta^{-1/2}) \sum_{n=1}^{\infty} (-1)^{n+1} q^{\binom{n}{2}} \sum_{j=-(n-1)}^{n-1} \zeta^j .
\end{align*}
   For $d \in \Z_{>0}$ let $\vartheta_d(\tau,z) = \vartheta(\tau,dz)$.  For $d_1,\dots,d_\ell \in \Z_{>0}$ 
and $ k \in \Z$,    
   define the \defi{theta block}
\begin{equation}
\TB_k[{\mathbf d}]=
\TB_k[d_1, d_2, \ldots, d_\ell] = \eta^{2k} 
\prod_{j=1}^{\ell} \dfrac{\vartheta_{d_j}}{\eta}.
\end{equation}
The theta block $\TB_k[{\mathbf d}]$ defines a meromorphic Jacobi form (with multiplier) of weight $k$ 
and  index $m = \frac12( d_1^2 + \cdots + d_{\ell}^2)$.  Moreover \cite{EZ} (compare Poor--Yuen \cite[Theorem 4.3]{PY15}), the theta block $\TB_k[{\mathbf d}]$ is a Jacobi cusp form if 
\begin{equation} 
12 \mid (k+\ell)\quad \text{and} \quad \frac{k}{12} + \frac12 \sum_{j=1}^{\ell} {\bar B}_2(d_j x) >0,
\end{equation}
where $B_2(x) \colonequals x^2-x+\frac16$ and ${\bar B}_2(x) \colonequals B_2(x- \lfloor x \rfloor)$.

Second, we use Borcherds products in the construction of paramodular forms. 
Let $\psi$ be a weakly holomorphic Jacobi form of weight $0$ and index $N$ with integral Fourier coefficients on singular indices with Fourier expansion \eqref{eqn:fourierJacobi}.  
Define 
\begin{equation*}
A(\psi) \colonequals \frac1{24}\sum_{r\in\Z}c(0,r;\psi),\quad 
B(\psi) \colonequals \frac12\sum_{r\ge1}r c(0,r;\psi),\quad 
C(\psi) \colonequals \frac14\sum_{r\in\Z}r^2 c(0,r;\psi).  
\end{equation*}
Then $A(\psi),B(\psi),C(\psi) \in \Q$.  The Borcherds product of $\psi$ is a meromorphic paramodular form $\BP(\psi)$, perhaps with nontrivial character on~$K(N)$, with 
\begin{equation}\label{eqn:BPinfiniteproduct}
\BP(\psi)=q^{A(\psi)}\zeta^{B(\psi)}\xi^{C(\psi)}\prod_{n,r,m}(1-q^n\zeta^r\xi^{mN})^{c(mn,r;\psi)},
\end{equation}
where the product is over $n,r,m \in \Z$ such that: (i) $m \geq 0$; (ii) if $m=0$, then $n \geq 0$; and (iii) if $m=n=0$, then $r<0$.
Borcherds products are not always holomorphic and, when holomorphic, not always cuspidal.  

\subsection{Construction of newforms}

In this section, we define the nonlift paramodular newforms of interest to this article, with levels $277,353,587$.  
We will see later that this way of writing paramodular forms makes the computation of Hecke eigenvalues feasible.   

We refer to section \ref{sec:paramodularforms} for notation.  We now define the nonlift paramodular form $f_{277} \in S_2( K(277), \Z)^{+}$ following Poor--Yuen \cite[Theorem 7.1]{PY15}.  Define the following ten theta blocks:
\begin{equation}
\begin{aligned}
\Xi_1&\colonequals \TB_2(2, 4, 4, 4, 5, 6, 8, 9, 10, 14) \quad & \Xi_6&\colonequals \TB_2(2, 3, 3, 5, 5, 7, 8, 10, 10, 13)\\
\Xi_2&\colonequals \TB_2(2, 3, 4, 5, 5, 7, 7, 9, 10, 14) & \Xi_7&\colonequals \TB_2(2, 3, 3, 4, 5, 6, 7, 9, 10, 15)\\
   \Xi_3&\colonequals \TB_2(2, 3, 4, 4, 5, 7, 8, 9, 11, 13) & \Xi_8&\colonequals \TB_2(2, 2, 4, 5, 6, 7, 7, 9, 11, 13)\\
 \Xi_4&\colonequals \TB_2(2, 3, 3, 5, 6, 6, 8, 9, 11, 13) & \Xi_9&\colonequals \TB_2(2, 2, 4, 4, 6, 7, 8, 10, 11, 12)\\
 \Xi_5&\colonequals \TB_2(2, 3, 3, 5, 5, 8, 8, 8, 11, 13) & \Xi_{10}&\colonequals \TB_2( 2, 2, 3, 5, 6, 7, 9, 9, 11, 12).\\
\end{aligned}
\end{equation}  
We have, for $i=1,\dots,10$, 
\begin{center}
$\Xi_i \in J_{2,277}\cusp(\Z)$ \quad and \quad $G_i \colonequals \Grit(\Xi_i) \in S_2( K(277), \Z)$.
\end{center}
Let $f_{277}$ be the (a priori) meromorphic function on~$\Half_2$ defined by
\begin{equation} \label{eqn:fasQL277}
\begin{aligned}
f_{277} \colonequals \, &( -14G_1^2-20G_8G_2+11G_9G_2+6G_2^2-30G_7G_{10}+15G_9 G_{10}\\
&\quad +15 G_{10} G_1
-30 G_{10} G_2   
-30 G_{10} G_3+5 G_4 G_5+6 G_4 G_6+17 G_4 G_7\\
&\quad -3 G_4 G_8-5 G_4 G_9
-5 G_5 G_6 +20 G_5 G_7   
-5 G_5 G_8-{10} G_5 G_9-3 G_6^2\\
&\quad +13 G_6 G_7+3 G_6 G_8
-{10} G_6 G_9-22 G_7^2+G_7 G_8+15 G_7 G_9  
+6 G_8^2\\
&\quad -4 G_8 G_9-2 G_9^2+20 G_1 G_2
-28 G_3 G_2+23 G_4 G_2+7 G_6 G_2\\
&\quad -31 G_7 G_2+15 G_5 G_2 
+45 G_1 G_3-{10} G_1 G_5
-2 G_1 G_4-13 G_1 G_6\\
&\quad -7 G_1 G_8+39 G_1 G_7-16 G_1 G_9-34 G_3^2  
+8 G_3 G_4
+20 G_3 G_5\\
&\quad +22 G_3 G_6+{10} G_3 G_8+21 G_3 G_9-56 G_3 G_7-3 G_4^2)/ \\
&      (-G_4 + G_6 + 2 G_7 + G_8 - G_9 + 2 G_3 - 3 G_2 - G_1).  
\end{aligned}
\end{equation}
A main result of Poor--Yuen \cite[Theorem 7.1]{PY15} is that 
$f_{277}$ is actually \emph{holomorphic}: in fact, $f_{277} \in S_2( K(277), \Z)^+$ is a cuspidal, nonlift, 
paramodular form of weight $2$ that is an eigenform for all Hecke operators and has integral Fourier coefficients whose greatest common divisor is~$1$.  There are no nontrivial weight $2$ paramodular cusp forms of level $1$, so since $277$ is prime, $f_{277}$ is a newform.
Equation~\eqref{eqn:heckeopersFC} and Lemma~\ref{lem:k2fint} imply that the Euler factors~$Q_p(f_{277},t)$ 
are integral.  
 
 The first few  eigenvalues for $f_{277}$ were computed \cite{PY15} as 
\begin{equation} \label{eqn:firstfewap277}
a_p(f_{277})=-2,-1,-1,1,-2 \quad \text{for $p=2,3,5,7,11$}
\end{equation}
and the first three Hecke polynomials, identifying $f_{277}$ as type~{\bf (G)}, are: 
\begin{equation} \label{eqn:firstfewL277}
\begin{aligned}
Q_2(f_{277},t) &= 1+2t+4t^2+4t^3+4t^4, \\
Q_3(f_{277},t) &= 1+t+t^2+3t^3+9t^4, \\
Q_5(f_{277},t) &= 1+t-2t^2+5t^3+25t^4.
\end{aligned}
\end{equation}

\begin{remark}
The form $f_{277}$ can also be realized as the sum of a Borcherds product and a Gritsenko lift, giving a second, independent construction by Poor--Shurman--Yuen \cite{pyinflation}.  
\end{remark}

In a similar way, we construct a second form
\begin{equation} \label{eqn:f353Q}
f_{353} \colonequals Q(G_1,\dots,G_{11}) \in S_2( K(353), \Z)^{+}
\end{equation}
(plus eigenspace for the Fricke involution, as in \eqref{eqn:eigenspaceplusminus}) a quotient of a quadratic polynomial by a linear polynomial of~$11$ Gritsenko lifts of theta blocks: 
see Poor--Yuen \cite[Theorem 7.4]{PY15} for the specific formula for $Q$ and the forms $G_i$.  This construction 
was contingent upon assuming the existence of some nonlift in $S_2(K(353))$; 
however, the dimension $\dim S_2(K(353)) =12$ is now known 
\cite{pyinflation} via the construction of a nonlift Borcherds product in $S_2(K(353))$.

The first two Euler factors, each showing that $f_{353}$ is of type~{\bf (G)}, are
\begin{equation} \label{eqn:firstfewL353}
\begin{aligned}
Q_2(f_{353},t) &= 1+t+3t^2+2t^3+4t^4, \\
Q_3(f_{353},t) &= 1+2t+4t^2+6t^3+9t^4.
\end{aligned}
\end{equation}

Finally, we construct a form of level $587$ as a Borcherds product.  An antisymmetric nonlift Borcherds product 
$f_{587}^{-} \in S_2\left( K(587), \Z\right)^{-}$ 
was recently constructed by Gritsenko--Poor--Yuen \cite{gpy16}.  
The form $f_{587}^{-}$ is necessarily an eigenform because $\dim S_2\left( K(587)\right)^{-}=1$.  
The Fourier expansion is given by formally expanding 
\begin{equation}  \label{eqn:f587def}
f_{587}^{-} = \BP(\psi)=\xi^{587} \phi \exp( -\Grit(\psi)) \quad \text{ for } \quad \psi = (  \phi\,|\, V_2 - \Xi)/ \phi, 
\end{equation}
where 
\begin{equation} \label{eqn:phiXi587}
\begin{aligned}
\phi=\,&\,{\TB}_2(
1, \ \, 1, 2, 2, \ \, 2, 3, 3, 4, 4, \ \, 5, 5, 6, 6, 7, 8, \ \, 8, 9, 10, 11, 12, 13, 14
) \in J_{2,587}\cusp, \\
\Xi=\,&\,{\TB}_2(1, 10, 2, 2, 18, 3, 3, 4, 4, 15, 5, 6, 6, 7, 8, 16, 9, 10, 22, 12, 13, 14) \in J_{2, 1174}\cusp.  
\end{aligned}
\end{equation}
For the Borcherds product that appears in the formula for $f_{587}^-$, 
we have $\BP(\psi)\in S_k(K(587))$ with $k=\frac12c(0,0;\psi)=2$ \cite{gpy16}.  
The first two Euler factors, verifying type~{\bf (G)},  are computed to be
\begin{equation} \label{eqn:firstfewL587}
\begin{aligned}
Q_2(f_{587}^{-},t) &= 1+3t+5t^2+6t^3+4t^4, \\
Q_3(f_{587}^{-},t) &= 1+4t+9t^2+12t^3+9t^4.
\end{aligned}
\end{equation}

\subsection{Specialization}

To compute the action of the Hecke operators directly on a Fourier expansion of a Siegel paramodular form would require manipulations with series in three variables.  To avoid this, we specialize our form.  Possibilities for this specialization include restriction to Humbert surfaces (typically producing Hilbert modular forms), restriction to modular curves (producing classical modular forms), or evaluation at CM points (producing a numerical result, see Colman--Ghitza--Ryan \cite{CGR}).  Each of these methods has certain advantages and disadvantages---we choose to restrict to modular curves and work with one-variable $q$-series to avoid rigorous analysis of the upper bounds on the tails of convergent numerical series.
The biggest advantage of our choice, however,  is that Proposition~\ref{speedup1} 
allows us to sum over only $O(p^2)$ cosets  instead of $O(p^3)$ cosets, a 
significant savings; it is not clear whether such a speedup is available to a 
method that numerically evaluates at CM points.  

\begin{remark}
Specialization of Siegel modular forms is not a new idea, but here we take a different approach.  In previous work of Poor--Yuen \cite{PY15}, only three Euler factors were computed for $f_{277}$ because the computation relied on multiplying initial expansions of multivariable Fourier series.  Instead, below we will write the action of the Hecke operator~$T(p)$ on a paramodular form~$f$ as a sum of slashes $f\,|_k\, T(p) = \sum_j f\,|_k\, M_j$, and the main innovation is to specialize each part of $f\vert M_j$ to a one variable $q$-series prior to any addition, multiplication, or division.  Specialization was also used by Poor--Yuen \cite{PY07} to compute upper bounds on dimensions and some Fourier coefficients by taking advantage of the known structure of the target space of elliptic modular forms, whereas here we only use the one variable nature of the target space.
\end{remark}

Let $s \in \M_2^{\text{\rm sym}}(\Q)_{>0}$  be a symmetric, positive definite matrix with rational coefficients. 
Let $\SiegelH_g$ be the Siegel upper half space of dimension $g$, so $\SiegelH_1$ is the upper half-plane.  Define 
the holomorphic map
\begin{equation}
\begin{aligned}
\phi_s \colon \SiegelH_1 &\to \SiegelH_2 \\
 \tau &\mapsto s\tau .
\end{aligned}  
\end{equation}

\begin{lem}
Let $R \subseteq \C$ be a subring.  
Let $s=\begin{psmallmatrix} a  &  b  \\  b  &  c/N  \end{psmallmatrix} \in \M_2^{\text{\rm sym}}(\Q)_{>0}$  
with $a,b,c \in \Z$.  
Then the pullback under $\phi_s$ defines a ring homomorphism 
\begin{equation}
\phi_s^*: M(K(N),R)  \to M(\Gamma_0( \det(s) N),R)
\end{equation}
from the graded ring of Siegel paramodular forms of level $N$ with coefficients in
 $R$ to the graded ring of classical modular forms of level $\det(s)N$ with coefficients in $R$.  The map 
 $\phi_s^*$ multiplies weights by $2$ and maps cusp forms to cusp forms.
\end{lem}

\begin{proof}
The proof follows from a straightforward modification of a result of Poor--Yuen \cite[Proposition 5.4]{PY07}. 
\end{proof}

Let $f \in M_k(K(N),R)$ be a paramodular form with Fourier expansion \eqref{eqn:Fourierexppara}, the Fourier expansion of the specialization 
$\phi_s^{*}f \in  M_{2k}(\Gamma_0( \det(s) N),R)$ is
\begin{equation} 
(\phi_s^*f)(\tau) = f(s\tau) = \sum_{n=0}^{\infty} \biggl(\sum_{\substack{\,T: \,\Tr(sT)=n}} a(T;f)\biggr) q^n. 
\end{equation}
Furthermore, the specialization of ~$f$ after slashing with 
a block upper-triangular matrix 
$\begin{psmallmatrix} A  &  B  \\  0  &  D \end{psmallmatrix} \in \GSp_4^{+}(\Q)$ 
with similitude $\mu = \det(AD)^{1/2}$ 
  is given by
\begin{equation} \label{eqn:phisslash}
\begin{aligned}
&\phi_s^*(f\,|_k \begin{psmallmatrix} A & B \\ 0 & D \end{psmallmatrix})(\tau) =(f\,|_k \begin{psmallmatrix} A & B \\ 0 & D \end{psmallmatrix})(s \tau)
 = \det(AD)^{k-3/2}    \det(D)^{-k}    f(AsD^{-1}\tau + BD^{-1}) \\
&\qquad= \det(A)^{k}\det(AD)^{-3/2}  \sum_{n \in \Q_{\geq 0}} \biggl(\sum_{\substack{T:\,\Tr(AsD^{-1}T)=n}} e\left(\Tr(BD^{-1}T)\right)\, a(T;f)  \biggr) q^n.
\end{aligned}
\end{equation}

Let $s=\left(
\begin{smallmatrix} 
a  &  b    \\
b  &  c/N  
\end{smallmatrix} 
\right) \in \M_2^{\text{\rm sym}}(\Q)_{>0}$ with $a,b,c \in \Z$.  Using \eqref{eqn:heckeopersM},
the specialization of $f\,|_k\, T(p)$ may be written 
\begin{equation} \label{eqn:restriction1}
\begin{aligned}
\phi^{*}_s(f\,|_k\,T(p))(\tau) &= p^{2k-3} f(ps\tau)
\\ &\quad 
+ p^{k-3}\sum_{i \bmod{p}}
f\Bigl(
\begin{psmallmatrix} 
a/p  &  b    \\
b  &  pc/N  
\end{psmallmatrix} 
\tau+
\begin{psmallmatrix} 
i/p  &  0    \\
0  &  0  
\end{psmallmatrix} 
\Bigr)
\\
& \quad + p^{k-3}
\sum_{i\bmod p}
\left(
\sum_{j \bmod{p}} 
f\Bigl(
\begin{psmallmatrix} 
pa &  b+ia    \\
b+ia  &  (c/N+2ib+i^2a)/p  
\end{psmallmatrix} 
\tau+
\begin{psmallmatrix} 
0  &  0    \\
0  &  j/p  
\end{psmallmatrix} 
\Bigr)\right)
\\
&\quad
 + p^{-3}\sum_{i,j,k \bmod{p}}  
 f\Bigl(
s
\tau/p+
\begin{psmallmatrix} 
i/p  &  j/p    \\
j/p  &  k/p  
\end{psmallmatrix}
\Bigr).
\end{aligned}
\end{equation}
Upon expanding in Puiseux $q$-series, 
there is cancellation among these sums of specializations. The following proposition shows that 
partial summation gives new specializations whose sum {\em over smaller index sets\/}  
equals the original sum 
for integral powers of~$q$.  
For a Puiseux series $f \in \C[[q^{1/\infty}]]$ and $e \in \Q_{\geq 0}$, we denote by $\Coeff_e f \in \C$ the coefficient of $q^e$ in $f$.

\begin{prop}\label{speedup1}
Let $s=\begin{psmallmatrix} a  &  b  \\  b  &  c/N  \end{psmallmatrix} \in \M_2^{\text{\rm sym}}(\Q)_{>0}$  
with $a,b,c \in \Z$.  Let $p$ be prime, and let $f \in M_k(K(N))$.  Then the following statements hold for all $e \in \Z_{\geq 0}$.
\begin{enumalph}
\item If $p \nmid a$, then
\begin{equation*}
\begin{aligned}
\Coeff_e \sum_{i \bmod{p}}
f\Bigl(
\begin{psmallmatrix} 
a/p  &  b    \\
b  &  pc/N  
\end{psmallmatrix} 
\tau+
\begin{psmallmatrix} 
i/p  &  0    \\
0  &  0  
\end{psmallmatrix} 
\Bigr) &= p \Coeff_e f\Bigl(
\begin{psmallmatrix} 
a/p  &  b    \\
b  &  pc/N  
\end{psmallmatrix} 
\tau\Bigr) 
\\
\Coeff_e \sum_{i,j,k \bmod{p}}  
 f\Bigl(
s
\tau/p+
\begin{psmallmatrix} 
i/p  &  j/p    \\
j/p  &  k/p  
\end{psmallmatrix}
\Bigr) 
&= 
p\Coeff_e \sum_{j,k \bmod{p}}  
 f\Bigl(
s
\tau/p+
\begin{psmallmatrix} 
0  &  j/p    \\
j/p  &  k/p  
\end{psmallmatrix}
\Bigr).
\end{aligned}
\end{equation*}

\item If $p \nmid b$, then
\begin{equation*}
\Coeff_e \sum_{i,j,k \bmod{p}}  
 f\Bigl(
s
\tau/p+
\begin{psmallmatrix} 
i/p  &  j/p    \\
j/p  &  k/p  
\end{psmallmatrix}
\Bigr) 
= 
p\Coeff_e \sum_{i,k \bmod{p}}  
 f\Bigl(
s
\tau/p+
\begin{psmallmatrix} 
i/p  &  0    \\
0  &  k/p  
\end{psmallmatrix}
\Bigr).
\end{equation*}
\item If $p \nmid c$, then
\begin{equation*}
\Coeff_e \sum_{i,j,k \bmod{p}}  
 f\Bigl(
s
\tau/p+
\begin{psmallmatrix} 
i/p  &  j/p    \\
j/p  &  k/p  
\end{psmallmatrix}
\Bigr) 
= 
p\Coeff_e \sum_{i,j \bmod{p}}  
 f\Bigl(
s
\tau/p+
\begin{psmallmatrix} 
i/p  &  j/p    \\
j/p  &  0  
\end{psmallmatrix}
\Bigr).
\end{equation*}
\item
For $i \in \Z$, if $p  \nmid  (c+2ibN+i^2aN)$, then
\begin{equation*}
\Coeff_e \sum_{j \bmod p} f\Bigl(
\begin{psmallmatrix} 
pa &  b+ia    \\
b+ia  &  (c/N+2ib+i^2a)/p  
\end{psmallmatrix} 
\tau+
\begin{psmallmatrix} 
0  &  0    \\
0  &  j/p  
\end{psmallmatrix} 
\Bigr)
=p\Coeff_e 
f\Bigl(
\begin{psmallmatrix} 
pa &  b+ia    \\
b+ia  &  (c/N+2ib+i^2a)/p  
\end{psmallmatrix} 
\tau
\Bigr).
\end{equation*}
\end{enumalph}
\end{prop}

\begin{proof}
We prove (c); the other proofs are similar.  Suppose $p \nmid c$.  Let $e \in {\Z}_{\ge 0}$.  Then the coefficient of $q^e$ in the left-hand side is equal to 
\begin{equation} \label{eqn:ijkpsum}
\sum_{\substack{i,j,k\bmod{p}\\
n,r,m:\, an+br+cm=pe}
}
e((in+jr+km)/p)a(\myT ;f) 
\end{equation}
where $\myT =\begin{psmallmatrix} 
n  &  r/2    \\
r/2  &  mN  
\end{psmallmatrix}$.
If any of $n,r,m$ is not a multiple of $p$,
then summing over $i,j,k$ modulo $p$ in \eqref{eqn:ijkpsum} would yield a contribution of zero.
Hence we may restrict the sum to the terms where $p \mid n$, $p \mid r$, and $p \mid m$.  
But since $p \nmid c$ and given $an+br+cm=pe$, the conditions $p \mid n$ and $p \mid r$ imply $p \mid m$.
Thus \eqref{eqn:ijkpsum} becomes simply
\begin{align*}
\sum_{\substack{i,j,k\bmod{p}\\
n,r,m:\, an+br+cm=pe \\
p\mid n,\, p \mid r}
}
e((in+jr+0)/p)a(\myT ;f)
& =p\sum_{\substack{i,j\bmod{p}\\
n,r,m:\, an+br+cm=pe \\
p \mid n,\, p \mid r}
}
e((in+jr)/p)a(\myT ;f) \\   
 =p\sum_{\substack{i,j\bmod{p}\\
n,r,m:\, an+br+cm=pe 
}
}
e((in+jr)/p)a(\myT ;f)
&=p \Coeff_e
\sum_{i,j \bmod{p}}  
 f\Bigl(
s
\tau/p+
\begin{psmallmatrix} 
i/p  &  j/p    \\
j/p  &  0  
\end{psmallmatrix} 
\Bigr).  \qedhere
\end{align*}
\end{proof}

\begin{remark}
Proposition \ref{speedup1} provides a certain subtle speedup because the coefficients at integral powers are equal, even though the series themselves are not necessarily equal. 
Further simplifying the above sums to
\begin{equation*}
p^3
\sum_{\substack{
n,r,m:\, an+br+cm=pe\\
p \mid n,p \mid r, p \mid m}
}
a(\myT ;f).
\end{equation*}
does not help: we want to leave the sums in terms of coefficients of specializations.
\end{remark}

In a similar way, we can compute the specialization $\phi^{*}_s(f\,|_k\,T_1(p^2))$ and there are similar cancellations in the character sums as in Proposition \ref{speedup1}.

\def\SSS{{\mathcal S}}

\subsection{Algorithmic detail}

In this section, we provide three further bits of algorithmic detail.

First, we describe the choice of $s$.  Suppose $f$ has a nonzero coefficient $a(t_0; f)$ where $t_0$ has small determinant and small entries.  
If we choose $s$ to be the adjoint of $2t_0$, 
then the restriction
$\phi_s^*(f)$ likely begins with $a(t_0;f)q^{\det(s)}$.
In particular if $t_0$ has minimal determinant, then this is forced.
In practice, we can just check the initial expansion to see that
\begin{equation*}
\phi_s^*(f)(\tau)=a(t_0;f)q^{\det(s)}+\text{higher powers of $q$}.
\end{equation*}
For each $T(p)$,
we want to expand
$\phi_s^*(f|T(p))$
to at least $q^e$ where $e=\det(s)$ is the target exponent of $q$.
For a polynomial combination of Gritsenko lifts and Borcherds products,
the target exponent of each part $g(\myA\tau+\myB)$ would also be $e$.
But for a rational function of Gritsenko lifts and Borcherds products, we have to be slightly more careful.
If the denominator of this rational functional restricted to $(\myA\tau+\myB)$ has leading
term~$q^\mu$,
then we must expand both the numerator and denominator to a higher target
term~$q^{e+\mu}$.
Therefore, we may end up evaluating the restriction of the denominator twice,
with the initial execution used to get the leading exponent~$\mu$.

Second, we provide our algorithm for finding all $\myT$ such that $\langle \myA,\myT\rangle\le u$.  Let $\myA$ and $\myB$ be two rational, symmetric $2\times2$ matrices with $\myA$ positive definite.  
We explain how to effectively compute specializations of the form $f(G\tau +H)$, as in equation~\ref{eqn:restriction1} 
or Proposition~\ref{speedup1}.  We adapt our index sets~$\SSS$ to the type used 
in~\eqref{eqn:BPinfiniteproduct} for Borcherds products but they can be used in all the cases we need to program.  
For any $u,\delta\in\R$, let
\begin{align*}
\SSS(N,\myA, u, \delta)=\Bigl\{ 
(n,r,m)\in\Z^3: 
&\tr\Bigl(\begin{psmallmatrix} 
n  &  r/2    \\
r/2  &  mN  
\end{psmallmatrix}
\myA\Bigr)\le u, 
m\ge0,
4mnN-r^2\ge\delta,
\\
&
\text{if $m=0$ then $n\ge0$ and if $m=n=0$ then $r<0$}\Bigr\}.
\end{align*}

\begin{prop}
Let $\myA=\begin{psmallmatrix} 
\alpha  &  \beta    \\
\beta  &  \gamma  
\end{psmallmatrix} \in \M_2(\R)$ be positive definite.  Let $u,\delta \in \R$.  
Let $\Delta=\det \myA=\alpha\gamma-\beta^2 > 0$.  Let $X =   4\alpha umN  -\alpha^2\delta  -4\Delta(mN)^2$.
Then the elements $(n,r,m) \in \SSS(N,\myA,u,\delta)$ satisfy the following bounds. 
\begin{enumalph} 
\item   
If $m\ge 1$, then
\begin{equation*}
\begin{aligned}
1 \le &\ m \le \frac{\alpha(u+\sqrt{u^2-\delta \Delta})}{2\Delta N},\\
\frac{-2\beta mN -\sqrt{X}}{\alpha} \le &\ r \le 
\frac{-2\beta mN +\sqrt{X}}{\alpha},  \quad \text{ and } \\
 \frac{r^2+\delta}{4mN}\le &\ n \le \frac{u - \beta r -\gamma mN}{\alpha}.
\end{aligned}
\end{equation*}
\item 
If $m=0$ and $n>0$, then
\begin{equation*}
\begin{aligned}
r^2 &\le -\delta \quad \text{and} \quad 
1\le n \le \frac{u - \beta r }{\alpha}.
\end{aligned}
\end{equation*}
\item 
If $m=n=0$, then
\begin{equation*}
r^2 \le -\delta \text{ and $r<0$}.
\end{equation*}
\end{enumalph}
\end{prop}

\begin{proof}
The main two conditions that need to be satisfied are
$\alpha n+\beta r+\gamma mN\le u$ and $4mnN-r^2\ge\delta$.
The case $m=0$ is straightforward, so we only deal with the case $m\ge1$ here.
These two inequalities lead immediately to the third inequality as stated in the proposition.  
From this third inequality, we work with terms on the left and right of $n$; multiply through by $4mN\alpha$ 
and put the terms on one side: 
\begin{equation*}
\alpha r^2+\alpha\delta
-4mNu+4mN\beta r+4\gamma(mN)^2\le 0.  
\end{equation*}
Solving this quadratic inequality for $r$ yields the second inequality stated in the proposition.
A condition for there to be a solution in $r$ is that the inside~$X$ of the square root must be nonnegative.
Solving the resulting quadratic inequality yields the first
inequality in the proposition.
\end{proof}

We conclude with a final speedup. 
Suppose we wish to calculate the coefficient of $q^e$ in $f(\myA\tau+\myB)$.
If
there are no 
$(n,r,m)\in\SSS(N,\myA,u,\delta)$
such that $\tr\Bigl(
\begin{psmallmatrix} 
n  &  r/2    \\
r/2  &  mN  
\end{psmallmatrix}
\myA\Bigr)=e$,
then we may skip the term involving $\myA$.
This simple observation is especially useful for terms in the second summand in \eqref{eqn:restriction1}: 
for well chosen $s$, there are typically at most $2$ choices of~$i$ 
for which such $(n,r,m)$ exist.  It often happens that,  for these surviving~$i$,  Proposition
\ref{speedup1}(d) applies.  

\subsection{Example of restricting $f_{277}$}
Now suppose that $f$ is represented as a rational function in Gritsenko lifts~$G_i$ with coefficients in a commutative ring $R$ by $f=Q(G_1, \ldots, G_r)$.
Both the slash by~$M$ and the specialization by $\phi_s^{*}$ 
may be applied directly to each Gritsenko lift, so that we obtain 
 
\begin{equation} 
\phi_s^*(f\,|\, M) = Q(\phi_s^*(G_1\,|\, M), \ldots, \phi_s^*(G_r\,|\, M)).
\end{equation}
If the Fourier coefficients of $f$ satisfy $a(T;f) \in R \subseteq \C$, then for the 
representative matrices $M_j$ appearing in the coset decomposition \eqref{eqn:heckeopersM} for the Hecke operator $T(p)$, the sum in \eqref{eqn:phisslash} can be taken over $n \in \frac{1}{p}\Z_{\geq 0}$ and the coefficients of $\phi_s^*(f \,|\, M_j)$ belong to the ring $R[1/p,\zeta_p]$ where $\zeta_p=e(1/p)$ is a primitive $p$th root of unity.
From specializing $f\,|\, T(p) = \sum_j f\,|\, M_j = a_p(f) f$, 
the eigenvalue $a_p(f)$ for $T(p)$ can be computed by performing field operations on \LP\ series in $q$ via
\begin{equation}
\label{computethis}
a_p(f)= 
\frac{1}{\phi_s^{*}\left(f\right) }
\sum_j \phi_s^{*}\left(f\,|\,M_j\right) \in R[1/p, \zeta_p][[q^{1/p}]]
\end{equation}
whenever the specializing curve $\phi_s$ is chosen so that $\phi_s^{*}\left(f\right)$ is not identically zero.  In practice, we choose a target exponent $e$ such that
$\Coeff_e \phi_s^{*}f \ne0$ and then
\begin{equation}
\label{computethis2}
a_p(f)= 
\frac{\Coeff_e\bigl(\sum_j \phi_s^{*}\left(f\,|\,M_j\right)\bigr)}{
\Coeff_e({\phi_s^{*}\left(f\right) })}.
\end{equation}

\begin{remark}
One practical advantage of this technique of restricting to modular curves is that 
when more than one coefficient in the $q$-expansion of \eqref{computethis} is computed, 
it constitutes a double check on the value of $a_p(f)$.
\end{remark}

\begin{example} \label{exm:f277exp}
We consider the core example of the form $f_{277}$ of level $N=277$ constructed above \eqref{eqn:fasQL277}.  A Fourier coefficient of $f_{277}$ whose matrix index has the smallest determinant is $a(t_0;f_{277})=-3$, where $t_0=\begin{pmatrix} 49 & -233/2 \\ -233/2 & 277 \end{pmatrix}$ and $\det(2t_0)=3$.  Accordingly we select 
$s=\begin{pmatrix} 554 & 233 \\ 233 & 98 \end{pmatrix}$, which is 
the adjoint of $2t_0$.  Working over $R=\Z$,   
we find
\begin{equation} 
\phi_s^*(f_{277}) = -3q^3 + 6q^6 + 6q^9 + 3q^{12} + 3q^{15} - 12q^{18} + 3q^{21} + O(q^{24}).
\end{equation}
As a sanity check, we recognized $\phi_s^*(f_{277})$ using modular symbols as a classical modular form of weight $4$ and level $3\cdot 277$ to order $O(q^{400})$.  We then compute
\begin{equation}
\phi_s^*(f_{277} \,|\, T_2) = 6q^3 - 12q^6 - 12q^9 - 6q^{12} - 6q^{15} + 24q^{18} - 6q^{21} + O(q^{24}) 
\end{equation}
so quite convincingly, $a_2(f_{277}) = -2$, in agreement with \eqref{eqn:firstfewap277}.
\end{example}

To compute the action of Hecke operators on the specialized expansion \eqref{computethis}, we work (to a finite degree of $q$-adic precision) with coefficients  over $\C$ or over $\Z/m\Z$ with $m$ suitably large---we consider these two approaches in turn in the next two sections.

\subsection{Over floating point complex numbers}

We may also compute $a_p(f)$ via equation~\eqref{computethis} over the complex numbers using interval arithmetic.  

\begin{example} \label{exm:f277compute}
We perform our Hecke computation with in-house C++ code.  Continuing with $f=f_{277}$ as in Example \ref{exm:f277exp}, for $p=2$ we work with 512 bits of precision: the upper size encountered was $3.40282\cdot 10^{38}$ and the lower size was $2.9387\cdot 10^{-39}$, giving
\[ a_2(f) =\frac{\phi_s^*(f \,|\, T_2)}{\phi_s^*(f)} \equiv \frac{6q^3 + O(q^5)}{-3q^3 + O(q^4)} = -2 + O(q) \]
up to an error $10^{-75}$ under a second on a standard desktop CPU.  The largest computation required for this $f$ was $a_{43}(f)  = 4$; with the same bit precision and maximum error smaller than $10^{-40}$,  it took less than 90 minutes.
\end{example}

\begin{rmk}
Given the first few Dirichlet coefficients of an $L$-function in the Selberg class with specified conductor and $\Gamma$-factors, Farmer--Koutsoliotas--Lemurell \cite{FKL} can (in principle) rigorously compute complex approximations to the next few Dirichlet coefficients using just the approximate functional equation.  This method is practical for small examples---and it is especially useful when the $L$-function is of unknown, speculative, or otherwise complicated origin.  Prolonging an initial $L$-series is a possible 
avenue for extending the range of examples of modularity proven in this article.  
\end{rmk}

\subsection{Expansion over a finite field} \label{sec:computeforms}

As an alternative to complex expansion, we may also work in a finite ring.  To do so, we need the following archimedean information about the Hecke eigenvalue.

\begin{prop} \label{eigenbound}
Let $f \in S_k(K(N))$ be an eigenform for the Hecke operators $T(p),T_1(p^2)$ with 
eigenvalues $a_p(f),a_{1,p^2}(f) \in \C$ where $p \nmid N$.  Then 
\begin{equation}
\begin{aligned}
\left| a_p(f) \right| &\le p^{k-3}(1+p)(1+p^2) ; \quad
\left| a_{1,p^2}(f) \right| \le p^{2k-6}(1+p)(1+p^2)p.
\end{aligned}
\end{equation}
\end{prop}

\begin{proof}
By an elementary estimate, 
there exists a $B>0$ such that $| a(T;f) | \le B \det(T)^{k/2}$ for all $T$.  
Clearly $B = \sup_{T>0} {| a(T;f)|}\,{\det(T)^{-k/2}}$ is optimal.  
By \eqref{eqn:heckeopersFC}, we have
\begin{align*} 
|a_p(f)| \,&|a(T;f)| =
\left|a\left( T; f | T(p) \right)\right| \\
\le  &\left|a\left( pT;f\right)\right| + p^{k-2} \sum_{j \ \text{\rm mod} \ p}| a( \textstyle{\frac1p} T \begin{bsmallmatrix} 1 & 0 \\ j & p \end{bsmallmatrix};f)|  
+ p^{k-2}  |a( \textstyle{\frac1p} T \begin{bsmallmatrix} p & 0 \\ 0 & 1 \end{bsmallmatrix};f)|
+ p^{2k-3} |a( \frac1p T;f )| \\
\le &B p^k \det(T)^{k/2} +Bp^{k-1} \det(T)^{k/2} +B p^{k-2} \det(T)^{k/2} +B p^{k-3} \det(T)^{k/2}.
\end{align*}
 From the equation 
$|a_p(f)|\,|a(T;f)| \det(T)^{-k/2} \le B\left( p^k+p^{k-1}+p^{k-2}+p^{k-3} \right)$, 
we obtain the desired result by taking the supremum over $T>0$.  

A similar argument shows the inequality for $a_{1,p^2}(f)$.
\end{proof}

If $a \in \Z$ and $\left| a \right| < C$, then we can recover $a \in \Z$ from its congruence class modulo $m$ whenever $m>2C$.  For our purposes, we might as well work with a \emph{prime} modulus $m$, and indeed, because of the needed $p$th roots of unity, we choose a large prime $m$ such that $m \equiv 1 \psmod{p}$ and work in $R=\Z[\zeta_p]/\mathfrak{m}$ where $\mathfrak{m}$ is a fixed choice of split prime 
above $m$, and we compute the expansion \eqref{computethis} in $R[[q]]$ as
\[ a_p(f) \equiv \frac{1}{\phi_s^{*}\left(f\right) }
\sum_j \phi_s^{*}\left(f\,|\,M_j\right) \pmod{\frakm} \]
and then lift the result to $\Z \subseteq\Z[\zeta_p]$.  
The computational benefit is that we may replace $\zeta_p$ by an integer and compute 
modulo~$m$. 

\begin{example}  \label{example:587}
Let $f_{587}^-\in S_2(K(587))^-$ be the Borcherds product defined in \eqref{eqn:f587def}.  
We choose $t_0=\begin{pmatrix} 4& -137/2\\-137/2&  1174\end{pmatrix}$
and have $a(t_0,f)=-1$.
We used $s=\begin{pmatrix} 2348& 137\\137&  8\end{pmatrix}$
 and target exponent $e=\tr(st_0)=15$.
We used the finite field method in our computations,
which required a choice of a prime modulus $m$ and
an integer $\gamma$ such that
$\gamma\not\equiv 1 \pmod{m}$ and
$\gamma^p\equiv 1 \pmod{m}$.
The modulus $m$ must be chosen large enough so that 
$m>\lfloor2C\rfloor$ where  $C=p^2(1+1/p)(1+1/p^2)$ from Proposition \ref{eigenbound}.
The code was written in C++ using FLINT for operations of polynomials in one variable modulo an integer, and the computation of the restriction method to compute $a_{41}(f_{587}^-)$ took  less than 2 hours on a typical CPU.  The computation of $a_{1,p^2}(f)$ for $p \leq 11$ took just a few minutes.
\end{example}

\section{Verifying paramodularity}  \label{sec:verifypara}

In this section, we carry out the Faltings--Serre method for our case of interest $\redG=\GSp_4$ and $\ell=2$, proving our main Theorem \ref{thm:mainthm277} as well as the other two advertised cases.  We employ the conventions and notation of section \ref{sec:galoisreps}, in particular for Galois representations and $L$-functions.

\subsection{The case $N=277$}

Let $X=X_{277}$ be the smooth projective curve over $\Q$ given by the equation
\begin{equation} \label{eqn:yx277}
X\colon y^2+(x^3+x^2+x+1)y=-x^2-x
\end{equation}
with LMFDB label \href{http://www.lmfdb.org/Genus2Curve/Q/277/a/277/1}{\textsf{277.a.277.1}}, or equivalently by
\begin{equation} \label{eqn:yx277-armand}
y^2+y=x^5-2x^3+2x^2-x. 
\end{equation}
Both models are minimal with discriminant $\Delta=277$.  Let $A=A_{277}=\Jac X_{277}$ be the Jacobian of $X_{277}$, a principally polarized abelian surface over $\Q$ of conductor $277$.  
Let $f=f_{277} \in S_2(K(277))$ be the Siegel modular form of weight $2$ constructed in \eqref{eqn:fasQL277}.  

Our main result (implying Theorem \ref{thm:mainthm277}) is as follows.

\begin{thm} \label{thm:proofofmodularity}
For all primes $p$, we have $L_p(A_{277},T)=Q_p(f_{277},T)$.  In particular, we have $L(A_{277},s) = L(f_{277},s,\spin)$ and the abelian surface $A_{277}$ is paramodular.
\end{thm}

To ease notation, we now dispense with subscripts.  To prove this theorem, we use the strategy described in section \ref{sec:appfaltserre}, with the further practical improvements from section \ref{sec:instages}.  Attached to $A$ by \eqref{eqn:rhoAell} and to $f$ by Theorem \ref{thm:galoisrep} and by the remarks afterward are $2$-adic Galois representations
\[ \rho_{A},\rho_{f} \colon \Gal_{\Q,S} \to \GSp_4(\Q_2^{\text{al}}) \]
where $S=\{2,277,\infty \}$ such that $\det \rho_A = \det \rho_f=\chi_2^2$ the square of the $2$-adic cyclotomic character.  Our first task is to verify equivalence of residual representations.  We start with Lemma \ref{lem:descentwhenirred}(a), which allows us to conclude that the residual representations $\overline{\rho}_A^{\textup{ss}},\overline{\rho}_f^{\textup{ss}} \colon \Gal_{\Q,S} \to \GSp_4(\F_2)$ take values in $\F_2$.  

\begin{lem} \label{lem:residuallyagree}
The residual representations $\overline{\rho}_A, \overline{\rho}_f \colon \Gal_{\Q,S} \to \GSp_4(\F_2)$ are equivalent and have absolutely irreducible image $S_5(b)$.
\end{lem}

\begin{proof}
We apply Algorithm \ref{alg:detresidual}.  The representation $\overline{\rho}_A$ is given by the action on $A[2]$; completing the square in \eqref{eqn:yx277-armand} to obtain the model $y^2=g(x)=4x^5-8x^3+8x^2-4x+1$ we obtain $\overline{\rho}_A$ via the action on the roots of $g(x)$, which we verify is isomorphic to $G=S_5(b)$ as the elements of order $3$ have trace $1$ by \eqref{table:subgroups}.  As implied by the general theory, the field $\Q(A[2])$ is ramified only at $2,277$. 

For $\overline{\rho}_f$, we only have indirect access to the Galois representation.  By \eqref{eqn:firstfewL277}, we have 
\[ \det(1-\overline{\rho}_f(\Frob_3)T)= 1+T+T^2+T^3+T^4 \in \F_2[T], \] 
so $\img \overline{\rho}_f$ contains an element of order $5$.  Similarly $\Frob_5$ has order divisible by $3$, so $\img \overline{\rho}_f$ is isomorphic to one of $A_5,S_5,A_6,S_6$.  Therefore the fixed field under $\ker \overline{\rho}_f$ is the splitting field of an irreducible, separable polynomial $g(x)$ of degree $5$ or $6$.  Let $F \colonequals \Q[x]/(g(x))$; then $F$ is unramified away from $2,277$.  But we know a bit more: by Lemma \ref{lem:fixedfieldordp}, the $277$-valuation of the Artin conductor of $\overline{\rho}_f$ is at most $1$, so $\ord_{277}(d_F) \leq 1$.  A Hunter search, or looking up the possible fields in the database of Jones--Roberts \cite{JonesRoberts}, shows that there are no such degree $6$ polynomials, and exactly two polynomials of degree $5$, namely $x^5 - x^4 + 2x^2 - x + 1$ and $x^5 - x^4 + 4x^3 + 5x - 1$.  Both polynomials have the same Galois closure, with Galois group $S_5$; we need to distinguish the representations afforded by the inclusion $S_5 \subseteq S_6$ and the fixed representation \eqref{eqn:S6GSp2}.   We refer to \eqref{table:subgroups}: for the second one $\Frob_3$ does not have order $5$, so we must have a match with the representation afforded by the first one.
\end{proof}

With Lemma \ref{lem:residuallyagree} in hand, we apply Lemma \ref{lem:descentwhenirred}(b) to conclude that our $2$-adic representations descend to $\rho_A,\rho_f \colon \Gal_{\Q,S} \to \GSp_4(\Z_2)$.  We now finish the proof of the theorem.

\begin{proof}[Proof of Theorem \textup{\ref{thm:proofofmodularity}}]
We apply Algorithm \ref{alg:detrep}.  Step 1 was done in Lemma \ref{lem:residuallyagree}, and the residual representations have a common image
\[ G \colonequals \img \overline{\rho} \leq \GSp_4(\F_2)=\Sp_4(\F_2) \]  
with $G \simeq S_5(b)$.  Let $K$ be the fixed field under $\ker \overline{\rho}$, so $\Gal(K\,|\,\Q) \simeq G$ under $\overline{\rho}$.

Using Theorem \ref{thm:Gcore}, we now find all obstructing extension
groups $E$, an exact core-free subgroup
$D \leq E$, and a list of conjugacy classes of
obstructing elements.  We refer to the field diagram \eqref{eqn:LL0S5b}.  The extension $K_0=K^H$ has
degree $10$, explicitly it is given by adjoining a root of the
polynomial
\[ x^{10}+3x^9+x^8-10x^7-17x^6-7x^5+11x^4+18x^3+13x^2+5x+1. \]
The possible obstructing extensions $\varphi\colon \Gal(L \,|\, \Q) \hookrightarrow E$ are obtained as the Galois closure of the quadratic extension $L_0 \supseteq K_0$, still unramified away from $S$ so they may be constructed using class field theory: we find there are $4095$ quadratic extensions $L_0 \supseteq K_0$ unramified away from $S$.  To write down polynomials (not necessarily small) that represent these fields takes about 5 minutes; as we developed the algorithm, we found it convenient to optimize these polynomials (using polredabs), which took about $6$ hours.  In the course of the algorithm we consider $24062$ obstructing pairs $(L,\varphi)$.  

For each such obstructing pair $(L,\varphi)$, we compute a small prime $p \neq 2,277$ such that the conjugacy class of $\Frob_p$ is obstructing, according to the stages of section \ref{sec:instages}.  Computing obstructing primes by their $L_0$-cycle type as in Step 4${}^{\prime}$, we obtain the list of primes $\{3,5,7,11,13,17,19,23,29,31,37,41,43,53\}$; going a bit further, considering obstructing primes by the pair of $L_0,K_0$-cycle type as in Step 4${}^{\prime\prime}$, we manage only to remove the prime $p=53$ from the list (but reduce the sizes of primes in many cases), so we refine the list of primes to those with $p \leq 43$.  The total running time for this step was about 90 minutes on a standard CPU.

There are $8$ pairs $(L,\varphi)$ that require $p=53$.  The field $L_0$ generated by a root of
\[ 
\begin{gathered} 
x^{20} + 121x^{18} + 7459x^{16} + 286418x^{14} + 7324711x^{12} + 126372663x^{10} + 1387797423x^8 + \\ 7013797890x^6 - 30031807329x^4 - 582846604659x^2 - 1630793025157
\end{gathered} 
\]
has Galois closure $L$ with $\Gal(L\,|\,\Q) \simeq E \leq \Liesp_4(\F_2) \rtimes G$ with $\#E = 2^{10} 5!$.  There are four outer automorphisms $\xi$, and with respect to one of these, we find that $\Frob_5$ is an obstructing conjugacy class based on the $L_0,K_0$-cycle type pair $6^3 1^2, 6^1 3^1 1^1$ but $\Frob_{53}$ is the first obstructing prime based \emph{only} on the $L_0$-cycle type $8^1 4^2 2^2$ (and this cycle type works for all four $\xi$).

We are now in Step 5 of the algorithm, and to conclude we will show that $\tr \rho_A(\Frob_p) = \tr \rho_f(\Frob_p)$ for all $p \leq 43$.  The former traces can be done by counting points, the latter traces were computed using the method in Example \ref{exm:f277compute}, and we check that they are equal, completing the proof.  (In fact, we went further than necessary and checked the equality of traces for all $p \leq 97$.)
\end{proof}

\subsection{The case $N=353$}

We now turn to a case with residual image $S_3 \wr C_2$.
Let $X=X_{353}$ be the genus $2$ curve with LMFDB label \textup{\href{http://www.lmfdb.org/Genus2Curve/Q/353/a/353/1}{\textsf{353.a.353.1}}} defined by
\[
X \colon y^2+(x^3+x+1)y = x^2
\]
and $A=A_{353}=\Jac X$, a typical abelian surface of conductor $353$.  Let $f=f_{353} \in S_2(K(353))$ be the paramodular form constructed in \eqref{eqn:f353Q}.

\begin{thm}
For all primes $p$, we have $L_p(A_{353},T)=Q_p(f_{353},T)$.  In particular, $L(A,s) = L(f_{353},s,\spin)$ and the abelian surface $A_{353}$ is paramodular.
\end{thm}

\begin{proof}
The proof is similar to that of Theorem \ref{thm:proofofmodularity}, but with some slightly different arguments.
To supplement the data \eqref{eqn:firstfewL353}, we compute $a_p(f),a_{1,p^2}(f)$ for $p \leq 11$, and counting points yields equality of the additional Euler factors
\begin{equation}
\begin{aligned}
L_5(A,T) = Q_5(f,T) &= 1-T+2T^2-5T^3+25T^4, \\
L_7(A,T) = Q_7(f,T) &= 1-6T^2+49T^4, \\
L_{11}(A,T) = Q_{11}(f,T) &= 1-2T+T^2-22T^3+121T^4.
\end{aligned}
\end{equation}

Our first task is to verify that the mod $2$ representations $\overline{\rho}_{A}$ and $\overline{\rho}_{f}$ are equivalent and absolutely irreducible.  For $A$, we find the $2$-torsion field generated by the splitting field of the polynomial $x^6 + 2x^4 + 2x^3 + 5x^2 + 2x + 1$ and Galois group $S_3 \wr C_2$.  

Let $K$ be the fixed field of $\ker \overline{\rho}_{f}$ and $G\colonequals \Gal(K\,|\,\Q)$.  Since $L_3(A,T) \equiv 1+T+T^3+T^4 \pmod{2}$ we see that $G$ has an element of order $3$ or $6$ with trace $0$.  Since $L_{11}(A,T) \equiv 1+T^2+T^4$, we see $G$ has an element of order $3$ or $6$ with trace $1$.  Squaring such elements preserves their trace, so $G$ contains elements of order $3$ with either trace.  Thus $G \leq S_6$  has an
element with cycle decomposition $3^1$ and one with cycle decomposition $3^2$.  Listing all subgroups of $S_6$ with this
property, we see that $G$ must be isomorphic to one of the permutation groups
$$C_3^2, C_3:S_3, C_3 \times S_3 \textup{ (twice)}, C_3:S_3\cdot C_2, S_3^2 \textup{ (twice)}, S_3\wr C_2, A_6, S_6. $$
The subgroups in this list that are intransitive are $C_3^2,C_3:S_3,C_3 \times S_3, S_3^2$.  
The groups $C_3^2,C_3 \times S_3$ have $C_3$ as a quotient, and by the Kronecker--Weber theorem there are no $C_3$-extensions unramified outside $2$ and $353$ since $353 \equiv 2 \pmod{3}$.  The groups $C_3:S_3$ and $S_3^2$ have as quotient $S_3$, but there is a unique
$S_3$ extension ramified only at $2$ and $353$ (verified by a class field calculation and the Jones--Roberts database \cite{JonesRoberts}) defined by $x^3-x^2-6x+14$, and we compute that there are no cyclic cubic extensions of this field unramified away from primes dividing $2,353$.  This leaves the transitive groups $C_3:S_3 \cdot C_2,S_3\wr C_2, A_6, S_6$ arising as the normal closure of a degree $6$ subfield $K'$.  If $G=C_3:S_3 \cdot C_2$, then as in the proof of Proposition~\ref{prop:discbound}, we have $\ord_{353} d_{K'}=0,1,3$ but if $\ord_{353} d_{K'}=3$ then $G$ contains an element with cycle structure $2^3$, a contradiction.  Combined with Proposition~\ref{prop:discbound} in the remaining cases, we have $\ord_{353} d_{K'} \leq 1$.  Again by consulting the Jones--Roberts database \cite{JonesRoberts}, 
we find exactly two candidates, the extensions defined by $x^6 - 2x^5 + 2x^4 - x^2 + 1$ and
$x^6 - 2x^5 - 3x^4 + 4x^3 + x^2 - 6x + 1$.  In the first extension,
$\Frob_3$ has order $6$ contradicting
$Q_3(f,T) \equiv 1+T^4 \psmod{2}$, so we have the latter, and $G$ is isomorphic to
$S_3\wr C_2$.  Finally, since the trace of $\overline{\rho}_f(\Frob_3)$ equals
that of $A$, we see that the two residual images are isomorphic and
absolutely irreducible (recall that there are two embeddings of
$ S_3\wr C_2$ into $\GSp_4(\F_2)$ up to inner automorphisms, and they
differ in the trace of order $3$ and $6$ elements). 

Next, using Theorem \ref{thm:Gcore} we compute the extension $K_0$ corresponding to the core-free subgroup $C_2^2$, defined by
\begin{equation} \label{eqn:353K0x18}
x^{18} - 10x^{14} + 3x^{12} + 25x^{10} - 5x^8 - 19x^6 + 5x^2 + 1.
\end{equation}
Using computational class field theory, we list all quadratic extensions $L_0 \supseteq K_0$ unramified away from primes above $2,353$.  We find that there are $65535$ such extensions.  For each extension, we find an obstructing element; after computing for just over $5$ hours on a standard CPU (about 0.2 seconds per field) we find the list of primes 
\begin{equation} 
\{ 3,5,7,11,13,19,23,29,31,37,41,43,53,97,137\}. 
\end{equation}
(The prime $p=181$ arose from 2 extensions $L_0$ and $4$ maps $\varphi$ each looking only at cycle types, but by identifying the precise conjugacy classes we find obstructing classes for $p=5,137$.)

To conclude, using the floating point algorithm we compute $\tr \rho_f(\Frob_p)$ for all primes $p \leq 109$ as well as the primes $p=137,139,251$ (for robustness) in 29 hours on a standard CPU, and we see they agree with the traces obtained from point counts on $X$, completing the proof.
\end{proof}

\begin{example}
We pause to consider an extreme example where the refinement in section \ref{sec:instages} provides a significant improvement.  Consider the extension defined by adjoining a square root of the element
\[ -430a^{16} + 302a^{14} + 3956a^{12} - 3904a^{10} - 6944a^8 + 5348a^6 + 3628a^4 -
    1454a^2 - 510 \]
    where $a$ is a root of \eqref{eqn:353K0x18}, the defining polynomial for $K_0$.

There are $4$ outer automorphisms giving rise to possible maps $\varphi$: but in fact, we will see below that only $2$ of these maps extend $\overline{\rho}$, which is to say the other $2$ do not preserve the residual representation.
If we only consider cycle types that obstruct all $4$ possible maps $\varphi$ as in Step $4{}^\prime$, we have the types $8^4 2^2,4^6 2^2 1^8,4^2 2^{10} 1^8$.  For one of these $4$ extensions, the smallest prime $p$ with this cycle type is $p=251$.  If we push further in this extension, and look at the $L_0$-cycle type and the order in $K_0$, we compute that $p=101$ works.  Going even further and using $L_0,K_0$-cycle type, we find that $p=11$ works!
\end{example}
    
\subsection{The case \texorpdfstring{$N=587$}{Neq587}}

We conclude with one final case.  Let $X=X_{587}$ be the genus $2$ curve with
LMFDB label
\textup{\href{http://www.lmfdb.org/Genus2Curve/Q/587/a/587/1}{\textsf{587.a.587.1}}}
defined by
\[
X \colon y^2 + (x^3+x+1)y = -x^2-x
\]
and $A=A_{587}=\Jac X_{587}$, a typical abelian surface of conductor $587$ and rank $1$.  Let $f=f_{587}^- \in S_2(K(587))$ be the paramodular form constructed using \eqref{eqn:phiXi587}.  

\begin{thm} \label{thm:587}
For all primes $p$ we have $L_p(A_{587},T)=Q_p(f_{587}^-,T)$, and $A_{587}$ is paramodular.
\end{thm}

\begin{proof}
  We first verify that the mod $2$ representations
  $\overline{\rho}_{A}$ and $\overline{\rho}_{f}$ are equivalent and
  absolutely irreducible.  For $A$, we find the $2$-torsion field
  generated by the splitting field of the polynomial $x^6-2x^5+2x^4-x^2+2x-1$ with
  Galois group $G=S_6$.  For $f$, we have 
  \[ Q_3(f,T)=1+4T+9T^2+12T^3+9T^4 \equiv 1+T^2+T^4 \pmod{2} \]
and
  \[ Q_{11}(f,T)=1+T-T^2+11T^3+121T^4 \equiv 1+T+T^2+T^3+T^4 \pmod{2} \]
  by Poor--Yuen \cite[Table 5]{PY07} and Example~\ref{example:587}.  In
  particular, the residual image has order divisible by $3$ and $5$. 

  The subgroups of $S_6$ (up to isomorphisms) of order divisible by
  $15$ are:
\[
 A_5,
  S_5, A_6, S_6. 
\]
In all cases, there exists a polynomial of degree $5$ or $6$
unramified outside $\{2, 587\}$ and we can choose them such that the
discriminant valuation is at most $1$ at $587$ by
Proposition~\ref{prop:discbound}. By \cite{JonesRoberts} there are
only two degree $5$ polynomials with field discriminant having
valuation $1$ at $587$,
namely: $ x^5 - x^3 - x - 2$ and $x^5 + 2x^3 - 8x^2 - 13x - 8$ and two
degree $6$ polynomials with field discriminant having valuation 
$1$ at $587$:
$x^6-2x^5+2x^4-x^2+2x-1$ and $x^6-2x^5+3x^4+4x^3-2x^2-4x+2$. For the
degree $5$ polynomials, the first field has $\Frob_{3}$ of order $4$
(then it would have even trace) while $\Frob_{11}$ has order $2$ in
the second field. Regarding the degree six ones, in the second
extension $\Frob_{11}$ has order $2$, but odd trace in $A$. We deduce
that the residual representation of $f_{587}^-$ correponds then to the
same extension as $A$, and since both representations have the same
trace at $\Frob_3$, we deduce that they are indeed equivalent and
absolutely irreducible.

\medskip By Theorem \ref{thm:Gcore} we are led to compute all
quadratic extensions of the degree $20$ extension
\begin{multline}
  x^{20} + x^{18} - 4x^{17} - 3x^{16} - 2x^{15} + 7x^{14} - 6x^{13} - 18x^{12} 
  - 8x^{11} + 8x^{10}+\\ + 8x^9 - 18x^8 + 6x^7 + 7x^6 + 2x^5 - 3x^4 + 4x^3 + x^2 + 1.  
\end{multline}
We find that there are $2^{19}-1=524287$ such extensions.  Writing down minimal polynomials (not necessarily small) that represent these fields takes about 10 minutes; for convenience, we also computed optimized representatives, which took many CPU weeks.  

Finding an obstructing element for each of them, we find the list of primes to verify:
\begin{equation}
\{  3, 5, 7, 11, 13, 17, 19, 23, 29, 37, 41 \}.
\end{equation}
The total CPU time to compute this list of primes was about 2.5 hours (about 0.2 seconds per field).  Finally, we computed the corresponding traces above and they match, completing the proof.
\end{proof}

\section*{Appendix: Reduction of \texorpdfstring{$G$}{G}-covariant bilinear forms, by J.-P. Serre}
      
 \subsection*{Introduction}
 
 This note is intended as a complement to \cite{Serre:altform} where reductions of $G$-invariant bilinear forms modulo primes were studied. Indeed, in most applications to $\ell$-adic
 representations the natural bilinear forms are not $G$-invariant; they are only covariant with respect to a character of the group $G$. The simplest example  of this is the $\Q_\ell$-Tate module $V_\ell$ of an abelian variety $A$ over a field  $F$ of characteristic $\neq \ell$: a polarization of $A$ defines a nondegenerate alternating form $B$ on $V_\ell$, which is covariant under the action of the absolute
 Galois group  $\Gamma_F = \Gal(F_s/F)$, namely:
\[ B(gx,gy) = \chi_\ell(g)B(x,y)\text{\ \ for every $g \in \Gamma_F, x,y \in V_\ell$,} \]
where $\chi_\ell$ is the $\ell$-cyclotomic character.
 
   We shall see that the results of \cite{Serre:altform} extend to the covariant case, with practically
   the same proofs.

\subsection*{1. The setting}
      
   \smallskip
   
 \ni  It is almost the same as that of \cite{Serre:altform}. Namely:
    
    $G$ is a group, 
    
    $K$ is a field with a discrete valuation, 
    
    $R$ is the ring of integers of $K$,
    
    $\pi$ is a uniformizer of $K$,
    
    $ k = R/\pi R$ is the residue field,
    
    $\e \colon G \to R^\times$  is a homomorphism,
       
    $V$ is a finite dimension $K$-vector space on which $G$ acts, in such a way  that there exists an $R$-lattice of $V$ which is $G$-stable (``bounded action''),
    
    $V_k$ is the $k$-vector space obtained by the semisimplification of the $k[G]$-module $L/\pi L$, where $L$ is a $G$-stable lattice of $V$; up to isomorphism, it is independent from 
    the choice of $L$,
        
    $B$ is a symmetric (resp.\ alternating) nondegenerate $K$-bilinear form on $V$,
    which is $\e$-covariant under the action of $G$, i.e. 
    \[ B(gx,gy)= \e(g)B(x,y) \text{\ \   for $g\in G, x,y \in V$.} \tag{1.1} \]
   
\subsection*{2. Statement of the theorems}

The main theorem is the analogue of Theorem A of \cite{Serre:altform}. Namely:

\begin{theorem1*}
There exists a nondegenerate symmetric \textup{(}resp.\ alternating\textup{)} $k$-bilinear form on $V_k$ such that
\[ b(gx,gy) = \e(g)b(x,y) \text{\ \    for $g\in G, x,y \in V_k$} \tag{1.2}. \]
\end{theorem1*}  

  As in \cite{Serre:altform}, the proof will use the following complement to a classical theorem of Brauer and Nesbitt:
  
\begin{theorem2*}
Let $E$ be a finite dimensional $k[G]$-module
  endowed with a nondegenerate symmetric \textup{(}resp.\ alternating\textup{)} $k$-bilinear form $b$
  having property \textup{(1.2)}.  Then, the semisimplification $E^{\rm ss}$ of $E$
  has a $k$-bilinear form with the same properties as $b$.
  \end{theorem2*}
  
\subsection*{3. Proof of theorem 2}
  
    Use  induction on $\dim E$. Assume $E \neq 0$ and choose a minimal nonzero $G$-submodule $S$ of $E$. Let $H \subset E$ be the orthogonal subspace of $S$ with respect to $b$. Since $S$ is minimal, there
      are two possibilities:
   
   \smallskip     
         a) $H \cap S = 0$, i.e. the restriction of $b$ to $S$ is nondegenerate.
         In that case, we have $E^{\rm ss} = S \oplus  H^{\rm ss}$ and we apply
         the induction hypothesis to $H$.
     
     \smallskip    
         b) $H \cap S = S$, i.e. $S$ is totally isotropic for $b$. We have
         
  \ni      $E^{\rm ss} = (S \oplus E/H) \oplus (H/S)^{\rm ss}.$
  
  \smallskip
         
         The induction hypothesis applies to $(H/S)^{\rm ss}$. As for the first factor $S \oplus E/H$, one defines a bilinear form $b_1(x,y)$ on it by the following rule: if $x,y$ both belong to $S$, or to $E/H$, then $b_1(x,y) = 0$; if $x\in S$ and $y \in E/H$, then
         $b_1(x,y) = b(x,y')$ where $y'$ is any representative of 
      $y$ in $E$; if $x \in E/H$ and $y\in S$, then $b_1(x,y) = b_1(y,x)$ in the symmetric case and $b_1(x,y)=-b_1(y,x)$ in the alternating case. It is clear that the form $b_1$ has the required properties.

\subsection*{4. Proof of Theorem 1}
    
    The first step (\cite[Theorem 5.2.1]{Serre:altform}) is to show the existence of a lattice $L$ in $V$, which is 
    $G$-stable, and almost self-dual, i.e. $\pi L' \subset L \subset L'$, where $L'$
    is the dual of $L$ (note that formula (1.1) implies that the dual of a $G$-stable lattice is $G$-stable). This is done by choosing a $G$-stable lattice $M$, and defining
    $L$ as the ``lower middle''  $m_-(M,M') $ of $M$ and its dual $M'$ :
    
    $m_-(M,M') = $       
      smallest lattice containing $\pi^nM \cap \pi^{-n}M'$  for every $n \in \Z$.
    
   \ni It is proved in \cite[Theorem 3.1.1]{Serre:altform} that $m_-(M,M')$ is an almost self-dual lattice.    
    \smallskip
     
      The second step is to define a bilinear form $b$ on the $k$-vector space
      $E = L/\pi L' \oplus L'/L$ by using the reduction mod $\pi$ of $B$ on $L/\pi L'$, and of 
      $\pi B$ on $L'/L$. It is clear that $b$ is nondegenerate, $\e$-covariant, and symmetric 
      (resp.\ alternating) if $B$ is. By Theorem 2, the semisimplification $E^{\rm ss}$ of $E$
      has a bilinear form with the required properties. Since $E^{\rm ss}$ is isomorphic
  to $V_k$, this proves Theorem 1.

\end{document}